\documentclass[a4paper,french,twoside,11pt]{smfartVS}

\usepackage{smfenum,smfthm,amsmath,amssymb,mathrsfs,euscript,color}
\usepackage{stmaryrd,mathabx}
\usepackage[latin1]{inputenc}
\input{xypic}

\numberwithin{equation}{section} 
\theoremstyle{plain}

\def\CC{\mathbb{C}}
\def\FF{\mathbb{F}}

\def\QQ{\mathbb{Q}}

\def\ZZ{\mathbb{Z}} 

\def\A{{\rm A}}
\def\B{{\rm B}}

\def\D{{\rm D}}
\def\E{{\rm E}}
\def\F{{\rm F}}
\def\G{{\rm G}}
\def\H{{\rm H}}
\def\I{{\rm I}}
\def\J{{\rm J}}
\def\K{{\rm K}}
\def\L{{\rm L}}
\def\M{{\rm M}}
\def\N{{\rm N}}
\def\O{{\rm O}}
\def\P{{\rm P}}
\def\Q{{\rm Q}}
\def\R{{\rm R}}
\def\SS{{\rm S}}
\def\T{{\rm T}}
\def\U{{\rm U}}
\def\V{{\rm V}}
\def\W{{\rm W}}
\def\X{{\rm X}}
\def\Y{{\rm Y}}
\def\Z{{\rm Z}}

\def\Cc{\mathscr{C}}
\def\Dd{\mathscr{D}}
\def\Ee{\mathscr{E}}

\def\Gg{\mathscr{G}}
\def\Hh{\mathscr{H}}

\def\Mm{\mathscr{M}}

\def\Oo{\mathscr{O}}

\def\Rr{\mathscr{R}}

\def\Ww{\mathscr{W}}
\def\Xx{\mathscr{X}}

\def\AA{\mathfrak{A}}

\def\KK{\mathfrak{K}}
\def\PP{\mathfrak{P}}

\def\a{\alpha} 
\def\b{\beta}
\def\d{\delta}
\def\e{\varepsilon}
\def\g{\gamma}
\def\h{\varphi}
\def\k{\kappa}
\def\l{\lambda}
\def\n{\eta}

\def\p{\mathfrak{p}}
\def\q{\star}
\def\s{\sigma}
\def\t{\theta}
\def\v{\upsilon}
\def\w{\varpi}

\def\Ga{\Gamma}
\def\La{\Lambda}
\def\Om{\Omega}

\def\ie{c'est-à-dire }

\def\rp{\rangle}
\def\>{\geqslant}
\def\<{\leqslant}

\def\Hom{{\rm Hom}}
\def\End{{\rm End}}

\def\Mat{\mathscr{M}}
\def\GL{{\rm GL}}
\def\Gal{{\rm Gal}}
\def\Ker{{\rm Ker}}

\def\tr{{\rm tr}}

\def\ind{{\rm ind}}
\def\cind{{\rm ind}}
\def\mult#1{{#1}^{\times}}
\def\fr#1{\smash{\mathop{\ \to\ }\limits^{#1}}}

\def\iso#1{\smash{\mathop{\longrightarrow}\limits^{#1}}}

\def\Cc{\EuScript{C}}
\def\Hh{\EuScript{H}}
\def\Mm{\textbf{\textsf{M}}}
\def\Oo{\EuScript{O}}

\def\Ww{\EuScript{W}}

\def\Xx{\X}

\def\GB{{\overline{\G}}}
\def\MB{{\overline{\M}}}
\def\PB{{\overline{\P}}}
\def\UB{{\overline{\U}}}

\def\NB{{\overline{\N}}}

\def\IA{\ip}

\def\aa{\mathfrak{a}}
\def\bb{\mathfrak{b}}
\def\jj{\mathfrak{j}}
\def\ff{\kk}
\def\hh{\mathfrak{h}}
\def\kk{\mathfrak{k}}
\def\ll{\mathfrak{l}}
\def\vv{\mathfrak{v}}
\def\mm{\mathfrak{m}}
\def\nn{\mathfrak{n}}
\def\ss{\mathfrak{s}}
\def\tt{\mathfrak{t}}
\def\xx{\mathfrak{x}}

\def\sc{{\rm sc}}
\def\ip{\boldsymbol{i}}
\def\jp{\boldsymbol{j}}
\def\te{j}
\def\rp{\boldsymbol{r}}

\def\qlb{\overline{\QQ}_{\ell}}
\def\zlb{\overline{\ZZ}_{\ell}}
\def\flb{\overline{\FF}_{\ell}}

\def\NJ{\K}
\def\nl{\Sigma}

\def\cusp{{\rm cusp}}

\def\Irr{{\rm Irr}}

\def\Iw{\I}
\def\r{{\bf r}}

\def\m{\mm}

\def\({\left(}
\def\){\right)}

\def\tmax{\t_{{\rm max}}}
\def\kmax{\k_{{\rm max}}}
\def\Tmax{\boldsymbol{\Theta}_{{\rm max}}}
\def\KM{\textbf{\textsf{K}}}

\def\qr{q(\rho)}
\def\om{\omega}
\def\Fr{\phi}

\def\sy#1{[#1]}

\def\BJ{{\bf J}}
\def\bl{{\boldsymbol{\l}}}
\def\bk{\boldsymbol{\k}}
\def\bs{\boldsymbol{\s}}
\def\bn{\boldsymbol{\n}}
\def\bt{\boldsymbol{\tau}}

\def\vr{\varrho}
\def\vk{\bk}

\def\ve{\bn}

\def\el{o_\ell}
\def\ee{o}

\newcounter{nonum}

\newcounter{introlet}

\newtheorem{theon}[introlet]{Théorème}

\newcounter{intronum}
\def\theintronum{\arabic{intronum}}
\newenvironment{introsec}{\refstepcounter{intronum}
\noindent{\bf\theintronum.}}{\medskip}

\newcounter{notanum}
\def\thenotanum{\arabic{notanum}}
\newenvironment{notasec}{\refstepcounter{notanum}
\noindent{\bf\thenotanum.}}{\medskip}

\newcounter{prelinum}
\def\theprelinum{\arabic{prelinum}}
\newenvironment{prelisec}{\refstepcounter{prelinum}
\noindent{\bf\thesection.\bf\theprelinum.}}{\medskip}

\author{Alberto M\'\i nguez}
\address{Institut de Mathématiques de Jussieu, Université Paris 6, 
4 place Jussieu, 75005, Paris, France. 
URL: {\rm http://www.math.jussieu.fr/$\sim$minguez/}} 
\email{minguez@math.jussieu.fr}

\author{Vincent Sécherre}
\address{Université de Versailles Saint-Quentin-en-Yvelines\\
Laboratoire de Mathémati\-ques de Versailles\\
45 avenue des Etats-Unis\\
78035 Versailles cedex, France}
\email{vincent.secherre@math.uvsq.fr}

\title[Types modulo $\ell$ pour $\GL_{m}(\D)$]
{Types modulo $\ell$ pour les formes intérieures de $\GL_n$ sur un corps local
non archimédien}

\begin{altabstract}
Let $\F$ be a non-Archimedean locally compact field of residue 
characteristic $p$, let $\D$ be a finite dimensional central 
division $\F$-algebra and let $\ell$ be a prime number different from $p$.  
We develop a theory of $\ell$-modular types for 
the group $\GL_m(\D)$, $m\>1$, in preparation of the study of the 
$\ell$-modular smooth representations of this group. 
\end{altabstract}

\thanks{
Ce travail a bénéficié de financements de l'EPSRC 
(GR/T21714/01, EP/G001480/1) 
et de l'Agence Natio\-nale de la Recherche 
(ANR-08-BLAN-0259-01, ANR-10-BLANC-0114).
Le premier auteur est aussi financé en partie 
par MTM2010-19298 et FEDER}

\begin{document}

\maketitle

\setcounter{tocdepth}{2}

\section*{Introduction}

\begin{introsec}
Soit $\F$ un corps commutatif localement compact non archimédien de 
ca\-rac\-téristique ré\-si\-duelle $p$ et soit $\D$ une algèbre à division 
centrale de dimension finie sur $\F$ dont le degré ré\-duit est noté $d$.  
Étant donné $m\>1$, on pose $\G=\GL_{m}(\D)$, qui est une forme 
intérieure de $\GL_{md}(\F)$. 
La théorie des types complexes pour $\G$ a été développée dans une série 
d'articles \cite{VS1,VS2,VS3,SeSt1,BSS1,SeSt2} 
à la suite des travaux de Bushnell et Kutzko \cite{BK,BK2} pour $\GL_n(\F)$, 
$n\>1$. 
C'est un outil puis\-sant, qui per\-met une description explicite 
de la catégorie des re\-pré\-sentations lisses complexes de $\G$.  
Dans cet article, nous développons une théorie des types modulaires pour $\G$
dans l'objectif d'é\-tu\-dier les re\-pré\-sentations lisses modulaires de $\G$, 
\ie à coefficients dans un corps $\R$ algébriquement clos de 
carac\-té\-ristique $\ell$ différente de $0$ et de $p$
(voir \cite{MS2,MSu,MS12}). 
\end{introsec}

\begin{introsec}
La théo\-rie des représentations mo\-du\-lai\-res des groupes 
réductifs $p$-adiques a été dé\-ve\-lop\-pée par Vignéras 
\cite{Vig1,Vig2}.
Comparée à la théorie complexe,
elle présente de gran\-des simi\-larités 
mais aussi des différences importantes, à la fois dans les résultats 
et dans les méthodes. 
Les représen\-tations modulaires d'un sous-groupe ouvert compact 
ne sont pas semi-simples en général. 
Le fait que $\ell$ soit différent de $p$ équivaut à 
l'exis\-tence d'une mesure de Haar à valeurs 
dans $\R$ sur le groupe, 
mais la mesure d'un sous-groupe ouvert compact peut être nulle.  
Il faut distin\-guer 
entre les deux notions de représentation irréductible cuspidale
(\ie dont tous les modules de Jacquet relativement à un sous-groupe 
parabolique propre sont nuls) et super\-cus\-pi\-da\-le (\ie qui 
n'est sous-quotient d'aucune induite pa\-ra\-bo\-li\-que d'une 
re\-pré\-sen\-ta\-tion irréductible d'un sous-groupe de Levi propre). 
On a une notion de support super\-cus\-pi\-dal pour une 
repré\-sen\-tation irréductible, 
mais on ignore en général s'il est unique.
Il n'existe pas de version modulaire de la for\-mule des traces ni du théorème de
Paley-Wie\-ner. 
Il y a des repré\-sen\-tations irréductibles non 
iso\-mor\-phes d'un même groupe dont tous les modules de Jacquet propres 
sont isomorphes. 
\end{introsec}

\begin{introsec}
L'un des principaux outils dont on dispose pour étudier les représentations 
mo\-du\-laires d'un groupe réductif $p$-adique est la théorie 
des types, pour les groupes pour lesquels une telle théorie existe. 
Une théorie des types modu\-laires a été développée pour le groupe 
$\GL_n(\F)$ (voir \cite{Vig1,Vig2}) 
et l'objet du présent article est de produire une théorie analogue pour ses 
for\-mes inté\-rieu\-res. 
On l'utilise dans \cite{MS12} pour construire une 
clas\-si\-fi\-ca\-tion à la Zele\-vinski des représentations modu\-lai\-res 
irréductibles de $\G$ en termes de multi\-segments,
et dans \cite{MS2} pour clas\-ser les 
représenta\-tions ba\-nales de $\G$. 
V.~Sécherre et S.~Stevens l'utilisent dans \cite{SeSt3}
pour obtenir une dé\-com\-position en blocs
de la catégorie des représentations lisses mo\-du\-lai\-res de $\G$. 
\end{introsec}

\begin{introsec}
\label{PI5}
Notre première tâche est d'étendre au cas modulaire la construction des 
types semi-sim\-ples de $\GL_n(\F)$ et de ses formes intérieures.
Pour cela, nous reprenons les arguments de la théorie com\-plexe 
\cite{VS2,VS3,SeSt1,SeSt2} 
en expliquant comment les adapter au cas modu\-laire.

La principale nouveauté tient à ce que, dans le cas modulaire, 
l'algèbre des endomorphismes d'une 
représentation lisse d'un groupe compact peut être triviale sans que la 
représentation soit irréductible.
Ceci in\-ter\-vient dans le problème du transfert des $\b$-extensions 
(voir le lemme \ref{indirredkappa} et la proposition \ref{TSF}) et à propos de 
l'irréductibilité de certai\-nes représentations 
(voir le lemme \ref{MlleDeLaMole} et le corollaire \ref{Mathilde}). 
Nous renvoyons aussi au calcul de l'algèbre de Hecke d'un type simple~: 
comme dans le cas complexe, c'est une algèbre de Hecke 
affine de type \textsf{A}, mais la preuve de \cite{VS3} né\-ces\-site quelques 
ajustements pour être valable dans le cas modulaire (voir le lemme \ref{pinv}). 

Une nouveauté d'un autre ordre concerne la construction des types semi-simples 
non homogè\-nes (paragraphe \ref{NonEndoEq}).  
L'argument utilisé dans le cas complexe par Bushnell et Kutzko (\cite{BK2}), 
que nous reprenons, 
repose sur le fait que tout élément inversible à gauche de l'algèbre de
Hecke d'un type semi-simple est inversible. 
La preuve dans le cas complexe (\cite[7.15]{BK1}) doit être légère\-ment 
modifiée. 
D'autre part, nous déterminons la structure de cette algèbre non pas 
\textit{a prio\-ri} à l'aide de \cite[Theorem 12.1]{BK1} comme dans le cas complexe, mais 
\textit{a posteriori} grâce à une majoration de l'en\-sem\-ble d'entrelacement 
d'un type semi-simple effec\-tuée dans l'ap\-pen\-dice par V.~Sécherre et
S.~Stevens 
(voir le corollaire \ref{cor:xiKint} et le lemme \ref{NicolasRostov}). 

Le lecteur trouvera des détails sup\-plé\-mentaires 
dans l'intro\-duction à la section \ref{TSmodl}. 

L'apparition de ces nouveautés n'est pas liée au fait que le groupe $\G$ n'est 
pas déployé.
Celles-ci sont déjà présentes dans le cas du groupe déployé $\GL_n(\F)$, $n\>1$~;
toutefois elles n'apparaissent pas clairement dans \cite{Vig1,Vig2}.
\end{introsec}

\begin{introsec}
\label{intro5}
Nous obtenons plusieurs résultats sur les types semi-simples, 
qui sont nouveaux même dans le cas com\-plexe.  
Étant donné un type semi-simple $(\BJ,\bl)$ de $\G$, nous définissons une 
décomposition~:
\begin{equation*}
\bl=\bk\otimes\bs
\end{equation*}
analogue à la décomposition bien connue pour les types simples, 
$\bk$ et $\bs$ étant des repré\-senta\-tions irré\-ductibles de $\BJ$ et $\bs$ 
étant triviale sur le radical pro-unipotent $\BJ^1$ de $\BJ$. 
La restriction de $\bk$ à $\BJ^1$ est une 
représentation ir\-ré\-ductible $\bn$ dont l'entrelacement est calculé 
(proposition \ref{propbn2}) 
grâce à la majoration de l'appendice. 
Ce calcul est essentiel dans la preuve du théorème \ref{THB} 
ci-dessous. 
Voir aussi la pro\-po\-sition \ref{jnpc}, qui fournit une 
pro\-priété de paire couvrante de $(\BJ^1,\bn)$ 
bien utile dans la section \ref{NaisKM}
(dans le cas homogène). 
\end{introsec}

\begin{introsec}
La section \ref{RCIC} est consacrée à la construction des représentations 
irréductibles cuspidales de $\G$ en termes de certains types semi-simples, 
appelés types simples maximaux, de $\G$.
Comme dans le cas complexe, les types simples maximaux de $\G$ forment 
une famille de paires composées d'un sous-groupe ouvert compact 
$\J\subseteq\G$ et d'une représentation lisse irréductible $\l$ de $\J$, 
possédant les propriétés suivantes (voir la pro\-po\-sition \ref{MartinDuGard} et 
le théorème \ref{TheZerPos})~: 
\begin{enumerate}
\item
pour toute représentation irréductible cuspidale $\rho$ de $\G$, il existe 
un type simple maxi\-mal $(\J,\l)$, unique à $\G$-conjugaison près, 
telle que la restriction de 
$\rho$ à $\J$ admette une sous-re\-pré\-sen\-tation isomorphe à $\l$~;
\item
si $\rho$ est une représentation irréductible cuspidale de $\G$,
et si $(\J,\l)$ est un type sim\-ple ma\-xi\-mal contenu dans $\rho$, 
il y a une unique représentation du $\G$-normalisateur de $(\J,\l)$ 
qui pro\-longe $\l$ et dont l'induite compacte à $\G$ soit isomorphe à 
$\rho$~;
\item
deux représentations irréductibles cuspidales de $\G$ contiennent un
même type simple maxi\-mal $(\J,\l)$ si et seulement si elles sont 
inertiellement équivalentes. 
\end{enumerate}
Comme dans le cas complexe, la preuve de ce résultat suit le schéma maintenant 
classique élabo\-ré par Bushnell et Kutzko et fondé sur la méthode des 
paires couvrantes \cite{BK1,Vig2}~: il s'agit de prou\-ver que toute représentation 
irréductible cuspidale de $\G$ contient une strate fondamentale non scindée, 
puis un caractère simple, puis un type semi-simple, en\-fin de prouver 
que ce\-lui-ci est un type simple maximal. 
Seule la preuve de la proposition \ref{MartinDuGard} nécessite quelques 
adaptations à cause du problème de l'irréductibilité déjà rencontré au 
paragraphe \ref{PI5}. 
Signalons qu'à ce stade, il n'existe aucune différence de traitement entre les 
représentations cuspidales et les représentations supercuspidales. 
\end{introsec}

\begin{introsec}
Cette description des représentations irréductibles cuspidales de $\G$ en 
termes de types simples maximaux permet, étant donné un nombre premier 
$\ell$ différent de $p$, d'étudier la réduction 
modu\-lo $\ell$ des $\qlb$-représentations irréductibles cuspidales entières. 
Ceci donne le premier résul\-tat dif\-férant du cas déployé~: 
voir le théo\-rè\-mes \ref{RedCusp}, 
que l'on com\-parera au théorème III.1.1 de \cite{Vig1}. 

\begin{theon}
\label{THA}
Étant donnée une $\qlb$-repré\-sen\-ta\-tion irréductible cuspidale 
${\rho}$ de $\G$, 
il y a un entier $a\>1$ et une $\flb$-repré\-sen\-ta\-tion irréductible 
cuspidale ${\bar\rho}$ de $\G$ telle que la réduction modulo $\ell$ de 
${\rho}$ est~:
\begin{equation*}
{\bar\rho}\oplus{\bar\rho}\bar\nu\oplus\cdots\oplus{\bar\rho}\bar\nu^{a-1}
\end{equation*}
où $\bar\nu$ est le $\flb$-caractère non ramifié défini comme la 
valeur absolue de la norme réduite de $\G$.
\end{theon}

Il y a ainsi des $\qlb$-représentations irréductibles cuspidales entières 
dont la réduction modulo $\ell$ n'est pas irréductible, 
et des $\flb$-représentations irréductibles cuspidales n'admettant 
pas de re\-lè\-ve\-ment à $\qlb$ (voir les exemples \ref{DuPlessis} et 
\ref{ExBadRel}).  
On renvoie aussi au théorème \ref{RelSuperCusp} 
(et à la remarque \ref{RemaRelSuperCusp}), 
premier résultat faisant une différence entre 
représentations cuspidales et représentations supercuspidales, 
que l'on com\-parera au théorème III.5.10 de \cite{Vig1}.
\end{introsec}

\begin{introsec}
La théorie abstraite des types de Bush\-nell et Kutzko (\cite{BK1}) -- 
par opposition à la théorie des types semi-simples qui en est une réalisation 
concrète pour certains groupes -- 
est un programme dont l'objectif est de 
décrire la décomposition de Bernstein de la catégorie des représentations 
lis\-ses complexes d'un groupe réductif $p$-adique en fonction de 
catégories de modules sur les algè\-bres d'entrelacement -- ou de Hecke -- 
des types impliqués. 
La notion de type est caractérisée par 
plusieurs propriétés importantes (voir le paragraphe d'intro\-duc\-tion à la 
section \ref{CostaCafe}).
Dans le cas modulaire, il n'existe pas à l'heure actuelle de théorème général de 
décomposition en blocs en termes de classes d'inertie de paires 
(super)cuspidales pour la catégorie des représen\-ta\-tions 
lisses modulaires d'un groupe réductif $p$-adique,
non plus qu'une notion claire de ce que devrait être un type pour un tel
groupe.
Notre objectif ici n'est pas de développer une théo\-rie abstraite des types 
modulaires, mais plutôt d'étudier comment la théorie des types semi-sim\-ples 
-- tels qu'ils sont définis dans la section \ref{TSmodl} -- 
permet de comparer la théorie des 
représentations lis\-ses modulaires de $\G$ à celle de certaines algèbres de 
Hecke affines. 
\end{introsec}

\begin{introsec}
Contrairement au cas complexe (\cite{BK2,SeSt2}), 
cette comparaison n'est pas parfaite 
(on n'a pas en gé\-né\-ral d'équivalence de ca\-té\-go\-ries 
décrivant les blocs de la catégorie des représentations lis\-ses de $\G$) 
mais permet toutefois d'étudier les représentations irréductibles de 
$\G$ grâce à la propriété de presque-projectivité introduite par Dipper 
\cite{DipHom1} et développée par Vignéras et Arabia \cite{Vig2,Vig3}
(voir le paragraphe \ref{RQP}).
Le point de départ est le résultat suivant, qui est l'un des principaux 
résultats de cet article (voir la proposition \ref{pptf1}). 

\begin{theon}
\label{THB}
Soit $(\BJ,\bl)$ un type semi-simple de $\G$.
Alors l'induite compacte $\cind_\BJ^\G(\bl)$ est quasi-projective.
\end{theon}

L'un des principaux ingrédients de la preuve de ce résultat est 
le calcul de l'entrelacement de la représentation $\bn$ introduite 
au paragraphe \ref{intro5}. 

Grâce à cette propriété, 
nous avons 
une bijection naturelle entre les modules à droite simples sur l'algèbre 
de Hecke de $(\BJ,\bl)$ et les clas\-ses d'isomorphisme de représentations 
irréductibles de $\G$ dont la restriction à $\BJ$ contient $\bl$ 
(voir le théorème \ref{qptf}).
En outre, la proposition \ref{BoyerDArgens} montre que ces représentations-ci 
sont caractérisées par la clas\-se inertielle de leur support cuspidal. 
\end{introsec}

\begin{introsec}
La suite de la section \ref{CostaCafe} est consacrée au problème important de 
la comparaison entre in\-dui\-tes paraboliques et modules induits. 
On donne notamment des conditions suffisantes pour qu'une in\-duite 
parabolique soit irréduc\-ti\-ble, 
dont un corol\-lai\-re est le théorème \ref{ss} permettant de ra\-me\-ner le 
problème de la classification de toutes les représentations irréductibles de 
$\G$
à celui des représentations ir\-ré\-ductibles dont le support cuspidal est 
simple, \ie inertiellement équi\-valent à un support cuspidal de la forme 
$\rho+\dots+\rho$
où $\rho$ est une représentation irréduc\-ti\-ble cuspidale fixée.
Plusieurs résultats de cette partie sont inspirés de \cite{Vig2} (notamment du 
para\-gra\-phe V.2) même si nos preuves sont en général différentes.
Mentionnons enfin le théorème de compa\-raison \ref{RussSamProp} 
(dont on ne trouve pas d'analogue dans \cite{Vig2}), 
qui permet d'établir des cri\-tè\-res d'ir\-ré\-ductibilité d'induites par 
changement de groupe (paragraphe \ref{ChgtGp})
et d'associer à toute représentation irréductible 
cuspi\-dale $\rho$ de $\G$ 
un caractère non ramifié $\nu_{\rho}$ de ce groupe possédant la 
propriété suivante (voir la propo\-sition \ref{cuspidal}). 

\begin{theon}
Si $\rho'$ est une représentation irréductible cuspidale de $\GL_{m'}(\D)$, 
$m'\>1$, alors l'induite normalisée de $\rho\otimes\rho'$ est réductible 
si et seulement si $m'=m$ et $\rho'$ est isomorphe à $\rho\nu_{\rho}^{}$ 
ou à $\rho\nu_{\rho}^{-1}$. 
\end{theon}

Comme déjà dans le cas complexe, et contrairement au cas où $\G$ est déployé, 
ce caractère $\nu_\rho$ dépend de $\rho$.
\end{introsec}

\begin{introsec}
\label{introKM}
Si l'on essaie d'étendre à $\G$ les techniques employées dans \cite{Vig1,Vig2} 
pour le groupe déployé $\GL_n(\F)$, on est confronté au fait que les 
représentations irréductibles cuspidales de $\G$ n'ont pas de 
mo\-dè\-le de Whitta\-ker et qu'il n'y a pas de théorie des dérivées pour les 
représentations de $\G$, dont l'usage est crucial dans \cite{Vig2} pour la 
classification des représentations irréductibles. 
C'est la raison pour laquelle on introduit dans la section \ref{NaisKM}
un outil technique important, 
permettant de faire un lien entre représenta\-tions de
$\G$ et re\-pré\-sentations de grou\-pes linéaires généraux sur 
une extension finie du corps résiduel de $\F$. 
Un tel outil a déjà été utilisé dans \cite{Vig1} pour étudier les représentations 
irréductibles cuspidales modulaires de $\GL_n(\F)$, et de façon plus systématique 
dans \cite{SZ} pour définir une stratification de la catégorie des 
représentations complexes de $\GL_n(\F)$ affinant la décomposition de 
Bernstein.
\end{introsec}

\begin{introsec}
Le point de départ est un processus 
associant à toute représentation irréductible cuspidale de $\G$ 
un objet appelé son endo-classe. 
Il est décrit dans \cite{BH1,BSS1} pour les représentations complexes 
et fonc\-tionne de façon simi\-laire pour les représentations modulaires.
On en trouve dans \cite{BHLTL4} une inter\-prétation arith\-mé\-ti\-que 
dans le cas complexe pour $\GL_n(\F)$.
Ensuite, étant donnée une repré\-sentation irré\-duc\-tible cuspidale $\rho$ de $\G$, 
on peut lui attacher 
un entier $r\>1$, une extension finie $\kk$ du corps résiduel de $\F$ 
et un foncteur $\KM$ de la catégorie des représentations lisses 
de $\G$ dans la catégorie des représentations 
du groupe fini $\GL_{r}(\kk)$ possédant les propriétés suivantes~:
\begin{itemize}
\item 
le foncteur $\KM$ est exact~;
\item
il envoie représentations admissibles sur représentations de 
dimension finie et repré\-sen\-ta\-tions cuspidales sur représentations 
cuspidales (ou nulles)~; 
\item
il annule les représentations irréductibles de $\G$ -- 
\textit{et uniquement celles-là} -- 
dont le support cuspidal contient une représentation cuspidale 
dont l'endo-classe diffère de celle de $\rho$. 
\end{itemize}
Par exemple, si $\rho$ est de niveau zéro, alors 
$\kk$ est le corps résiduel de $\D$ et 
$\KM$ est le foncteur associant à toute représentation lisse de 
$\G$ la représentation de $\GL_{m}(\kk)$ sur l'espace de ses invariants 
sous le radical pro-unipotent du sous-groupe 
compact maximal $\GL_{m}(\Oo)$, où $\Oo$ est l'anneau des entiers 
de $\D$.
\end{introsec}

\begin{introsec}
La propriété d'annulation du foncteur $\KM$ (troisième point ci-dessus), 
qui jouera un rôle essen\-tiel dans des travaux ultérieurs (voir par exemple 
\cite{MS12}), constitue un des principaux résultats 
de cet article (voir tout le paragraphe \ref{Brehier}, 
no\-tam\-ment la pro\-po\-sition \ref{Cavalcanti}).  
Dans le cas complexe, le résultat découle du résultat principal de 
\cite{SeSt2}, à savoir que tout type semi-simple de $\G$ est un type pour un 
bloc de Bernstein bien déterminé. 
Dans le cas modulaire, nous avons procédé par relèvement de $\flb$ à $\qlb$
pour pouvoir utiliser le résultat en caractéristique nulle.

La suite et fin de la section \ref{NaisKM} est consacrée à l'établissement de 
deux autres propriétés importantes du foncteur $\KM$, \ie sa compatibilité 
à l'induction et à la restriction paraboliques
(voir les propositions \ref{SZ} et \ref{ZS}).
Toutes ces propriétés fondamentales du foncteur $\KM$ 
sont dans le cas modulaire des résultats nouveaux, même dans le cas de 
$\GL_n(\F)$.  
Dans le cas complexe, la compatibilité à l'induction parabolique est 
prouvée dans \cite{SZ}, mais la preuve qui y est donnée ne s'applique pas au 
cas modulaire. 

Ces foncteurs sont utilisés dans \cite{MS12} pour prouver l'unicité 
du support supercus\-pi\-dal d'une représentation irréductible et définir la
notion de repré\-sen\-tation résiduellement non dégénérée de $\G$, 
qui généralise celle de repré\-sen\-tation non dégénérée et 
est à la base de notre clas\-si\-fication des re\-pré\-sen\-ta\-tions irréductibles 
de $\G$ en ter\-mes de multisegments. 
\end{introsec}

\section*{Remerciements}

Nous remercions Jean-François Dat, Guy Henniart, Vanessa Miemietz, 
Shaun Stevens et Ma\-rie-France Vignéras
pour de nombreuses discussions à propos de ce travail, 
et plus particu\-liè\-rement Shaun Stevens pour ses idées qui ont abouti 
à la rédaction de l'appendice.

Une partie de ce travail a été réalisée lors du séjour des auteurs à 
l'Erwin Schrödinger Institute en janvier-février 2009 et du second 
auteur à l'Institut Henri Poincaré de janvier à mars 2010~;
que ces deux institutions soient remerciées pour leur accueil 
et leur soutien financier. 
Une autre partie en a été réalisée lors de plusieurs 
séjours à l'Univer\-sity of East Anglia~: nous remercions 
celle-ci pour son accueil et Shaun Stevens 
pour ses nombreuses invitations. 

Alberto M\'\i nguez remercie le CNRS
pour les six mois de délégation dont il a bénéficié en 2011. 

Vincent Sécherre remercie l'Uni\-ver\-sité de la Méditerranée 
et l'Institut de Mathé\-matiques de Luminy, où il 
était en poste durant la majeure partie de ce travail. 

Enfin, nous remercions chaleureusement le rapporteur pour sa lecture 
rigoureuse et sa critique d'une première version de ce texte, qu'il a 
contribué à améliorer. 

\section*{Notations et conventions}

\begin{notasec}
Dans cet article, $\F$ est un corps commutatif localement compact 
et non archi\-médien de carac\-téristique résiduelle notée $p$, 
et $\R$ est un corps algébriquement clos de ca\-ractéristique dif\-fé\-rente de $p$. 
\end{notasec}

\begin{notasec}
Toutes les $\F$-algèbres sont supposées unitaires et de dimension finie.
Par $\F$-al\-gèbre à di\-vi\-sion on entend $\F$-algèbre 
centrale dont l'anneau sous-jacent est un corps, pas forcément
com\-mu\-tatif. 
Si $\K$ est une algèbre à division sur une extension finie de $\F$, 
notons $\Oo_\K$ son anneau d'entiers, $\p_\K$ son idéal maximal, 
$\kk_\K$ son corps résiduel et $q_\K$ le cardinal de $\kk_\K$.
Posons enfin $q=q_\F$.  
\end{notasec}

\begin{notasec}
Une $\R$-{\it re\-pré\-sen\-ta\-tion lisse} d'un groupe topologique $\G$ 
est la donnée d'un $\R$-espace vectoriel $\V$ et d'un homomorphisme 
de $\G$ dans $\GL(\V)$ tel que le stabilisateur dans $\G$ de tout 
vecteur de $\V$ soit ouvert. 
Une $\R$-représentation lisse est dite {\it admissible} si, pour tout sous-groupe 
ouvert et compact $\K$ de $\G$, 
l'espace $\V^\K$ des vecteurs de $\V$ invariants par $\K$ est de dimension 
finie. 
{\it Dans cet article, toutes les représentations sont des 
  $\R$-re\-pré\-sen\-ta\-tions lisses.} 

Un $\R$-{\it caractère} de $\G$ est un homo\-mor\-phis\-me de $\G$ dans 
$\mult\R$ de noyau ouvert.  
Si $\pi$ est une $\R$-représentation de $\G$ et $\chi$ un $\R$-caractère de 
$\G$, on note $\chi\pi$ ou $\pi\chi$ la représentation 
$g\mapsto\chi(g)\pi(g)$. 

Si aucune ambiguïté n'est à craindre, on écrira \textit{caractère} et 
\textit{re\-présentation} plutôt que $\R$-ca\-rac\-tère et 
$\R$-repré\-sen\-ta\-tion. 
\end{notasec}

\section{Préliminaires}
\label{SectionPreliminaire}

\begin{prelisec}
\label{Gred}
On fixe une $\F$-algèbre à division $\D$ 
de degré réduit $d$.
Pour $m\>1$, on note $\A_m=\Mat_{m}(\D)$ la 
$\F$-algèbre des matrices de taille $m\times m$ à coefficients dans 
$\D$, et on pose $\G_{m}=\GL_{m}(\D)$.
\end{prelisec}

\begin{prelisec}
Soit $\G=\G_m$ pour $m\>1$.
On désigne par $\Rr_{\R}(\G)$ la catégorie abélienne des 
re\-pré\-sen\-ta\-tions de $\G$ (qui sont lisses et à coefficients dans $\R$), 
par $\Irr_{\R}(\G)$ l'ensemble des classes d'isomorphisme des
re\-pré\-sen\-ta\-tions ir\-ré\-duc\-tibles de $\G$ et par $\Gg(\G,\R)$ le groupe 
de Gro\-then\-dieck des re\-pré\-sen\-ta\-tions de longueur finie de $\G$. 
Ce der\-nier est un $\ZZ$-module libre de base $\Irr_{\R}(\G)$, canonique\-ment 
muni d'une relation d'or\-dre partiel, notée $\<$. 

Pour toute représentation de longueur finie $\s$ de $\G$, notons
$\sy{\s}$ son image dans $\Gg(\G,\R)$.
Si $\s$ est irréductible, $\sy{\s}$ désigne donc sa classe d'isomorphisme. 

D'après \cite[II.2.8]{Vig1}, toute représentation irréductible de $\G$ est admissible. 
\end{prelisec}

\begin{prelisec}
\label{MenelasEtHermione}
\textit{On fixe une fois pour toutes une racine carrée de $q$ dans $\R$.}
Si $\P=\M\N$ est un sous-groupe parabolique de $\G$ muni d'une décomposition 
de Levi, on note $\rp_{\P}^{\G}$ le foncteur de restriction para\-bolique nor\-ma\-li\-sé
de $\Rr_{\R}(\G)$ dans $\Rr_{\R}(\M)$ et 
$\ip_{\P}^{\G}$ son adjoint à droite, le foncteur d'induction 
para\-bo\-li\-que nor\-ma\-li\-sé lui correspondant.
Ces foncteurs sont exacts et
ils pré\-ser\-vent l'admissibilité et le fait d'être de longueur finie
(voir \cite{Vig1}, chapitre II, paragraphes 2.1, 3.8 et 5.13).

Soit $\P^-$ le sous-groupe parabolique de $\G$ opposé à $\P$ relativement à 
$\M$. 

\begin{prop}
Si $\pi$ et $\s$ sont des représentations admissibles de $\G$ et de $\M$ 
respectivement, 
il y a un iso\-mor\-phisme de $\R$-espaces vec\-to\-riels~:
\begin{equation}
\label{SecondeAdjonctionAdmissible}
\Hom_{\G}(\ip^{\G}_{\P^{-}}(\s),\pi)\simeq
\Hom_{\M}(\s,\rp^{\G}_{{\P}}(\pi))
\end{equation}
dit de seconde adjonc\-tion (voir \cite[II.3.8]{Vig1}).  
\end{prop}

\label{GeEm}
Si $\a=(m_{1},\ldots,m_{r})$ est une famille 
d'entiers $\>1$ de somme $m$, il lui correspond le 
sous-groupe de Levi standard $\M_{\a}$ de $\G_{m}$ constitué des matrices 
diagonales par blocs de tailles $m_{1},\ldots,m_{r}$ respectivement, que 
l'on identifie naturellement à 
$\G_{m_{1}}\times\cdots\times\G_{m_{r}}$. 
On note $\P_{\a}$ 
le sous-groupe para\-bo\-li\-que de $\G_{m}$ de facteur de Levi 
$\M_{\a}$ formé des matrices tri\-an\-gu\-lai\-res su\-pé\-rieu\-res 
par blocs de tailles $m_{1},\ldots,m_{r}$ respectivement, et on note 
$\N_{\a}$ son radical unipotent.
Les foncteurs $\ip_{\P_{\a}}^{\G_{m}}$ et $\rp_{\P_{\a}}^{\G_{m}}$ sont simplement 
notés respectivement $\ip_{\a}$ et $\rp_{\a}$.
Si, pour chaque en\-tier $i\in\{1,\ldots,r\}$, on a une 
représentation $\pi_{i}$ de $\G_{m_i}$, on pose~: 
\begin{equation*}
\label{VentreDieu}
\pi_1\times\cdots\times\pi_r=\ip_{\a}(\pi_1\otimes\cdots\otimes\pi_r).
\end{equation*}
\end{prelisec}

\vspace{-0.7cm}

\begin{prelisec}
\label{DefRepCusp}
Une représentation irréductible de $\G$ est {\it cus\-pi\-da\-le} 
si son image par $\rp^{\G}_{\P}$ est nulle pour tout sous-groupe parabolique 
propre $\P$ de $\G$, \ie si elle n'est isomorphe à aucune sous-re\-pré\-sentation 
(ou, de fa\c{c}on équivalente, à aucun quotient) d'une induite 
parabolique propre.

Une représentation irréductible de $\G$ 
est {\it super\-cus\-pi\-da\-le} si elle n'est isomorphe à aucun 
sous-quo\-tient d'une représentation de la forme $\ip^\G_\P(\s)$, où 
$\P$ est un sous-groupe parabolique propre de $\G$ et 
$\s$ une représentation \textit{irréductible}\footnote{
On peut prouver qu'on peut omettre la condition d'irréductibilité sur $\s$
dans cette définition, 
\ie qu'une représentation supercuspidale de $\G$ 
n'est isomorphe à aucun sous-quo\-tient d'une induite parabolique propre 
de la forme $\ip^\G_\P(\s)$ avec $\s$ \textit{lisse}~;
voir \cite{SeSt3}.
} 
d'un facteur de Levi de $\P$.

Étant donnée une re\-pré\-sen\-ta\-tion ir\-ré\-ductible $\pi$ de $\G$, 
il existe des entiers $m_{1},\ldots,m_{r}\>1$ de somme $m$, 
et, pour chaque $i\in\{1,\dots,r\}$, il existe une
re\-pré\-sen\-ta\-tion irréductible cuspi\-dale $\rho_i$ de $\G_{m_i}$, 
de telle sorte que $\pi$ soit une sous-représentation de 
$\rho_1\times\dots\times\rho_r$.  
On note~:
\begin{equation*}
\cusp(\pi)
\end{equation*}
la somme formelle $\sy{\rho_1}+\dots+\sy{\rho_r}$
dans le monoïde commutatif libre de base la réunion disjointe des 
$\Irr_{\R}(\G_{m})$, $m\>1$. 
Elle est uniquement déterminée et s'appelle le sup\-port cuspidal de $\pi$. 
En outre, il y a une permutation $w$ de $\{1,\dots,r\}$ telle que $\pi$ 
soit un quotient de $\rho_{w(1)}\times\dots\times\rho_{w(r)}$.  
Pour tous ces résultats, on renvoie à \cite[II.2.20]{Vig1} et \cite[\S2]{MS12}.
\end{prelisec}

\begin{prelisec}
\label{DefHecke}
Soient $\H$ un sous-groupe ouvert de $\G$ et 
$\s$ une représentation de $\H$ sur un $\R$-espace vectoriel $\V$.  
On note $\cind_\H^\G(\s)$ l'induite compacte de $\s$ à $\G$, constituée 
des fonctions $f:\G\to\V$ localement constantes à support compact modulo 
$\H$ telles que $f(hg)=\s(h) f(g)$ pour $h\in\H$, $g\in\G$, et~:
\begin{equation*}
\label{DefHecke1}
\Hh(\G,\s)
\end{equation*}
l'{\it al\-gè\-bre de Hecke} de $\G$ relativement à $\s$, \ie l'algèbre
des $\G$-endomorphismes de $\cind_\H^\G(\s)$. 
Par réciprocité de Fro\-be\-nius et décomposition de Mackey, elle 
s'identifie à l'algèbre de convolution des fonctions $f:\G\to\End_\R(\V)$ 
telles que $f(hgh')=\s(h)\circ f(g)\circ\s(h')$ pour tous $h,h'\in\H$ et $g\in\G$
et dont le support est une union finie de $\H$-doubles classes. 

Si $\s$ est le caractère trivial du groupe $\H$, on note simplement 
$\Hh(\G,\H)$ l'algèbre de Hecke qui lui correspond.

On appelle {\it ensemble d'en\-tre\-la\-ce\-ment} de $\s$ dans ${\G}$ 
l'ensemble des $g\in\G$ pour lesquels il existe une fonction 
$f\in\Hh(\G,\s)$ telle que $f(g)\neq0$.
\end{prelisec}

\begin{prelisec}
\label{ROk}
\label{Antisthene}
Soit $\ell$ un nom\-bre premier différent de $p$. 
On note 
$\QQ_{\ell}$ le corps des nombres $\ell$-adiques, 
$\ZZ_{\ell}$ son anneau d'entiers et 
$\FF_{\ell}$ le corps résiduel de $\ZZ_{\ell}$. 
On fixe une clôture algébrique 
$\overline{\QQ}_{\ell}$ de $\QQ_{\ell}$, on note 
$\overline{\ZZ}_{\ell}$ son anneau d'entiers et 
$\overline{\FF}_{\ell}$ le corps résiduel de $\overline{\ZZ}_{\ell}$,
qui est une clôture algébrique de $\FF_{\ell}$. 

\begin{defi}
Une représentation de $\G$ sur un $\qlb$-es\-pa\-ce vectoriel $\V$ est 
{\it entière} si elle est ad\-mis\-si\-ble et si elle admet une 
{\it structure entière}, \ie un sous-$\zlb$-module de $\V$ stable par $\G$ et
engendré par une base de $\V$ sur $\qlb$ 
(voir \cite{Vig1,Vig6}). 
\end{defi}

Soit $\pi$ une $\qlb$-représentation irréductible entière de $\G$.  
On a les propriétés suivantes, 
d'après \cite[Theorem 1]{Vig7} et \cite[II.5.11]{Vig1}~:
\begin{enumerate}
\item 
toutes les struc\-tures en\-tières de $\pi$ sont de type fini comme 
$\zlb\G$-modules~; 
\item
si $\vv$ est une struc\-ture en\-tière de $\pi$, 
la re\-pré\-sen\-ta\-tion de $\G$ sur $\vv\otimes\flb$ est de lon\-gueur finie~; 
\item
la semi-sim\-pli\-fiée de $\vv\otimes\flb$, qu'on note $\r_{\ell}(\pi)$
et qu'on appelle la {\it réduction modulo $\ell$} de $\pi$, ne dé\-pend pas du choix 
de $\vv$ mais seulement de la classe d'isomorphisme de $\pi$. 
\end{enumerate}
Par li\-néa\-rité, ceci définit un mor\-phis\-me de groupes~: 
\begin{equation}
\label{HomRes}
\r_{\ell}:\Gg(\G,\qlb)^{{\rm ent}}\to\Gg(\G,\flb),
\end{equation}
le membre de gauche étant le sous-groupe de $\Gg(\G,\qlb)$
engendré par les clas\-ses 
de $\qlb$-représen\-ta\-tions irré\-duc\-tibles entières de $\G$. 

\begin{rema}
\label{lamaline}
Si $\H$ est un groupe profini, alors toute $\qlb$-re\-pré\-sen\-ta\-tion de 
dimension finie de $\H$ est en\-tiè\-re (\cite[théorème 32]{Serre})
et on a un morphisme de réduction $\r_{\ell}$ analogue à (\ref{HomRes}).
\end{rema}
\end{prelisec}

\vspace{-0.3cm}

\begin{prelisec}
\label{focales}
Soit $\s$ une $\qlb$-représentation entière d'un sous-groupe ouvert 
$\H\subseteq\G$.  
Si $\vv$ est une structure entière de $\s$, alors 
d'après \cite[Proposition II.3]{Vig7} le sous-module $\mathfrak{i}(\vv)$
des fonctions à valeurs dans $\vv$ est une structure entière de 
$\ind^\G_\H(\s)$ et l'homomorphisme naturel de $\ind^\G_\H(\vv\otimes\flb)$ 
dans $\mathfrak{i}(\vv)\otimes\flb$ est un 
isomorphisme de $\flb$-représentations. 
Si en outre $\ind^\G_\H(\s)$ est de longueur finie, alors sa ré\-duc\-tion
modulo $\ell$ est égale à $\ind^\G_\H(\r_{\ell}(\s))$ dans 
$\Gg(\G,\flb)$. 
\end{prelisec}

\begin{prelisec}
On choisit des racines carrées de $q$ dans $\qlb$ et $\flb$ de sorte que 
la seconde soit la réduction modulo $\ell$ de la première. 
Soit $\P=\M\N$ un sous-groupe parabolique de $\G$. 
Si $\vv$ est une structure entière d'une $\qlb$-représentation entière
$\s$ de $\M$, 
alors d'après \cite[II.4.14]{Vig1}
le sous-espace $\ip(\vv)$ 
des fonc\-tions à valeurs dans $\vv$ est 
une structure entière de $\ip^{\G}_{\P}(\s)$ et 
le morphisme naturel de 
$\ip(\vv)\otimes\flb$ dans $\ip^{\G}_{\P}(\vv\otimes\flb)$
est un isomorphisme de $\flb$-représentations. 
Si en outre $\s$ est de longueur finie, la réduction
modulo $\ell$ de $\ip^{\G}_{\P}(\s)$ est égale à $\ip^\G_\P(\r_{\ell}(\s))$ 
dans $\Gg(\G,\flb)$.  
\end{prelisec}

\section{Types semi-simples}
\label{TSmodl}

Dans \cite{BK,BK2}, 
Bushnell et Kutzko ont construit des repré\-sen\-tations 
irréductibles complexes de certains sous-groupes ou\-verts com\-pacts 
de $\GL_n(\F)$, pour $n\>1$, appelées types semi-simples, per\-mettant 
une étude très fine et exhaustive des représentations lisses complexes de ce groupe. 
Cette construction a ensuite été adaptée aux re\-pré\-sentations 
modu\-lai\-res de $\GL_{n}(\F)$ 
par Vignéras \cite{Vig1,Vig2} et généralisée aux représentations complexes de
$\GL_m(\D)$, pour $m\>1$, par Broussous, Sé\-cherre, Stevens \cite{Br2,VS1,VS2,VS3,SeSt2}.  
Dans cette section, nous définissons plus généralement les types semi-simples 
de $\GL_m(\D)$ à coefficients dans $\R$. 

Dans les paragraphes \ref{StrSim} à \ref{Paralipomenes}, nous construisons les 
$\R$-types simples de $\GL_m(\D)$ à partir des strates et des caractères 
simples, et nous calculons leurs algèbres de Hecke.
La construction est très proche de celle développée dans le cas complexe 
et ne s'en écarte qu'en trois occasions~:
\begin{enumerate}
\item
Pour prouver l'existence des $\b$-extensions 
(voir le paragraphe \ref{bext}), elle-même s'appuyant sur l'existence des relations de 
cohérence et du transfert (voir la proposition \ref{TSF}), nous utilisons le 
lemme \ref{indirredkappa} dont la preuve nécessite un argument adapté au cas 
où la caractéristique de $\R$ est non nulle.
Ceci est dû au fait que, dans ce cas, une représentation lisse d'un
groupe compact n'est pas toujours semi-simple, et avoir une algèbre
d'endomorphismes de dimension $1$ n'équivaut pas à être irréductible. 
\item
Pour la même raison, on ne peut pas appliquer l'argument de \cite[5.3.2]{BK} 
dans le cas modu\-laire pour prouver l'irréductibilité de certai\-nes
représentations. 
On introduit le lemme \ref{MlleDeLaMole}, d'où dérivent les 
corollaires \ref{Mathilde} et \ref{LAmourLApresMidi}. 
\item
Pour déterminer la structure de l'algèbre de Hecke d'un type simple
(paragraphe \ref{Paralipomenes}), 
un argument adapté au cas 
où la caractéristique de $\R$ est non nulle est nécessaire pour prouver le 
lemme \ref{pinv}, car dans le cas modulaire une homothétie non nulle peut être 
de trace nulle.
\end{enumerate}

Aux paragraphes \ref{EndoEq} et \ref{NonEndoEq}, 
nous définissons les $\R$-types semi-simples de $\GL_m(\D)$ com\-me des 
paires cou\-vran\-tes de types simples maximaux de sous-groupes de Levi.
Dans le cas homo\-gène, \ie lorsqu'une seule endo-classe de ps-caractère apparaît,
la méthode suivie dans le cas 
complexe (\cite{BK2,SeSt2}) s'adapte sans difficulté au cas modu\-laire.  
Des difficultés appa\-rais\-sent dans le cas non homogène~: 
\begin{enumerate}
\item
Une fois construit un type semi-simple $(\BJ,\bl)$ par récurrence comme dans 
le cas complexe (pro\-position \ref{PCTSM}), nous suivons l'argument de 
\cite[Corollary 6.6]{BK2} pour prouver que $(\BJ,\bl)$ est une paire couvrante. 
Celui-ci repose sur le fait que tout élément inversible 
à gau\-che de l'algèbre de Hecke 
$\Hh(\G,\bl)$ est inversible, ce fait lui-même étant prouvé à 
partir de \cite[7.15]{BK1}, \ie que tout module 
à droite simple sur $\Hh(\G,\bl)$ est de dimension finie.
Dans le cas modulaire, on ne peut pas appliquer l'argument prouvant 
\cite[7.15]{BK1} dans le cas complexe, 
car on ne sait pas que tout tel module est, 
à isomorphisme près, de la forme $\Hom_{\BJ}(\bl,\pi)$ pour une 
représentation irréductible $\pi$ de $\G$ con\-ve\-nable | du moins 
on ne le sait pas encore puisque ce sera une conséquence de la 
pro\-po\-si\-tion \ref{pptf1} et du théorème \ref{qptf}.
A la place, on utilise \cite[Proposition III.2]{Blondel}. 
\item 
Dans le cas complexe, la structure de $\Hh(\BJ,\bl)$ 
est déterminée à l'aide de \cite[12.1]{BK1}, dont la preuve n'est pas valable 
(ou pour le moins nécessite des explications) dans le cas modulaire. 
A la place, nous utilisons une majoration de l'ensemble d'entrelacement 
de $\bl$ dans $\G$ effec\-tuée en appendice (voir ci-dessous).
\end{enumerate}

Au paragraphe \ref{defeta}, les résultats que nous obtenons sont 
nouveaux, même dans le cas complexe. 
Étant donné un type semi-simple $(\BJ,\bl)$ de $\G$, nous définissons une 
décomposition~:
\begin{equation*}
\bl=\bk\otimes\bs
\end{equation*}
analogue à la décomposition bien connue pour les types simples.
La restriction de $\bk$ au radical pro-unipotent $\BJ^1$ de $\BJ$ est une 
représentation ir\-ré\-ductible $\bn$ dont l'entrelacement est déterminé grâce 
au travail effectué dans l'appendice.
Il est de la forme 
$\BJ(\L\cap\mult\B)\BJ$ où $\mult\B$ est le centra\-lisateur d'un certain 
élément $\b$ dans $\G$ et où $\L$ est un sous-groupe de Levi de $\G$ 
correspondant à la présence de plusieurs endo-classes lorsque $(\BJ,\bl)$ 
n'est pas homogène.
Cette formule d'entre\-lacement est utilisée dans la preuve de la proposition 
importante \ref{pptf1}.

Dans le paragraphe \ref{RedmodlK} enfin, étant donné un nombre premier $\ell\neq p$,
nous étudions la ré\-duc\-tion modulo $\ell$ des $\qlb$-types simples,
ce qui fournit une preuve de l'existence des $\flb$-types simples différente 
de celle des paragraphes \ref{bext} et \ref{Pontesprit}.

\subsection{Strates simples}
\label{StrSim}

Dans tout ce qui suit, on suppose le langage des strates simples 
connu du lecteur. 
On trouvera plus de détails dans \cite{VS1,SeSt1} 
(voir aussi \cite{BK,BK2} dans le cas déployé).
On fixe un entier $m\>1$~; on pose $\A=\A_m$ et $\G=\G_m$.

Soit $[\La,n,0,\b]$ une stra\-te simple de $\A$ (au sens par exemple 
de \cite[\S1.6]{SeSt1}).
Elle est constituée d'une $\Oo_\D$-suite de réseaux $\La$ de $\D^m$ 
et d'un élément $\b\in\A$ satisfaisant à certaines conditions. 
Il correspond à $\La$ une famille décroissante $(\U_k(\La))_{k\>0}$ 
de sous-groupes ouverts compacts de $\G$ et un $\Oo_\F$-ordre 
héréditaire de $\A$, noté $\AA(\La)$.
La sous-$\F$-algèbre de $\A$ engendrée par $\b$ est un corps, noté $\F(\b)$. 

À cette strate simple on associe
(voir \cite[\S2.4]{SeSt1}) 
deux sous-groupes ou\-verts compacts $\H(\b,\La)$, $\J(\b,\La)$ 
de $\U(\La)=\U_0(\La)$. 
Chacun d'eux est filtré par une suite décrois\-sante de pro-$p$-sous-grou\-pes 
ouverts compacts~:
\begin{equation*}
\H^k(\b,\La)=\H(\b,\La)\cap\U_k(\La),
\quad
\J^k(\b,\La)=\J(\b,\La)\cap\U_k(\La), 
\quad k\>1. 
\end{equation*}
On renvoie à \cite[\S3.3]{VS1} et à \cite[\S2]{SeSt1} pour une étude détaillée des 
propriétés de ces groupes. 

\subsection{Caractères simples}
\label{CarSim}
 
On choisit un homomorphisme injectif $\iota_{p,\R}$
du groupe $\mu_{p^{\infty}}(\CC)$ 
des racines complexes de l'unité d'ordre une puissance 
de $p$ vers le groupe multiplicatif $\mult\R$ ainsi qu'un 
caractère com\-ple\-xe additif $\psi_{\F,\CC}:\F\to\mult\CC$
tri\-vial sur $\p_\F$ mais pas sur $\Oo_\F$.

Soit $[\La,n,0,\b]$ une strate simple de $\A$, et soit 
$q_0=-k_0(\b,\La)$ l'opposé de son exposant criti\-que.  
Dans \cite[\S2.4]{SeSt1}, on asso\-cie à cette strate simple et à tout entier 
$m$ tel que $0\<m\<q_0-1$ un ensemble fini $\Cc_{\CC}(\La,m,\b)$ 
(qui dépend de $\psi_{\F,\CC}$) de carac\-tè\-res complexes de
$\H^{m+1}(\b,\La)$ appelés ca\-rac\-tè\-res simples de niveau $m$.  

Les $\H^{m+1}(\b,\La)$ étant des pro-$p$-groupes pour 
$m\in\{0,\dots,q_0-1\}$,
l'ap\-pli\-ca\-tion $\t\mapsto\iota_{p,\R}\circ\t$ 
est bien définie pour tout $\t\in\Cc_{\CC}(\La,m,\b)$.
L'image de $\Cc_{\CC}(\La,m,\b)$ par cette application, 
notée $\Cc_{\R}(\La,m,\b)$ 
-- ou simplement $\Cc(\La,m,\b)$ si aucune ambiguïté n'en résulte --
est un en\-sem\-ble fini de $\R$-carac\-tè\-res de
$\H^{m+1}(\b,\La)$ appelés $\R$-{\it caractères simples de niveau $m$}.

Toutes les propriétés des caractères simples complexes se transportent 
aux $\R$-ca\-ractères simples. 
On renvoie le lecteur à \cite{BK,BK2,VS1,SeSt1,BSS1,SeSt2} pour une étude 
dé\-tail\-lée de ces pro\-priétés.  
On rappelle simplement ici le principe du transfert.
Soit un entier $m'\>1$, et soit $[\La',n',0,\b']$ une strate simple de $\A_{m'}$. 
On suppose qu'il y a un isomorphisme de $\F$-algèbres $\h$ 
de $\F(\b)$ vers $\F(\b')$ tel que $\h(\b)=\b'$.  
Il existe alors une bijection~:
\begin{equation*}
\Cc(\La,0,\b)\to\Cc(\La',0,\b')
\end{equation*}
canoniquement associée à $\h$, appelée application de transfert 
(voir \cite[\S2.6]{SeSt1}).

On utilise\-ra également la notion de ps-caractère, ainsi que la relation 
d'endo-équivalence entre ps-caractères, pour lesquelles on renvoie à 
\cite[\S8]{BH1} et à \cite{BSS1}. 

\subsection{Représentations de Heisenberg}
\label{BetaExt}

Soit $[\La,n,0,\b]$ une strate sim\-ple de $\A$.
Posons $\E=\F(\b)$.
Le centralisateur $\B$ de $\E$ dans $\A$ 
est une $\E$-al\-gèbre centrale simple.
L'égalité $\J(\b,\La)=(\U(\La)\cap\mult\B)\J^1(\b,\La)$ induit un isomorphisme 
canonique de groupes~: 
\begin{equation}
\label{QuoRedFin}
\J(\b,\La)/\J^{1}(\b,\La)\simeq
(\U(\La)\cap\mult\B)/(\U_1(\La)\cap\mult\B)
\end{equation}
permettant d'associer canoniquement et bijectivement une représentation de
$\J=\J(\b,\La)$ tri\-via\-le sur $\J^1=\J^{1}(\b,\La)$ à une représentation de 
$\U(\La)\cap\mult\B$ triviale sur $\U_1(\La)\cap\mult\B$. 

Soit $\t\in\Cc(\La,0,\b)$ un caractère simple. 
On note $\KK(\La)$ le normalisateur de $\La$ dans $\G$~;
c'est un sous-groupe ouvert et compact modulo le centre de $\G$. 
D'après \cite[\S2.10]{SeSt1}, le groupe $(\KK(\La)\cap\mult\B)\J$ 
normalise $\t$ et l'en\-sem\-ble 
d'entrelacement (voir le paragraphe 1.5) 
de $\t$ dans ${\G}$ est égal à $\J^1\mult\B\J^1$.  

\begin{prop}
\label{blaeta}
Il existe une représentation irréductible $\n$ de $\J^1$, 
unique à isomorphisme près, dont la restriction à $\H^1=\H^1(\b,\La)$ 
contient $\t$. 
Elle possède les propriétés suivantes~:
\begin{enumerate}
\item
sa dimension sur $\R$ est égale à $(\J^1:\H^1)^{1/2}$ et 
sa restriction à $\H^1$ est un multiple de $\t$~;
\item 
son ensemble d'entrelacement dans ${\G}$ est égal à 
$\J^1\mult\B\J^1$ et, pour tout 
$y\in\mult\B$, le $\R$-espace d'en\-tre\-lacement
$\Hom_{\J^1\cap(\J^1)^{y}}(\n,\n^y)$
est de di\-men\-sion $1$.
\end{enumerate}
\end{prop}

\begin{proof}
Puisque $\J^1$ est un pro-$p$-groupe, et $\R$ étant 
algébriquement clos et de carac\-té\-ristique différente de $p$,
toutes ses
représentations sont semi-simples, et la preuve existant dans le cas 
complexe (voir \cite[\S8.3]{BF} et \cite[Lemma 5.1]{SeSt2}) s'applique.
\end{proof}

Une telle représentation $\n$ s'appelle {\it re\-pré\-sen\-ta\-tion de
Heisenberg} associée à $\t$. 
Sa classe d'iso\-morphisme est normalisée par $(\KK(\La)\cap\mult\B)\J$, et 
l'induite de $\n$ à $\U_1(\La)$ est irréductible. 

\subsection{$\b$-extensions}
\label{bext}

Une $\b$-{\it extension} de $\n$ (ou de $\t$) est une 
représentation irré\-duc\-tible de $\J$ prolon\-geant $\n$ 
et dont l'ensemble d'entre\-lacement contient $\mult\B$. 
Lorsque $\R$ est le corps des nombres complexes, 
on sait d'après \cite[Théorème 2.28]{VS2} et 
\cite[Proposition 6.4]{SeSt2} 
que tout ca\-rac\-tère simple admet une $\b$-ex\-tension.
La méthode utilisée dans \cite{VS2,SeSt2} est encore valable 
lorsque le corps $\R$ est quel\-con\-que.
Toutes les étapes peuvent être re\-pri\-ses une à une, à l'exception du 
résultat suivant qui nécessite un argument adapté au cas où la 
caractéristique de $\R$ est non nulle.

\begin{lemm}
\label{indirredkappa}
Soit $\mu$ une représentation irréductible de $\J$ contenant $\n$. 
Alors l'induite de $\mu$ au groupe $(\U(\La)\cap\mult\B)\U_1(\La)$ est irréductible. 
\end{lemm}

\begin{proof}
On note $\rho$ l'induite de $\mu$ à $(\U(\La)\cap\mult\B)\U_1(\La)$
et $\W$ l'espace de $\mu$. 
Comme $\J^1$ est un pro-$p$-grou\-pe, la restriction de $\mu$ à $\J^1$ 
se décompose sous la forme~:
\begin{equation*}
\W=\W(\eta)\oplus\W'
\end{equation*}
où $\W(\eta)$ est une somme directe non nulle 
de copies de $\eta$ et où aucun sous-quotient irréductible de $\W'$ 
n'est isomorphe à $\eta$.
Comme $\J$ normalise la classe d'isomorphisme 
de $\eta$, il laisse stable le sous-espace $\W(\eta)$.
Comme $\mu$ est irréductible, on en déduit que $\W=\W(\eta)$.
La restric\-tion de $\mu$ à $\J^1$ est donc un multiple de $\eta$.

Par application de la formule de Mackey, la restriction de $\rho$ à $\J$ est 
la somme directe des~:
\begin{equation*}
\label{IurhoJ}
\I(u)=\ind^\J_{\J\cap\J^u}(\mu^u),
\end{equation*}
pour $u$ décrivant un système de représentants de doubles classes de 
$(\U(\La)\cap\mult\B)\U_1(\La)$ mod $\J$.
Comme la restriction de $\mu$ à $\J^1$ est un multiple de $\eta$, et comme 
l'ensemble d'entre\-lacement de $\eta$ dans 
$(\U(\La)\cap\mult\B)\U_1(\La)$ est égal à $\J$,
l'espace $\Hom_{\J^1}(\eta,\I(u))$ est non nul si et seulement si $u\in\J$. 
Par conséquent,
$\mu$ est un facteur direct de la restriction de $\rho$ à $\J$, 
et c'en est le seul sous-quotient irréductible contenant $\eta$. 
En particulier, l'algè\-bre des endo\-morphismes de $\rho$ est de dimension 
$1$ sur $\R$. 
Mais ceci ne suffit pas, quand $\R$ est de carac\-téristique non nulle, 
pour en déduire que $\rho$ est irréductible. 
Nous allons appliquer le critère d'irré\-duc\-tibilité \cite[Lemma 4.2]{Vig3}.

Il faut montrer que, pour tout quo\-tient irréductible $\pi$ 
de $\rho$, la représentation $\mu$ est un quotient 
de la restriction de $\pi$ à $\J$. 
Soit $\pi$ un tel quo\-tient irréductible. 
Par réciprocité de Frobenius, sa restriction à $\J$ contient $\mu$. 
On déduit du paragraphe précédent de la preuve que $\mu$ est un facteur direct, 
donc un quotient, de la 
restriction de $\pi$ à $\J$.
\end{proof}

Grâce à ce lemme, on peut prouver comme dans le cas où $\R$ est le corps 
des nombres comple\-xes 
(voir \cite[\S2.4]{VS2} et \cite[Proposition 6.4]{SeSt2}) 
le résultat suivant. 

\begin{prop}
\label{TSF}
Soit $[\La',n',0,\b]$ une strate simple de $\A$ telle qu'on ait 
$\U(\La)\subseteq\U(\La')$,
soit $\t'\in\Cc(\La',0,\b)$ le transfert du caractère simple $\t$ 
et soit $\n'$ la représentation de Heisenberg de $\t'$.
\begin{enumerate}
\item
Si $\k$ est une représentation de $\J$ prolongeant $\n$,
il existe une unique représentation $\k'$ de $\J(\b,\La')$ 
prolongeant $\n'$ telle qu'on ait~:
\begin{equation*}
\label{RelCoh}
\ind^{(\U(\La)\cap\mult\B)\U_1(\La)}_{\J(\b,\La)}(\k)\simeq
\ind^{(\U(\La)\cap\mult\B)\U_1(\La)}_{(\U(\La)\cap\mult\B)\J^1(\b,\La')}(\k'),
\end{equation*}
ces induites étant irréductibles.
\item
L'application $\k\mapsto\k'$ ainsi définie induit une bijection 
(appelée application de transfert)
entre les $\b$-extensions de $\n$ et les $\b$-ex\-tensions de $\n'$.
\end{enumerate}
\end{prop}

A partir de là, on prouve comme dans le cas complexe l'existence de 
$\b$-extensions pour tout $\R$-ca\-rac\-tère simple. 

\begin{prop}
Tout $\R$-ca\-rac\-tère simple $\t\in\Cc(\La,0,\b)$ admet une 
$\b$-ex\-ten\-sion.  
\end{prop}

Dans le cas où $\R=\flb$,
on verra au paragraphe \ref{RedmodlK} 
comment prouver ce résultat par réduction modulo $\ell$,
à partir de l'existence de $\b$-extensions dans le cas complexe. 

Comme dans \cite[Proposition 6.4]{SeSt2}, 
si $\k$ est une $\b$-extension de $\t$, alors l'ensemble des $\b$-ex\-tensions 
de $\t$ est égal à~:
\begin{equation}
\label{StHonore}
\{\k\otimes(\chi\circ\N_{\B/\E})\ |\ 
\text{$\chi$ caractère de $\Oo_\E^\times$ trivial sur $1+\p_\E$}\},
\end{equation}
où $\N_{\B/\E}$ désigne la norme réduite de $\B$ sur $\E$, 
et où $\chi\circ\N_{\B/\E}$ est vu comme un caractère de $\J$ trivial sur 
$\J^{1}$ grâce à \eqref{QuoRedFin}. 

L'importance des $\b$-ex\-ten\-sions est justifiée par le résultat suivant, 
qui se démontre exactement comme dans le cas complexe
(voir \cite[5.3.2]{BK} et \cite[Lemme 4.2(2)]{VS3}).

\begin{lemm}
\label{FactorisationEntrelacement}
Soient $\k$ une $\b$-extension de $\t$ et $\xi$ une représentation 
irréductible de $\J$ tri\-viale sur $\J^1$.
Pour tout $g\in\mult\B$, on a~:
\begin{equation*}
\Hom_{\J\cap\J^g}(\k\otimes\xi,\k^g\otimes\xi^g)
=\Hom_{\J\cap\J^g}(\k,\k^g)\otimes\Hom_{\J\cap\J^g}(\xi,\xi^g).
\end{equation*}
\end{lemm}

Pour prouver qu'une représentation de $\J$ 
de la forme $\k\otimes\xi$ comme ci-dessus est irréductible, 
l'argument de \cite[5.3.2]{BK} ne s'applique pas dans le cas modulaire, 
pour la même raison que celle invoquée dans la preuve du lemme 
\ref{indirredkappa}. 
Voici d'abord un lemme général.

\begin{lemm}
\label{MlleDeLaMole}
Soit $\k$ une représentation irréductible d'un groupe profini 
$\emph{\textsf{J}}$.
Supposons qu'il y a un pro-$p$-sous-groupe ouvert 
distingué $\emph{\textsf{J}}^1$ de $\emph{\textsf{J}}$ tel que la restriction $\n$
de $\k$ à $\emph{\textsf{J}}^1$ soit irréductible. 
Alors le foncteur $\W\mapsto\k\otimes\W$ induit une équivalence 
entre~: 
\begin{enumerate}
\item 
la catégorie 
des représentations de $\emph{\textsf{J}}/\emph{\textsf{J}}^1$~;
\item 
la catégorie $\Rr(\emph{\textsf{J}},\n)$ des représentations de 
$\emph{\textsf{J}}$ dont la restriction à $\emph{\textsf{J}}^1$ est $\n$-isotypi\-que. 
\end{enumerate}
\end{lemm}

\begin{proof}
Comme $\textsf{J}^1$ est un pro-$p$-groupe, et comme l'induction préserve la 
projectivité, l'induite $\P$ de $\n$ à $\textsf{J}$ est projective dans la catégorie 
des représentation de $\textsf{J}$.
Comme $\textsf{J}$ norma\-lise la classe d'isomorphisme de $\n$, la 
restriction de $\P$ (donc de chacun de ses sous-quotients) à 
$\textsf{J}^1$ est $\n$-isotypique. 
Ainsi $\P$ est un progénérateur de $\Rr({\textsf{J}},\n)$.
Il s'ensuit (voir par exemple 
\cite[Pro\-po\-sition 1.1]{Roche}) que 
le foncteur $\pi\mapsto\Hom_{\textsf{J}^1}(\n,\pi)$ 
induit une équi\-va\-lence entre $\Rr({\textsf{J}},\n)$ 
et la catégorie des modules à droite sur $\End_\textsf{J}(\P)$~; un quasi-inverse 
est donné par le foncteur $\m\mapsto\m\otimes\P$.

Identifions maintenant $\P$ à la représentation 
$\k\otimes\P_1$ où $\P_1$ 
désigne la représentation régulière de $\textsf{J}/\textsf{J}^1$, \ie 
l'induite à $\textsf{J}$ du caractère trivial de $\textsf{J}^1$.
L'application $f\mapsto \k\otimes f$ définit un isomorphisme de 
$\R$-algèbres de $\R[\textsf{J}/\textsf{J}^1]$, l'algèbre de groupe de 
$\textsf{J}/\textsf{J}^1$, vers $\End_\textsf{J}(\P)$. 
Le résultat s'ensuit. 
\end{proof}

Appliquons ce lemme au cas où 
$\k$ est une représentation du groupe $\textsf{J}=\J(\b,\La)$ prolongeant $\n$ 
et $\textsf{J}^1=\J^1(\b,\La)$.

\begin{coro}
\label{Mathilde}
Soit $\k$ une représentation de $\J$ prolongeant $\n$,
et soit $\xi$ une représentation irréductible de $\J/\J^1$.
Alors la représentation $\k\otimes\xi$ est irréductible. 
\end{coro}

\begin{coro}
\label{LAmourLApresMidi}
Soit $\pi$ une représentation irréductible d'un sous-groupe ouvert de 
$\G$ contenant $\J$.
On suppose que $\pi$ contient $\t$.
Alors il y a une représentation $\k$ de $\J$ prolongeant $\n$ 
et une 
re\-pré\-sentation irréductible $\xi$ de $\J$ triviale sur 
$\J^{1}$ telle que $\k\otimes\xi$ soit une sous-représentation 
(irréductible) de la restriction de $\pi$ à $\J$.
\end{coro}

\begin{proof}
Puisque $\J^{1}$ est un pro-$p$-groupe, 
la restriction de $\pi$ à $\J^{1}$ est semi-sim\-ple.
Elle contient donc la représentation $\n$ et se décompose sous la forme~: 
\begin{equation*}
\W=\W(\eta)\oplus\W'
\end{equation*}
où $\W(\eta)$ est une somme directe non nulle de copies de $\eta$ et où 
aucun sous-quotient irréductible de $\W'$ n'est isomorphe à $\eta$. 
Notons $\V$ la représentation de $\J$ sur $\W(\eta)$, et notons 
$\Y$ la représenta\-tion de $\J/\J^1$ lui correspondant par le 
lemme \ref{MlleDeLaMole}. 
Comme $\J/\J^1$ est fini, $\Y$ admet une sous-re\-pré\-sentation 
irréductible $\xi$. 
Par le lemme \ref{MlleDeLaMole} encore, $\k\otimes\xi$ est une 
sous-représentation irré\-duc\-ti\-ble de $\V$, donc de la restriction de $\pi$ 
à $\J$. 
\end{proof}

\subsection{Types simples}
\label{Pontesprit}

On définit les $\R$-types simples de $\G$ comme dans le cas complexe 
(voir \cite[\S4.1]{VS3} 
-- en particulier la remarque 4.1 permettant de faire 
des types simples de niveau zéro un cas particulier de types simples -- 
ainsi que \cite[\S5.5]{SeSt2}).  
Rappelons brièvement cette définition. 

Soit $[\La,n,0,\b]$ une strate simple de $\A$ telle que 
$\AA(\La)\cap\B$ soit un $\Oo_\E$-ordre principal de $\B$, de période 
notée $r$. 
On fixe un isomorphisme de $\E$-algèbres~: 
\begin{equation}
\label{MonEpousee}
\Phi:\B\to\Mat_{m'}(\D')
\end{equation} 
où $m'$ est un entier et $\D'$ une $\E$-algèbre à division 
convenables, identifiant $\AA(\La)\cap\B$ à un ordre principal standard. 
Plus précisément, si $d'$ désigne le degré réduit de $\D'$ sur $\E$, on a~: 
\begin{equation}
\label{MPRIME}
m'd'=\frac{md}{[\E:\F]}.
\end{equation}
On en déduit des iso\-mor\-phis\-mes de groupes~:
\begin{equation}
\label{IDENT}
\J/\J^1\simeq(\U(\La)\cap\mult\B)/(\U_1(\La)\cap\mult\B)\simeq\GL_{s}(\kk_{\D'})^r
\end{equation}
où $s$ est l'entier défini par $m'=rs$.
Soit $\t\in\Cc(\La,0,\b)$ 
et soit $\k$ une $\b$-ex\-ten\-sion de $\t$. 
Fixons une représentation irréductible cus\-pidale $\s_0$ de 
$\GL_{s}(\kk_{\D'})$. 
La représentation $\s_0^{\otimes r}$ définit par \eqref{IDENT} et par inflation une 
re\-pré\-sen\-ta\-tion irréductible $\s$ de $\J$ triviale sur $\J^1$. 
Posons $\l=\k\otimes\s$.

\begin{defi}
Un couple de la forme $(\J,\l)$ est appelé un {\it type simple} de $\G$.
\end{defi}

Compte tenu du corollaire \ref{Mathilde}, tout type simple de $\G$ 
est irréductible.  

Si $\U(\La)\cap\mult\B$ est un sous-groupe ouvert compact maximal 
de $\mult\B$, le type simple $(\J,\l)$ est dit \textit{maximal}. 
C'est équivalent à la condition $r=1$.

Si la strate simple $[\La,n,0,\b]$ est nulle, le type simple $(\J,\l)$ est dit
\textit{de niveau $0$}. 
Dans ce cas, on a $\E=\F$ et $\J=\U(\La)$, 
et on choisira systématiquement pour $\k$ le caractère trivial de $\U(\La)$.

\begin{exem}
\label{ExeIwa}
Soit $\Iw$ le sous-groupe d'Iwahori standard de $\G$, 
\ie le sous-groupe de $\GL_{m}(\Oo_\D)$ formé des matrices dont la 
réduction modulo $\p_\D$ est triangulaire supérieure,
et soit $1_\Iw$ son caractère trivial.
Alors le couple $(\Iw,1_\Iw)$ est un type simple de niveau $0$ de $\G$.
\end{exem}

Comme dans le cas complexe (voir \cite[Proposition 1]{BHisc}), 
on a la propriété suivante.

\begin{prop}
\label{isc}
Fixons un caractère simple $\t\in\Cc(\La,0,\b)$.
\begin{enumerate}
\item 
Soit $[\La',n',0,\b]$ une strate simple de $\A$ telle 
que $\U(\La)\subseteq\U(\La')$
et soit $\t'\in\Cc(\La',0,\b)$ le transfert de $\t$.  
Toute représentation irréductible de $\G$ contenant $\t$ con\-tient 
aussi $\t'$.
\item
Soit $\pi$ une représentation irréductible de $\G$ contenant une 
représentation irréductible de $\J$ de la forme $\k\otimes\xi$ où 
$\xi$ est une représentation irréductible de $\J$ triviale sur $\J^1$.
Supposons que $\xi$ considérée comme représentation de 
$\GL_{s}(\kk_{\D'})^r$ n'est pas cuspidale.  
Alors il y a une strate simple $[\La',n',0,\b]$ de $\A$ telle 
que $\U(\La')\subsetneq\U(\La)$
et telle que $\pi$ contient le transfert $\t'\in\Cc(\La',0,\b)$ de $\t$.  
\end{enumerate}
\end{prop}

\begin{proof}
L'assertion 1 suit des propriétés de transfert des représentations de 
Heisenberg (voir \cite[Proposition 2.12]{VS2}).
Pour l'assertion 2, on raisonne comme dans la preuve de 
\cite[Proposition 5.15]{SeSt1}.
\end{proof}

\subsection{Algèbre de Hecke d'un type simple}
\label{Paralipomenes}

Soit $(\J,\l)$ un type simple de $\G$.
On note $b(\l)$ le cardinal de l'orbite de la classe d'isomorphis\-me de $\s$ 
sous l'action du groupe de Galois~: 
\begin{equation}
\label{Bliblij}
\Ga=\Gal(\kk_{\D'}/\kk_\E)
\end{equation}
où $\s$ est considérée comme une re\-pré\-sen\-ta\-tion de 
$\GL_{s}(\kk_{\D'})^{r}$ grâce au choix de $\Phi$.
Cet entier ne dépend ni de ce choix ni de celui de $\k$.

\begin{rema}
\label{Etoc1}
Dans le cas déployé, \ie lorsque $\D$ est égal à $\F$, le groupe de Galois 
$\Ga$ est trivial et on a $b(\l)=1$ pour tout type simple $(\J,\l)$. 
\end{rema}

On fixe une uniformisante $\w$ de $\D'$, et on note $\Ww_\l$ le
sous-groupe de $\GL_r(\D')$ constitué 
des matrices monomiales dont les coeffi\-cients non 
nuls sont des puissances de~:
\begin{equation}
\label{pilam}
\w_\l=\w^{b(\l)}.
\end{equation}
On voit $\Ww_\l$ comme un sous-groupe de $\mult\B$ par l'intermédiaire 
de $\Phi$ et du plongement diagonal de $\D'$ dans $\Mat_s(\D')$. 

\begin{lemm}
\label{EntTypSim}
L'ensemble d'entrelacement de $\l$ dans $\G$ est la réunion dis\-join\-te des 
$\J{w}\J$ pour $w\in\Ww_\l$ 
et, pour chaque $w\in\Ww_\l$, on a $\dim_{\R}\Hom_{\J\cap\J^w}(\l,\l^w)=1$. 
\end{lemm}

\begin{proof}
Dans le cas où $(\J,\l)$ est de niveau $0$, on utilise la proposition 1.2 
et le lemme 1.5 de \cite{GSZ} dont les preuves sont encore valables dans le 
cas modulaire.  
Dans le cas de niveau non nul, on utilise le lemme 
\ref{FactorisationEntrelacement} pour se ramener au cas de niveau $0$. 
\end{proof}

\begin{rema}
\label{Etoc2}
Supposons que $(\J,\l)$ est un type simple maximal.
L'ensemble d'entrelace\-ment de $\l$ dans $\G$ 
est égal au nor\-ma\-li\-sa\-teur de $\l$ dans $\G$, noté $\N_\G(\l)$.  
Celui-ci est engendré par $\J$ et $\w_\l$. 
Comme le normalisateur de $\U(\La)\cap\mult\B$ dans $\mult\B$ est engendré 
par $\U(\La)\cap\mult\B$ et $\w$, l'entier $b(\l)$ est égal à l'indice
de $\N_\G(\l)$ dans $\N_{\mult\B}(\U(\La)\cap\mult\B)\J$. 
\end{rema}

Notons $\Hh(r,q(\l))$ l'algèbre de Hecke affine de type $\A_{r-1}$
et de paramètre $q(\l) =q^{fd's}$, 
où $f$ désigne le degré résiduel de $\E$ sur $\F$. 
C'est la $\R$-al\-gèbre engendrée par des éléments 
$\SS_0,\dots,\SS_{r-1}$ et $\Pi,\Pi^{-1}$ avec pour tout $ i\in\ZZ/r\ZZ$ les 
relations~: 
\begin{eqnarray}
\label{R1aff}
(\SS_i+1)(\SS_i-q(\l))&=&0,\\ 
\label{R2aff}
\SS_i\SS_j&=&\SS_j\SS_i, \quad j\notin\{i-1,i+1\},\\
\label{R3aff}
\SS_i\SS_{i+1}\SS_i&=&\SS_{i+1}\SS_i\SS_{i+1},\\ 
\Pi\SS_i&=&\SS_{i-1}\Pi, \\
\Pi\Pi^{-1}=\Pi^{-1}\Pi&=&1.
\end{eqnarray}

Le groupe $\Ww_\l$ est engendré par les matrices de permutation
$s_1,\dots,s_{r-1}$ correspondant aux trans\-positions 
$i\leftrightarrow i+1$ et par l'élément~:
\begin{equation*}
\label{Defhl}
h=h_\l=
\begin{pmatrix}
0 & {\rm id}_{r-1}\\
\w_\l & 0 
\end{pmatrix}\in\GL_{r}(\D')
\end{equation*}
où ${\rm id}_{r-1}$ est la matrice identité de $\Mat_{r-1}(\D')$.
Si l'on pose $s_0=hs_1h^{-1}$, tout élément $w\in\Ww_\l$ peut s'écrire 
sous la forme $w=h^{a}s_{i_1}\dots s_{i_l}$ avec $a\in\ZZ$ et $l\>0$ et 
$i_1,\dots,i_l\in\ZZ/r\ZZ$, et~:
\begin{equation}
\label{DEFSW}
\SS_w=\Pi^a\SS_{i_1}\dots\SS_{i_l}
\end{equation}
ne dépend pas du choix de cette écriture. 
Les $\SS_w$, $w\in\Ww_\l$ forment une base d'espace vectoriel 
de $\Hh(r,q(\l))$ sur $\R$. 

\begin{prop}
\label{ANenPlusFinirPosDuplicata}
On a un isomorphisme de $\R$-al\-gèbres~: 
\begin{equation}
\label{BaudrillardPosDuplicata}
\Psi:\Hh(r,q(\l))\to\Hh(\G,\l)
\end{equation}
tel que, pour tout $w\in\Ww_\l$, la fonction $\Psi(\SS_w)$ soit de support 
$\J{w}\J$. 
\end{prop}

\begin{proof}
On reprend l'argument de \cite{VS3} en précisant les modifications devant 
lui être apportées. 
On fixe une strate simple $[\La_{{\rm max}},n_{{\rm max}},0,\b]$ de $\A$ telle 
que l'intersection $\U(\La_{{\rm max}})\cap\mult\B$ soit un sous-groupe 
compact maximal de $\mult\B$ contenant $\U(\La)\cap\mult\B$ et dont 
l'image par \eqref{MonEpousee} soit le sous-groupe compact maximal standard 
$\GL_{m'}(\Oo_{\D'})$. 
On pose~:
\begin{eqnarray*}
\GB&=&(\U(\La_{{\rm max}})\cap\mult\B)/(\U_1(\La_{{\rm max}})\cap\mult\B), \\
\PB&=&(\U(\La)\cap\mult\B)(\U_1(\La_{{\rm max}}) \cap\mult\B)/(\U_1(\La_{{\rm max}})\cap\mult\B), \\
\MB&=&(\U(\La)\cap\mult\B)/(\U_1(\La)\cap\mult\B).
\end{eqnarray*}
On identifie $\GB$ au groupe $\GL_{m'}(\kk_{\D'})$ et $\MB$ au sous-groupe de 
Levi standard $\GL_{s}(\kk_{\D'})^{r}$.
La re\-pré\-sentation $\s$ identifiée à une représentation de $\MB$ est de la 
forme $\s_0^{\otimes r}$. 
D'après \cite[Theorem 4.12]{James}, 
l'algèbre des endomorphismes de l'induite parabolique de ${\s}$ à $\GB$ 
le long de $\PB$ est isomorphe à la sous-$\R$-algèbre $\Hh^0(r,q(\l))$ 
engendrée par $\SS_1,\dots,\SS_{r-1}$. 
Plus précisément, il y a un isomorphisme de $\R$-al\-gèbres~: 
\begin{equation}
\label{Psi0}
\Psi^0:\Hh^0(r,q(\l))\to\Hh(\U(\La_{{\rm max}})\cap\mult\B,\s)
\end{equation}
tel que, pour tout $i\in\{1,\dots,r-1\}$, la fonction $\Psi^0(\SS_{i})$ soit 
de support $\J{s_i}\J$. 

Notons $\t_{{\rm max}}\in\Cc(\La_{{\rm max}},0,\b)$ le 
transfert de $\t$ et $\k_{{\rm max}}$ la 
$\b$-extension de $\t_{{\rm max}}$ trans\-fert de $\k$. 
Reprenant la preuve du lemme 4.4 et de la proposition 4.5 de \cite{VS3}, 
on obtient des mor\-phismes de $\R$-algèbres~: 
\begin{equation}
\label{Psi1}
\Hh(\U(\La_{{\rm max}})\cap\mult\B,\s)
\to\Hh(\J(\b,\La_{{\rm max}}),\k_{{\rm max}}\otimes\s)
\hookrightarrow\Hh(\G,\l),
\end{equation}
le premier étant un isomorphisme et le second injectif. 
En composant \eqref{Psi0} et \eqref{Psi1}, on obtient 
un homomorphisme injectif $\Psi$ de $\R$-algèbres de 
$\Hh^0(r,q(\l))$ dans $\Hh(\G,\l)$.
Fixons maintenant une fonction $\h\in\Hh(\G,\l)$ de support $\J h\J$. 
On va montrer qu'il existe un unique isomorphisme 
de $\R$-al\-gèbres \eqref{BaudrillardPosDuplicata} 
prolongeant $\Psi$ et satisfaisant à la condition $\Psi(\Pi)=\h$. 
Pour cela, on reprend la preu\-ve de \cite[Théorème 4.6]{VS3}, 
à ceci près que la preu\-ve de \cite[Lemme 4.14]{VS3}
doit être mo\-di\-fiée comme indiqué ci-dessous. 

\begin{lemm}
\label{pinv}
L'élément $\h$ est inversible dans $\Hh(\G,\l)$. 
\end{lemm}

\begin{proof}
Soit $\h'\in\Hh(\G,\l)$ une fonction de support $\J h^{-1}\J$.
Le produit de convolution 
$\h*\h'$ est un multiple scalaire de l'élément neutre dans $\Hh$. 
Il est donc soit nul, soit inversible, et on va montrer que sa valeur en $1$ 
n'est pas nulle. 
Fixons un ensemble $\X$ de représentants de $\J$ modulo $\J\cap h^{-1}\J h$, 
qu'on peut supposer être contenu dans $\J^1$.
On a~:
\begin{equation*}
\h*\h'(1) = \sum\limits_{x\in\X}\l(x)\circ\h(h) \circ\h'(h^{-1}) \circ\l(x)^{-1}.
\end{equation*}
D'après le lemme \ref{FactorisationEntrelacement}, 
l'opérateur d'entrelacement $\h(h)\in\Hom_{\J\cap\J^h}(\l,\l^h)$ 
se décom\-po\-se sous la forme $\Phi_\k\otimes\Phi_\s$
avec $\Phi_\k\in\Hom_{\J\cap\J^h}(\k,\k^h)$ et 
$\Phi_\s\in\Hom_{\J\cap\J^h}(\s,\s^h)$.
En reprenant l'argument de \cite[Lemma 1.5]{GSZ}, 
on montre qu'on a un isomorphisme de $\R$-espaces vectoriels~:
\begin{equation*}
\Hom_{\J\cap\J^h}(\s,\s^h)\simeq\Hom_{\J/\J^1}(\s,\s^h),
\end{equation*}
ce qui prouve que $\Phi_\s$ est inver\-sible.
De façon analogue, l'opérateur d'entrelacement $\h'(h^{-1})$ se décompose 
sous la forme 
$\Phi'_\k\otimes\Phi'_\s$, où $\Phi'_\s$ est l'inverse de $\Phi_\s^{}$.  
On a donc $\h*\h'(1)=\Omega\otimes{\rm id}_{\s}$ où 
${\rm id}_\s$ est l'identité sur l'espace de $\s$ et où~: 
\begin{equation*}
\Omega = \sum\limits_{x\in\X}\n(x)\circ\Phi_\k^{}\circ\Phi'_\k \circ\n(x)^{-1}.
\end{equation*}
Comme $\J^1\cap h^{-1}\J^1h$ est un pro-$p$-groupe, la restriction de $\n$ à
celui-ci est semi-simple. 
Il y a donc un unique facteur irréductible en commun $\V$ entre les restrictions à 
$\J^1\cap h^{-1}\J^1h$ de $\n$ et de $\n^{h}$, et 
l'opérateur $\Phi_\k^{}\circ\Phi'_\k$ est un multiple non nul de la projection 
sur $\V$.
En tant que représentation lisse 
irréductible d'un pro-$p$-groupe, la dimension de $\V$ est 
une puissance de $p$, qui est non nulle dans $\R$.
On en déduit que la tra\-ce de $\Omega$ est non nulle, ce qui termine la 
démonstration du lemme \ref{pinv}. 
\end{proof}

À partir de là, on termine en reprenant l'argument de \cite[Théorème 4.6]{VS3}.
Ceci met fin à la preuve de la proposition \ref{ANenPlusFinirPosDuplicata}.
\end{proof}

Compte tenu de ce qui précède, on obtient le résultat suivant.

\begin{coro}
\label{InvSupUni}
Tout élément non nul de $\Hh(\G,\l)$ supporté par une double 
classe mod $\J$ est inver\-sible.
\end{coro}

\begin{exem}
\label{EmilPost}
Si l'on prend pour $(\J,\l)$ le type simple $(\I,1_\I)$ de l'exemple 
\ref{ExeIwa}, la proposition \ref{ANenPlusFinirPosDuplicata} montre que 
l'algèbre de Hecke-Iwahori $\Hh(\G,\I)$ est isomorphe à la $\R$-algèbre de 
Hecke affine $\Hh(m,q^{d})$.
\end{exem}

\subsection{Paires couvrantes}
\label{PCMJ}

Dans ce paragraphe, nous rappelons quelques faits concernant la théorie 
des paires couvran\-tes.
Pour plus de détails, nous renvoyons le lecteur à \cite{BK1,Vig2}.

Soit $\tau$ une représentation irréductible d'un sous-groupe ouvert compact
$\K$ de $\G$. 
Il lui corres\-pond la $\R$-al\-gè\-bre 
$\Hh=\Hh(\G,\tau)$ définie au paragraphe 
\ref{SectionPreliminaire}.\ref{DefHecke} et le foncteur~: 
\begin{equation*}
\label{Lalka}
\Mm_{\tau}:\s\mapsto\Hom_{\K}(\tau,\s)
\end{equation*}
de $\Rr=\Rr_{\R}(\G)$ vers la catégorie des $\Hh$-modules à droite. 
Par réciprocité de Frobenius, celui-ci s'identifie au foncteur 
$\s\mapsto\Hom_{\G}(\cind^{\G}_{\K}(\tau),\s)$.

Soit $\M$ un sous-groupe de Levi de $\G$ et soit $\tau_\M$ une 
représentation irréductible d'un sous-grou\-pe ouvert compact $\K_\M$ 
de $\M$.
Supposons que $(\K,\tau)$ est une paire couvrante de $(\K_{\M},\tau_{\M})$ 
au sens de \cite[8.1]{BK1}.  

\begin{rema}
\label{PosNeg}
Dans \cite{BK1}, l'algèbre associée à la paire $(\K,\tau)$ est l'algèbre 
$\Hh(\G,\tau^\vee)$ asso\-ciée à la contragré\-dien\-te $\tau^\vee$ de $\tau$.
Elle est opposée à $\Hh$ (voir \cite[2.3]{BK1}). 
Dans la condition (iii) de \cite[8.1]{BK1} il faut donc remplacer 
``$(\P,\J)$-positive'' par ``$(\P,\J)$-negative''.
\end{rema}

On a fixé au para\-graphe 
\ref{SectionPreliminaire}.\ref{MenelasEtHermione} une racine carrée de $q$ dans $\R$. 
Soit $\P$ un sous-groupe parabolique de $\G$ muni d'une décomposition 
de Levi $\P=\M\N$.
Il lui correspond un homo\-morphis\-me injectif de $\R$-algèbres~:
\begin{equation}
\label{tePe}
\te_{\P}:\Hh(\M,\tau_{\M})\to\Hh(\G,\tau)
\end{equation}
faisant de $\Hh$ un module à gauche sur 
$\Hh_{\M}=\Hh(\M,\tau_{\M})$
(voir \cite[II.10]{Vig2}~: pour la raison don\-née à la remarque \ref{PosNeg}, 
l'homomorphisme $\te_\P$ préserve les supports des fonctions de $\Hh_\M$ 
supportées par les éléments $\P$-négatifs de $\M$, et non pas $\P$-positifs 
comme dans \cite{BK1}).
Cet homo\-morphisme définit un fonc\-teur de restriction, 
noté $\jp_{\P}^{*}$, 
de la catégorie des $\Hh$-modules à droite vers la catégorie 
des $\Hh_{\M}$-mo\-du\-les à droite.
Pour toute re\-pré\-sen\-ta\-tion $\s$ de $\G$, la pro\-jec\-tion 
na\-tu\-rel\-le de $\s$ vers son module de Jacquet $\rp_{\P}^{\G}(\s)$ 
induit un iso\-mor\-phis\-me~:
\begin{equation}
\label{Gobineau}
\jp_{\P}^{*}(\Mm_{\tau}(\s))\simeq\Mm_{\tau_\M}(\rp_{\P}^{\G}(\s))
\end{equation}
de $\Hh_{\M}$-modules à droite (voir \cite[II.10.1]{Vig2}).

\begin{rema}
\label{SolecismeEvident}
On déduit de \eqref{Gobineau} que, 
si une représentation $\s$ de $\G$
contient une paire couvrante, 
alors $\rp_{\P}^{\G}(\s)$ est non nul.
En particulier, une représentation ir\-ré\-duc\-tible
de $\G$ conte\-nant une paire couvrante pour $\M\subsetneq\G$ 
n'est pas cuspidale.
\end{rema}

La propriété \eqref{Gobineau} permet d'obtenir le résultat suivant. 

\begin{prop}
\label{GobineauAdj} 
Pour tout $\Hh_\M$-module à droite $\mm$ de dimension finie, 
on a un iso\-mor\-phis\-me de représentations de $\G$~:
\begin{equation*}
\mm\otimes_{\Hh_\M}\cind_{\K}^{\G}(\tau)\simeq
\ip_{\P^-}^{\G}(\mm\otimes_{\Hh_\M}\cind_{\K_{\M}}^{\M}(\tau_{\M})),
\end{equation*}
où $\P^-$ désigne le sous-groupe parabolique de $\G$ opposé
à $\P$ relativement à $\M$.
\end{prop}

\begin{rema}
L'hypothèse sur la dimension de $\mm$ provient du fait que notre 
preuve utilise la propriété de seconde adjonction 
\eqref{SecondeAdjonctionAdmissible}. 
\end{rema}

\begin{proof}
Le foncteur $\Mm_{\tau} $ admet un adjoint à gauche~:
\begin{equation*}
\mm\mapsto\mm\otimes_{\Hh}\cind^{\G}_{\K}(\tau)
\end{equation*}
de la catégorie des $\Hh$-modules à droite vers la catégorie $\Rr$, et 
le foncteur $\jp_\P^*$ admet un adjoint à gauche
$\mm\mapsto\mm\otimes_{\Hh_\M}\Hh$
de la catégorie des $\Hh_{\M}$-modules à droite vers celle des 
$\Hh$-modules à droi\-te.
À partir de (\ref{Gobineau}) et en utilisant la 
seconde adjonction \eqref{SecondeAdjonctionAdmissible}, 
on obtient le résultat. 
\end{proof}

D'après Blondel \cite[II]{Blondel} 
(voir aussi Dat \cite[§2]{Dat1}), 
on a un isomorphisme~: 
\begin{equation}
\label{BlondelDat}
\Phi_{\P}:\cind_{\K}^{\G}(\tau)\to\ip_{\P}^{\G}
(\cind_{\K_{\M}}^{\M}(\tau_{\M})),
\quad\Phi_{\P}(f)(g)(x)=\int\limits_{\N}f(uxg)\ du,
\end{equation}
de représentations de $\G$ et de $\Hh_{\M}$-modules à gauche, 
où $du$ est la mesure de Haar sur $\N$ normalisée de telle sorte que 
$\K\cap\N$ soit de volume $1$, avec $f\in \cind_{\K}^{\G}(\tau)$, 
$g\in\G$ et $x\in\M$.

\begin{prop}
\label{ScotErigene}
Soit $\s$ une représentation de $\M$ 
engendrée par sa composan\-te $\tau_{\M}$-iso\-ty\-pique.
Alors l'induite $\ip_\P^\G(\s)$ est engendrée par 
sa composante $\tau$-isotypique.
\end{prop}

\begin{proof}
Par hypothèse, $\s$ est quotient d'une somme directe 
de copies de $\cind_{\K_{\M}}^{\M}(\tau_{\M})$, \ie qu'il existe un
ensemble $\SS$ et un homo\-mor\-phis\-me 
surjectif de représentations de $\M$~:
\begin{equation*}
\bigoplus\limits_{\SS}\cind_{\K_{\M}}^{\M}(\tau_{\M})\to\s.
\end{equation*}
Si l'on applique le foncteur exact $\ip_\P^\G$ (qui commute aux sommes
directes arbitraires) et la formule 
(\ref{BlondelDat}), on voit que l'induite parabolique $\ip_\P^\G(\s)$ 
est un quotient d'une somme de copies de $\cind_{\K}^{\G}(\tau)$,
\ie qu'elle est engendrée par sa composante $\tau$-isotypique.
\end{proof}

\subsection{Types semi-simples I : le cas homogène}
\label{EndoEq}

Soit $\a=(m_1,\dots,m_r)$ une fa\-mil\-le d'entiers $\>1$ de somme $m$. 
Pour chaque entier $i\in\{1,\dots,r\}$, fixons un type simple maximal 
$(\J_i,\l_i)$ de $\G_{m_i}$.
On pose $\M=\M_\a$, $\P=\P_\a$, $\N=\N_\a$
et on note $\l_\M$ la représen\-ta\-tion 
$\l_1\otimes\dots\otimes\l_r$ du groupe $\J_\M=\J_1\times\cdots\times\J_r$. 

Pour chaque $i\in\{1,\dots,r\}$, 
fixons une strate simple $[\La_i,n_i,0,\b_i]$ de $\A_{m_i}$ telle que $\J_i$ 
soit égal à $\J(\b_i,\La_i)$ et
fixons un caractère simple $\t_i\in\Cc(\La_{i},0,\b_i)$ contenu dans le type 
simple $\l_i$.  
On note $({\bf\Theta}_i,0,\b_i)$ le ps-caractère 
(voir \cite[Definition 1.5]{BSS1}) défini par le couple 
$([\La_i,n_i,0,\b_i],\t_i)$.

\begin{defi}
\label{DefTSMlevi}
Un couple de la forme $(\J_\M,\l_\M)$ est appelé un 
\textit{type simple maximal de $\M$}.
\end{defi}

\textit{Dans ce paragraphe, on suppose que les ps-caractères 
$({\bf\Theta}_i,0,\b_i)$, $i\in\{1,\dots,r\}$ sont tous endo-équivalents 
(au sens de \cite[Definition 1.10]{BSS1})
à un ps-caractère fixé $({\bf\Theta},0,\b)$.} 
On peut supposer, comme dans \cite[\S8.1]{SeSt2}, que les caractères 
simples $\t_1,\dots,\t_r$ sont des réalisations de $({\bf\Theta},0,\b)$. 

Fixons une strate simple $[\La,n,0,\b]$ de $\A$ telle que~:
\begin{enumerate}
\item
on a $\La=(\La\cap\D^{m_1})\oplus\dots\oplus(\La\cap\D^{m_r})$ et la 
$\Oo_\D$-suite de réseaux $\La\cap\D^{m_i}$ est dans la classe affine de $\La_i$ pour 
chaque $i\in\{1,\dots,r\}$~;
\item
il existe un isomorphisme de $\E$-algèbres $\Phi:\B\to\Mat_{m'}(\D')$
identifiant $\AA(\La)\cap\B$ à un ordre héréditaire standard dont la réduction
modulo $\p_{\D'}$ est formée des matrices triangulaire par blocs de taille 
$(m'_1,\dots,m'_r)$, où $m'_i$ est l'entier associé à $m^{}_i$ par la relation 
\eqref{MPRIME}. 
\end{enumerate}
Quitte à dilater chaque suite $\La_i$ par un entier convenable,
on supposera dans la suite que toutes les $\La_i$ ont la même période, 
égale à celle de $\La$.
Notons $\t\in\Cc(\La,0,\b)$ la réalisation de $({\bf\Theta},0,\b)$ associée à 
cette strate simple et choisissons une $\b$-extension $\k$ de $\t$. 
On définit des sous-groupes~:
\begin{equation*}
\BJ=\H^1(\b,\La)(\J(\b,\La)\cap\P), 
\quad
\BJ^1=\H^1(\b,\La)(\J^1(\b,\La)\cap\P).
\end{equation*}
Comme aux paragraphes 6.3 et 6.4 de \cite{SeSt2} 
(où les groupes définis ci-dessus sont notés $\J^{}_\P$ et 
$\J^1_\P$) on dé\-fi\-nit des représentations irréductibles 
$\bn$ et $\bk$ de $\BJ^1$ et $\BJ$ respectivement,
dont les propriétés sont décrites par les pro\-po\-sitions 6.5 et 
6.6 de \cite{SeSt2}
(où ces représentations sont notées $\n_\P$ et $\k_\P$).
En particulier, l'induite de $\bn$ à $\J^1(\b,\La)$ est isomorphe à la 
restriction de $\k$ à ce groupe et, 
pour toute représentation irréductible $\xi$ de 
$\U(\La)\cap\mult\B$ triviale sur $\U^1(\La)\cap\mult\B$, 
on a un isomorphisme canonique de représentations~:
\begin{equation}
\label{poumi}
\ind_{\BJ}^{\J(\b,\La)}(\bk\otimes\xi)\simeq \k\otimes\xi.
\end{equation}
La restriction de $\bk$ à $\BJ\cap\M$ se décompose sous la forme 
$\k_1\otimes\dots\otimes\k_r$ où $\k_i$ est une $\b$-extension de $\t_i$ pour 
chaque $i\in\{1,\dots,r\}$.
Décomposons $\l_i$ sous la forme $\k_i\otimes\s_i$ où $\s_i$ est une 
représentation irréductible de $\J_i$ triviale sur $\J^1(\b,\La_i)$.
L'isomorphisme $\Phi$ induit un isomorphisme de $\A_{m_i}\cap\B$ vers 
$\Mat_{m'_i}(\D')$ permettant d'identifier 
$\s_i$ à une représentation cuspidale de $\GL_{m'_i}(\kk_{\D '})$.

On réunit $\s_1,\dots,\s_r$ suivant leur classe de conjugaison sous le groupe 
$\Ga=\Gal(\kk_{\D'}/\kk_\E)$, ce qui définit une partition~:
\begin{equation*}
\{1,\dots,r\}=\I_1\cup\dots\cup\I_u,
\quad u\>1.
\end{equation*}
Quitte à renuméroter les types simples 
$(\J_i,\l_i)$, on peut supposer pour simplifier qu'il existe des entiers 
$0=a_0<a_1<\dots<a_u=r$ tels qu'on ait 
$\I_j=\{i\in\{1,\dots,r\}\ |\ a_{j-1}<i\<a_j\}$ pour chaque
$j\in\{1,\dots,u\}$.  
Ceci définit un sous-groupe de Levi \textit{standard} $\M'$
de $\G$ contenant $\M$. 

Notons $\bs$ la représentation $\s_1\otimes\dots\otimes\s_r$ considérée 
comme représentation de $\BJ$ triviale sur $\BJ^1$ et posons 
$\bl=\bk\otimes\bs$. 
On pose $\BJ_{\M'}=\BJ\cap\M'$ et on note $\bl_{\M'}$ la restriction de $\bl$ à 
$\BJ_{\M'}$.

\begin{rema}
\label{Gurgeh}
Dans le cas où $u=1$, \ie où $\M'=\G$, 
la propriété \eqref{poumi} implique 
que l'induite de $\bl$ à $\J(\b,\La)$ est un type simple $\l$.
Il y a donc un isomorphisme canonique de $\R$-algèbres de 
$\Hh(\G,\bl)$ sur $\Hh(\G,\l)$ préservant les supports, 
la structure de cette dernière étant donnée par la 
propo\-sition \ref{ANenPlusFinirPosDuplicata}. 
En reprenant l'argument de \cite[Proposition 6.7]{SeSt2}, 
on en déduit que 
$(\BJ,\bl)$ est une paire couvrante de $(\J_\M,\l_\M)$.
\end{rema}

Pour $j\in\{1,\dots,u\}$, 
notons $r_j$ le cardinal de $\I_j$ et posons $f_j=fd'm'_i$, où 
$f$ est le degré résiduel de $\E$ sur $\F$ et où $m'_i$ ne dépend pas de 
$i\in\I_j$. 

\begin{prop}
\label{JplPM}
\begin{enumerate}
\item 
La paire $(\BJ,\bl)$ est une paire couvrante de $(\BJ_{\M'},\bl_{\M'})$, 
qui est elle-même une paire couvrante de $(\J_\M,\l_\M)$. 
\item
On a un isomorphisme de $\R$-algèbres de $\Hh({\M'},\bl_{\M'})$ dans 
$\Hh(\G,\bl)$ préservant les supports.
\item
On a un isomorphisme~: 
\begin{equation*}
\label{BaudrillardPosDuplicata81}
\Hh(r_1,q^{f_1})\otimes\dots\otimes\Hh(r_u,q^{f_u})
\to\Hh(\G,\bl)
\end{equation*}
de $\R$-al\-gèbres.
\end{enumerate}
\end{prop}

\begin{proof}
Dans le cas où $\R$ est $\CC$,
le résultat est donné par la proposition 8.1 de \cite{SeSt2}. 
Dans le cas général, il 
suffit de reprendre la preuve de \cite[Proposition 5.17]{SeSt1} 
pour prouver que $(\BJ,\bl)$ est décomposée au-dessus de 
$(\BJ_{\M'},\bl_{\M'})$ et que l'ensemble d'entrelacement de $\bl$ dans $\G$ est 
inclus dans $\BJ{\M'}\BJ$,
les propositions \cite[1.2]{GSZ} et \cite[6.6]{SeSt2} étant encore 
valables dans le cas mo\-du\-laire.
En raisonnant comme dans la preuve de \cite[Theorem 7.2]{BK1}
(voir aussi \cite[II.8]{Vig2}), on en déduit que $(\BJ,\bl)$ est une paire 
couvrante de $(\BJ_{\M'},\bl_{\M'})$ et on a un isomorphisme de $\R$-algèbres 
préservant les supports de $\Hh({\M'},\bl_{\M'})$ dans $\Hh(\G,\bl)$.

Le fait que $(\BJ_{\M'},\bl_{\M'})$ est une paire couvrante de $(\J_\M,\l_\M)$ 
est une conséquence de la remarque \ref{Gurgeh}, et l'isomorphisme du point 3 
est une conséquence de la proposition \ref{ANenPlusFinirPosDuplicata}. 
\end{proof}

Pour cha\-que $i\in\{1,\dots,r\}$, notons $\n_i$ la res\-tric\-tion de $\k_i$ 
à $\J^1(\b,\La_i)$ et posons~:
\begin{equation}
\label{cosmos}
\n_\M=\n_1\otimes\dots\otimes\n_r,
\quad
\k_\M=\k_1\otimes\dots\otimes\k_r,
\quad
\J^1_\M=\J_\M^{}\cap\U_1(\La).
\end{equation}
Les paires $(\BJ^1,\bn)$ et $(\BJ,\bk)$ sont 
décomposées au-dessus de 
$(\J^1_\M,\n^{}_\M)$ et $(\J_\M^{},\k_\M^{})$
respectivement.
On a le résultat suivant\footnote{V.~Sécherre remercie Shaun 
Stevens pour une discussion sur une version préliminaire de la 
preuve de ce résultat.}, qui à notre connaissance est nouveau même dans le cas complexe. 

\begin{prop}
\label{jnpc}
La paire $(\BJ^1,\bn)$ est une paire couvrante de $(\J^1_\M,\n_\M^{})$.
\end{prop}

\begin{proof}
D'après le corollaire qui suit \cite[Proposition III.2]{Blondel}, 
il suffit de prouver que, pour toute représenta\-tion lisse irréductible $\pi$ 
de $\G$, l'application naturelle~: 
\begin{equation}
\label{Ho2}
\Hom_{\BJ^1}(\bn,\pi)\to\Hom_{\J^1_\M}(\n_\M,\rp_\a(\pi))
\end{equation}
induite par la surjection canonique de $\pi$ sur son module de Jacquet 
est injective, ce que l'on va faire par récurrence sur la dimension de $\M$. 
Fixons une repré\-sen\-tation irréductible $\pi$ de $\G$.
Au moyen de l'isomorphisme de $\E$-algèbres $\Phi$ fixé au début du 
paragraphe \ref{EndoEq}, identifions~:
\begin{equation*}
\BJ/\BJ^1\simeq\J_\M^{}/\J^1_\M\simeq(\U(\La)\cap\mult\B)/(\U_1(\La)\cap\mult\B)
\end{equation*}
à un sous-groupe de Levi standard~:
\begin{equation*}
\MB=\GL_{m'_1}(\kk_{\D'})\times\dots\times\GL_{m'_r}(\kk_{\D'})\subseteq \GL_{m'}(\kk_{\D'}).
\end{equation*}
Munissons $\Hom_{\BJ^1}(\bn,\pi)$ de l'action de $\BJ$ définie, 
pour $x\in\BJ$ et $f\in\Hom_{\BJ^1}(\bn,\pi)$, par la formule~:
\begin{equation*}
x\cdot f=\pi(x)\circ f\circ\bk(x)^{-1}.
\end{equation*}
C'est une représentation de $\BJ$ triviale sur $\BJ^1$ que l'on voit comme une 
représentation de $\MB$.
De façon analogue, le membre de droite de \eqref{Ho2} est muni d'une 
structure de représentation de $\MB$.

Considérons maintenant \eqref{Ho2} comme un homomorphisme de 
représentations de $\MB$ et notons $\EuScript{V}$ son noyau, que l'on suppose 
non nul.  
Il existe donc un sous-groupe de Levi standard $\MB{}^\star\subseteq\MB$ 
et une repré\-sen\-tation irréductible cuspidale $\bt$ de $\MB{}^\star$ 
telle que $\EuScript{V}$ possède une sous-représentation irréductible 
de support cuspidal $(\MB{}^\star,\bt)$. 
Si $\MB{}^\star=\MB$, le calcul du composant $\bt$-isotypique 
don\-ne une suite exacte de $\R$-espaces vectoriels~:
\begin{equation*}
\label{E4}
0\to\EuScript{V}^{\bt}\to\Hom_{\BJ}(\vk\otimes\bt,\pi)\to
\Hom_{\J_\M}(\k_\M\otimes\bt,\rp_\a(\pi)).
\end{equation*}
Comme $(\BJ,\vk\otimes\bt)$ est une paire cou\-vrante de 
$(\J_\M,\k_\M\otimes\bt)$ d'après la proposition \ref{JplPM}, 
on trouve $\EuScript{V}^{\bt}=0$, ce qui contredit l'hypothèse 
faite sur $\bt$.
En particulier, si $m'_1=\dots=m'_r=1$, alors on a toujours 
$\MB{}^\star=\MB$ et ainsi \eqref{Ho2} est un isomorphisme dans ce cas. 

On suppose maintenant que $\MB{}^\star$ est un sous-groupe de Levi standard propre de 
$\MB$.
Il y a une famille d'en\-tiers $\g=(t_1,\dots,t_s)$ de somme $m$ telle que, 
si l'on note $(t'_1,\dots,t'_s)$ la famille d'entiers de somme $m'$
associée à $\g$ par la relation \eqref{MPRIME}, alors~:
\begin{equation*}
\MB{}^\star=\GL_{t'_1}(\kk_{\D'})\times\dots\times\GL_{t'_s}(\kk_{\D'}) 
\subseteq \GL_{m'}(\kk_{\D'}).
\end{equation*}
En particulier, $\g$ est plus fine que $\a$.
Fixons une strate simple $[\La',n_{\La'},0,\b]$ de $\A$ telle que~:
\begin{enumerate}
\item
on a $\La'=(\La'\cap\D^{t_1})\oplus\dots\oplus(\La'\cap\D^{t_s})$~;
\item
$\Phi(\AA(\La')\cap\B)$ est un ordre héréditaire standard dont la réduction 
modulo $\p_{\D'}$ est formée des matrices triangulaire par blocs de taille 
$(t'_1,\dots,t'_s)$.
\end{enumerate}
On définit les deux sous-groupes~:
\begin{equation*}
\BJ'=\H^1(\b,\La')(\J(\b,\La')\cap\P_\g),
\quad
\BJ'^1=\H^1(\b,\La')(\J^1(\b,\La')\cap\P_\g).
\end{equation*}
Soit $\k'$ le transfert de $\k$ en une $\b$-extension du transfert 
$\t'\in\Cc(\La',0,\b)$ de $\t$.
Notons 
$\bk'$ la repré\-sen\-ta\-tion de $\BJ'$ sur l'espace des vecteurs de $\k'$ 
invariants par $\J(\b,\La')\cap\N_\g$ et $\bn'$ la restriction de $\bk'$ à 
$\BJ'^1$. 
La restriction de $\bn'$ à~: 
\begin{equation*}
\J'^1_{\g} = \J^1(\b,\La'\cap\D^{t_1})\times\dots\times\J^1(\b,\La'\cap\D^{t_s})
\end{equation*}
est égale à $\n'_\g=\n'_1\otimes\dots\otimes\n'_s$ où 
$\n'_j$ est la représentation de Heisenberg du transfert de $\t$ dans 
$\Cc(\La'\cap\D^{t_j},0,\b)$, pour chaque $j\in\{1,\dots,s\}$.

Notons $\UB{}^\star$ le sous-groupe unipotent standard de $\MB$ correspondant 
à $\MB{}^\star$. 
Calculant les es\-pa\-ces de vecteurs $\UB{}^\star$-invariants 
grâce aux relations de cohérence entre $\b$-extensions, 
on obtient des isomorphismes de représentations de $\MB{}^\star$~:
\begin{eqnarray*}
\Hom_{\BJ^1}(\bn,\pi)^{\UB{}^\star}
&\simeq&\Hom_{\BJ'^1}(\bn',\pi), \\
\Hom_{\J^1_\M}(\n_\M,\rp_\a(\pi))^{\UB{}^\star}
&\simeq&\Hom_{\BJ'^1\cap\M}(\bn'_{\M},\rp_\a(\pi)),
\end{eqnarray*}
où $\bn'_\M$ est la restriction de $\bn'$ à $\BJ'^1\cap\M$. 
Par hypothèse de récurrence, 
la paire $(\BJ'^1\cap\M,\bn'_\M)$ est une paire couvrante de $(\J'^1_{\g},\n'_{\g})$. 
On en déduit que l'application naturelle~:
\begin{equation*}
\Hom_{\BJ'^1\cap\M}(\bn'_\M,\rp_\a(\pi))\to
\Hom_{\J'^1_\g}(\n'_\g,\rp_\g(\pi))
\end{equation*}
est un isomorphisme. 
Appliquant 
le foncteur des $\UB{}^\star$-invariants à \eqref{Ho2} puis prenant 
la com\-po\-san\-te $\bt$-isotypique, on obtient une suite exac\-te 
de $\R$-espaces vectoriels~:
\begin{equation*}
\label{E4bis2}
0\to\EuScript{V}^{\UB{}^\star,\bt}\to\Hom_{\BJ'}(\vk'\otimes\bt,\pi)
\to\Hom_{\BJ'_\g}(\k'_\g\otimes\bt,\rp_{\g}(\pi)) 
\end{equation*}
où $\k'_\g=\k'_1\otimes\dots\otimes\k'_s$ est la restriction de $\bk'$ à~: 
\begin{equation*}
\J'_{\g} = \J(\b,\La'\cap\D^{t_1})\times\dots\times\J(\b,\La'\cap\D^{t_s}),
\end{equation*}
$\k'_j$ étant une $\b$-extension de $\n'_j$.
Comme $(\BJ',\vk'\otimes\bt)$ est une paire couvrante de 
$(\BJ'_\g,\k'_\g\otimes\bt)$ par la proposition \ref{JplPM}, on en déduit que 
$\EuScript{V}^{\UB{}^\star,\bt}=0$, ce qui donne une contradiction. 
\end{proof}

\subsection{Types semi-simples II : le cas général}
\label{NonEndoEq}

Nous traitons maintenant le cas général. 
Réunissons les $({\bf\Theta}_i,0,\b_i)$ suivant leur classe
d'endo-équivalence, ce qui définit une par\-ti\-tion~:
\begin{equation*}
\{1,\dots,r\}=\I_1\cup\dots\cup\I_l,
\quad
l\>1.
\end{equation*}
Ceci définit un sous-groupe de Levi $\L$ de $\G$ contenant $\M$.  
Quitte à renuméroter les types simples $(\J_i,\l_i)$, 
nous pouvons supposer pour simplifier que $\L$ est standard et 
définit une décomposition~: 
\begin{equation*}
\label{DedDNi}
\D^m = \V^{1}\oplus\dots\oplus\V^{l}
\end{equation*}
où, pour chaque $k\in\{1,\dots,l\}$, 
$\V^{k}$ désigne la somme directe des $\D^{m_i}$, $i\in\I_k$.
On suppose comme dans le paragraphe \ref{EndoEq} que les 
$\t_i$, $i\in\I_k$ sont des réalisations d'un mê\-me ps-caractère 
$(\Theta^k,0,\b^k)$ et que les $\La_i$, $i\in\I_k$ ont toutes la même 
période~;
on fixe une suite de réseaux $\La^k$ de $\V^k$ satisfai\-sant à la 
condition du paragraphe \ref{EndoEq}
relativement aux $\La_i$, $i\in\I_k$~;
on note $\M_k$ le sous-groupe de Levi standard 
formé des $\G_{m_i}$, $i\in\I_k$~; on a un type simple maximal 
$(\J_{\M_k},\l_{\M_k})$ de $\M_k$ formé des $(\J_i,\l_i)$, $i\in\I_k$.  
On note enfin $(\BJ_k,\bl_k)$ la paire couvrante de $(\J_{\M_k},\l_{\M_k})$ dans 
$\G_{\dim\V^k}$ associée à la paire $(\b^k,\La^k)$ par la proposition \ref{JplPM}.  

Posons~:
\begin{equation}
\label{Jacque}
\BJ_\L=\BJ_1\times\dots\times\BJ_l,
\quad
\bl_\L=\bl_1\times\dots\times\bl_l.
\end{equation}
Par construction, $(\BJ_\L,\bl_\L)$ est une paire couvrante de $(\J_\M,\l_\M)$ dans 
$\L$. 
Ici encore, quitte à dilater les suites de réseaux, on peut supposer que 
$\La^1,\dots,\La^l$ ont toutes  la même période~;
notons $\La$ la somme directe des $\La^k$ et notons $\b$ la somme des 
$\b^k$. 
Ceci définit une strate $[\La,n,0,\b]$ de $\A$, pas forcément simple. 
On note $\B$ le centralisateur de $\b$ dans $\A$.

Pour ce qui suit, nous aurons besoin de la notion 
d'approximation commune d'une famille de caractères simples 
(\cite[Definition 5.5]{SeSt2}).

\begin{prop}
\label{PCTSM}
Il existe une paire couvrante $(\BJ,\bl)$ de $(\BJ_\L,\bl_\L)$, donc de 
$(\J_\M,\l_\M)$, dans $\G$ possédant les propriétés suivantes~:
\begin{enumerate}
\item
On a $\U_{n+1}(\La)\subseteq\BJ\subseteq\U(\La)$. 
\item
Si $([\La,n,0,\g],\vartheta,t)$ est une approximation commune de 
$(\t_1,\dots,\t_r)$,
le groupe $\BJ$ contient et normalise  
$(\H^t(\g,\La)\cap\L)\H^{t+1}(\g,\La)$, 
la restriction de $\bl$ à $\H^{t+1}(\g,\La)$ est un multiple de $\vartheta$
et la restriction de $\bl$ à $\H^{t}(\g,\La)\cap\M$ est un multiple de 
$\t_1\otimes\dots\otimes\t_r$.
\end{enumerate}
\end{prop}

\begin{proof}
On procède par récurrence sur $l$ 
comme dans le cas complexe (\cite[Theorem 8.2]{SeSt2}).
Une fois construite une paire $(\BJ,\bl)$ vérifiant les propriétés 1 et 2, 
le seul passage requérant un peu d'attention 
pour prouver que c'est une paire couvrante de $(\BJ_\L,\bl_\L)$
est la fin de la preuve de \cite[Corollary 6.6]{BK2}.  

D'après \cite[Proposition III.2]{Blondel}, tout module à droite simple sur 
$\Hh(\G,\bl)$ est de dimension finie.
D'après l'argument de \cite[7.15]{BK1}, qui est encore valable dans le cas modulaire, 
on en déduit que tout élément inversible à gau\-che de $\Hh(\G,\bl)$ est 
inversible. 
A partir de là, on conclut comme dans le cas complexe que $(\BJ,\bl)$ est une 
paire couvrante.
Voir l'introduction à la section \ref{TSmodl} pour plus de détails. 
\end{proof}

\begin{defi}
\label{DefTypSeSim}
Un {\it type semi-simple} de $\G$ est une paire couvrante 
d'un type simple maxi\-mal d'un sous-groupe de Levi de $\G$ 
de la forme $(\BJ,\bl)$ définie par la proposition \ref{PCTSM}. 
\end{defi}

\subsection{La représentation $\bn$ et son ensemble d'entrelacement} 
\label{defeta}

Soit $(\BJ,\bl)$ un type semi-simple. 
Le groupe $\BJ$ a une décomposition d'Iwahori~: 
\begin{equation*}
\label{DECIWAJ}
\BJ = (\BJ\cap\N^-)\cdot(\BJ\cap\M) \cdot (\BJ\cap\N)
\end{equation*}
où $\N^-$ est le radical unipotent opposé à $\N$ par rapport à $\M$. 
Pour $k\in\{1,\dots,l\}$, notons $\bn_k$ la représentation 
irréductible de $\BJ^1_k=\BJ_k^{}\cap\U_1(\La^k)$ contenue dans $\bl_k$, 
et posons~:
\begin{equation*}
\bn_\L=\bn_1\otimes\dots\otimes\bn_l.
\end{equation*}
C'est une représentation irréductible de 
$\BJ_\L^{1}=\BJ^1_1\times\dots\times\BJ^1_l$.  
Posons $\BJ^{1}=\BJ\cap\U_1(\La)$ et reprenons 
les notations de \eqref{cosmos}.
Nous voulons définir des représentations irréductibles $\bn$ et $\bk$ 
de $\BJ^1$ et $\BJ$ généralisant celles du paragraphe \ref{EndoEq}
dans le cas homogène, 
avec une décomposition $\bl=\bk\otimes\bs$.

\begin{prop}
\label{propbn1}
Il y a une unique paire décomposée $(\BJ^1,\bn)$ au-dessus de 
$(\BJ^1_\L,\bn_\L)$, donc aussi au-dessus de $(\J^1_\M,\n^{}_\M)$.
La restriction de $\bl$ à $\BJ^1$ est $\bn$-isotypique.
\end{prop}

\begin{proof}
On a une décomposition d'Iwahori~:
\begin{equation*}
\BJ^1=(\BJ^1\cap\N^-)\cdot\J^1_\M\cdot(\BJ^1\cap\N).
\end{equation*}
Pour tout $y\in\BJ^1$ se décomposant sous la forme 
$y=uxu'$ avec $u\in \BJ^1\cap\N^-$, $x\in\J^1_\M$, 
$u'\in \BJ^1\cap\N$, on pose $\bn(y)=\n_\M(x)$.
Ceci définit une application $\bn$ de $\BJ^1$ dans 
$\GL(\mathcal{V})$, où $\mathcal{V}$ désigne l'espace 
de $\n_\M$, et il s'agit de 
montrer que c'est un homomorphisme de groupes. 
La restriction de $\bl$ à $\BJ^1$ est triviale sur $\BJ^1\cap\N^-$ 
et $\BJ^1\cap\N$ et la restriction de $\bl$ à ${\BJ^1\cap\M}$, 
qui est égale à la restriction de $\l_\M$ à ${\J^1_\M}$, 
est un multiple de $\n_\M$.
Ainsi l'ensemble~:
\begin{equation*}
(\BJ^1\cap\N^-)\cdot\Ker(\n_\M)\cdot(\BJ^1\cap\N) = \Ker(\bl)\cap{\BJ^1}
\end{equation*}
est un sous-groupe de $\G$.
Par conséquent (voir \cite[Lemme 1]{BlondelCRAS}) l'application $\bn$ 
est une repré\-sen\-ta\-tion irré\-duc\-tible de $\BJ^1$ et la paire 
$(\BJ^1,\bn)$ est décomposée au-dessus de $(\J^1_\M,\n^{}_\M)$. 
\end{proof}

\begin{prop}
\label{propbn2}
L'ensemble d'entrelacement de $\bn$ dans $\G$ est égal à 
$\BJ(\L\cap\mult\B)\BJ$ et, pour tout $g\in\L\cap\mult\B$, 
l'espace $\Hom_{\BJ^1\cap(\BJ^1)^g}(\bn,\bn^g)$ est de dimension $1$ sur $\R$.
\end{prop}

\begin{proof}
Si $g\in\G$ entrelace $\bn$, on peut supposer grâce au 
lemme \ref{NicolasRostov} qu'il appartient à $\L$. 
Comme $\bn$ et $\bn_\L$ ont le même ensemble d'entrelacement 
dans $\L$, il suffit de calculer celui de $\bn_\L$.
Pour $k\in\{1,\dots,l\}$, notons $\B^k$ le centralisateur de $\b^k$ dans 
$\End_\D(\V^k)$. 
La composante $g_k$ de $g$ dans le $k$-ième bloc de $\L$ 
entrelace $\bn_k$, donc elle appartient à $\BJ^1_k(\B^{k})^\times\BJ^1_k$ et~:
\begin{equation*}
\dim_\R \Hom_{\BJ^1\cap(\BJ^1)^{g_k}}(\bn_k^{\phantom{g}},\bn_k^{g_k}) = 1
\end{equation*}
d'après la proposition \ref{blaeta}.
On conclut en remarquant que 
$\L\cap\mult\B=(\B^{1})^\times\times\dots\times(\B^{l})^\times$.
\end{proof}

La proposition \ref{propbn2} fournissant une majoration de l'ensemble 
d'entrelacement de $\bl$ dans $\G$, nous en déduisons le résultat suivant. 
Notons $\Q$ le sous-groupe parabolique standard de $\G$ de sous-groupe 
de Levi standard $\L$. 

\begin{coro}
\label{propbn3}
Le morphisme de $\R$-algèbres $\jp_\Q:\Hh(\L,\bl_\L)\to\Hh(\G,\bl)$ est un 
isomorphis\-me préservant les supports.  
\end{coro}

Pour chaque $k\in\{1,\dots,l\}$, fixons une décomposition 
$\bl_k=\bk_k\otimes\bs_k$ comme au paragraphe \ref{EndoEq} et posons~:
\begin{equation*}
\bk_\L=\bk_1\otimes\dots\otimes\bk_l,
\quad
\bs_\L=\bs_1\otimes\dots\otimes\bs_l.
\end{equation*}
Ainsi $\bk_\L$ est une représentation 
irréductible de $\BJ_\L$ prolongeant $\bn_\L$.
Aussi, $\bs_\L$ est une représen\-ta\-tion 
irréductible de $\BJ_\L$ triviale sur $\BJ_\L^{1}$ et on a la 
décomposition $\bl_\L=\bk_\L\otimes\bs_\L$.
Notons $\bs$ la représentation $\bs_\L$ considérée comme 
une représentation irréductible de $\BJ^{}$ triviale sur $\BJ^{1}$ 
grâce à l'isomorphisme $\BJ/\BJ^1\simeq\BJ_\L^{}/\BJ_\L^1$.
On reprend les notations de \eqref{cosmos}.

\begin{prop}
\label{bnpairedec}
\begin{enumerate}
\item
Il y a une unique paire décomposée $(\BJ,\bk)$ au-dessus de 
$(\BJ_\L,\bk_\L)$, donc aussi au-dessus de $(\J_\M,\k_\M)$.
La restriction de $\bk$ à $\BJ^1$ est égale à $\bn$ et on a 
$\bl=\bk\otimes\bs$.
\item
L'ensemble d'entrelacement de $\bk$ dans $\G$ est $\BJ(\L\cap\mult\B)\BJ$.
\end{enumerate}
\end{prop}

\begin{proof}
Grâce à la décomposition d'Iwahori de $\BJ$, on peut définir 
de façon analogue à $\bn$ une application $\bk$ de $\BJ$ dans 
$\GL(\mathcal{V})$ prolongeant $\k_\M$ et triviale sur 
$\BJ\cap\N^-$, $\BJ\cap\N$.
Comme $\BJ\cap\N=\BJ^1\cap\N$ et $\BJ\cap\N^-=\BJ^1\cap\N^-$, on a~: 
\begin{equation*}
(\BJ\cap\N)\cdot(\BJ\cap\N^-) 
\subseteq \Ker(\bl)\cap{\BJ^1} 
\subseteq (\BJ\cap\N^-)\cdot\Ker(\k_\M)\cdot(\BJ\cap\N). 
\end{equation*}
On en déduit que 
$(\BJ\cap\N^-)\cdot\Ker(\k_\M)\cdot(\BJ\cap\N)$
est un sous-groupe de $\G$. 
Ainsi l'application $\bk$ 
est une repré\-sen\-ta\-tion irré\-duc\-tible de $\BJ$ et la paire 
$(\BJ,\bk)$ est décomposée au-dessus de $(\J_\M,\k_\M)$. 

Le point 2 se montre comme pour $\bn$. 
\end{proof}

Le résultat suivant sera utile pour prouver le lemme \ref{LemCritV} 
et la proposition \ref{pptf1}.

\begin{prop}
\label{Empire}
Soit $(\BJ,\bl)$ un type semi-simple de $\G$. 
Alors la représentation de $\BJ$ sur la com\-posante $\bn$-isotypique de 
$\ind^\G_{\BJ}(\bl)$ est une somme directe de conjugués de $\bl$. 
\end{prop}

\begin{proof}
Fixons une décom\-position de $\bl$ sous la forme $\bk\otimes\bs$.
Comme $\BJ^1$ est un pro-$p$-groupe,
la restriction de $\cind_\BJ^\G(\bl)$ à $\BJ^1$ est semi-simple et se 
décompose sous la forme $\V\oplus\W$ où $\V$ est une somme 
directe de copies de $\bn$ et où aucun sous-quotient irré\-duc\-tible de 
$\W$ n'est isomorphe à $\bn$.  
Comme $\BJ$ normalise la classe d'iso\-morphisme 
de $\bn$ (puisque $\bk$ prolonge $\bn$ à $\BJ$), 
il stabilise les sous-espaces $\V$ et $\W$.
Posons~: 
\begin{equation*}
\Pi=\Hom_{\BJ^1}(\bn,\cind_\BJ^\G(\bl))
\end{equation*}
que l'on munit de l'action de $\BJ$ définie, 
pour $x\in\BJ$ et $f\in\Pi$,
par la formule~:
\begin{equation*}
x\cdot f=\pi(x)\circ f\circ\bk(x)^{-1}
\end{equation*}
où $\pi$ est la représentation de $\BJ$ sur $\cind_\BJ^\G(\bl)$. 
C'est une représentation de $\BJ$ triviale sur $\BJ^1$, 
et $\V$ est isomorphe à $\bk\otimes\Pi$. 
Nous allons montrer que la représentation $\Pi$ est une somme directe de 
conjugués de $\bs$, ce dont on déduira le résultat voulu. 

Par la formule de Mackey, 
la restriction de $\cind_\BJ^\G(\bl)$ à $\BJ^1$ est la somme directe des 
$\ind^{\BJ^1}_{\BJ^1\cap\BJ^x}(\bl^x)$ pour $x\in\BJ\backslash\G/\BJ^1$. 
Par conséquent, on a la décomposition~:
\begin{equation*}
\Pi = \bigoplus\limits_{x\in\BJ\backslash\G/\BJ^1}
\Hom_{\BJ^1\cap\BJ^x}(\bn,\bl^x)
\end{equation*}
de $\Pi$ en une somme directe de $\R$-espaces vectoriels. 
Si l'on écrit~:
\begin{equation*}
\Hom_{\BJ^1\cap\BJ^x}(\bn,\bl^x)\subseteq
\Hom_{\BJ^1\cap(\BJ^1)^{x}}(\bn,\bn^x)\otimes\EuScript{V}_{\bs}
\end{equation*}
où $\EuScript{V}_{\bs}$ désigne l'espace de $\bs$ qui compte ici comme un espace 
de multiplicité, 
on voit que seuls les $x\in\G$ qui entrelacent $\bn$
apportent une contribution non nulle.
On peut donc supposer que $x$ appartient à $\L\cap\B^\times$.

\begin{lemm}
\label{Azad2}
Supposons que $x\in\L\cap\B^\times$. 
Alors~:
\begin{equation*}
\Hom_{\BJ^1\cap\BJ^x}(\bn,\bl^x)
\simeq \Hom_{\BJ^1\cap\BJ^x}(1,\bs^x)
= \Hom_{(\U_1(\La)\cap\mult\B) \cap(\U(\La)\cap\mult\B)^x}(1,\bs^x).
\end{equation*}
\end{lemm}

\begin{proof}
Il suffit de reprendre la preuve de la proposition \cite[5.3.2]{BK}, 
en sachant que 
l'es\-pa\-ce $\Hom_{\BJ^1\cap\BJ^x}(\bn,\bk^x)$ est de dimension $1$ sur $\R$ 
d'après la proposition \ref{propbn2}. 

L'égalité provient du fait que 
$\BJ^1\cap\mult\B = \U_1(\La)\cap\mult\B$ 
et $\BJ\cap\mult\B = \U(\La)\cap\mult\B$.
\end{proof}

Comme $\bs$ est cuspidale, seuls les $x\in\L\cap\B^\times$ normalisant 
$\U(\La)\cap\mult\B$ apportent une contribu\-tion non nulle.  
Pour de tels $x$, on voit que $\Hom_{\BJ^1\cap\BJ^x}(\bk,\bl^x)$ est 
isomorphe à $\bs^x$ comme représenta\-tion de $\BJ$. 
\end{proof}



\subsection{Réduction modulo $\ell$ des types simples}
\label{RedmodlK}
\label{RedTypSimNivPos+}

Soit $\ell$ un nombre premier différent de $p$.
Rappelons que toute $\qlb$-re\-pré\-sen\-ta\-tion 
irréductible d'un groupe profini peut être réduite modulo $\ell$
(voir la remarque \ref{lamaline}).

Rappelons d'abord la théorie de la réduction modulo $\ell$ des 
représentations cuspidales de $\GL_n$ sur un corps fini de 
caractéristique $p$ (voir \cite{James}). 
Soient $\ff$ un corps fini de caractéristique $p$ 
et $\overline{\ff}$ une clôture algé\-brique de $\ff$. 
Pour $n\>1$, notons $\ff_n$ l'ex\-ten\-sion de $\ff$ 
de degré $n$ contenue dans $\overline{\ff}$. 
Si $\R$ est $\qlb$ ou $\flb$,
notons $\Xx_{n}(\R)$ l'ensemble des $\R$-carac\-tè\-res 
de $\mult\ff_n$ dont l'orbite sous $\Gal(\overline{\ff}/\ff)$ 
est de cardinal $n$. 
D'après Green \cite{Green}, on a une corres\-pon\-dance surjective~:
\begin{equation}
\label{Quilty}
\tilde\chi\mapsto{\bf g}_{\qlb}(\tilde\chi)
\end{equation}
de $\Xx_{n}(\qlb)$ vers l'ensemble des classes de $\qlb$-représentations 
irréductibles cuspidales de $\GL_n(\ff)$, 
et l'ensemble des antécédents de ${\bf g}_{\qlb}(\tilde\chi)$ par 
(\ref{Quilty}) 
est l'or\-bi\-te de $\tilde\chi$ sous $\Gal(\overline{\ff}/\ff)$.

\begin{theo}[James \cite{James}]
\label{theojames}
Soit un entier $n\>1$.
\begin{enumerate}
\item
Pour tout caractère $\tilde\chi\in\Xx_{n}(\qlb)$, 
la représentation $\r_{\ell}({\bf g}_{\qlb}(\tilde\chi))$ 
est ir\-ré\-duc\-ti\-ble et cuspidale, et elle 
ne dépend que de la 
réduction modulo $\ell$ de $\tilde\chi$.
\item
On a une correspondance surjective~:
\begin{equation}
\label{QuiltyL}
\chi\mapsto {\bf j}_{\flb}(\chi)=\r_{\ell}({\bf g}_{\qlb}(\tilde\chi))
\end{equation}
de l'ensemble des $\flb$-ca\-rac\-tères $\chi$ de $\mult\ff_n$
admettant un relèvement $\tilde\chi$ dans $\Xx_{n}(\qlb)$
vers l'en\-sem\-ble des classes de $\flb$-représentations 
irréductibles cuspidales de $\GL_n(\ff)$~;
l'ensemble des antécédents de ${\bf j}_{\flb}(\chi)$ 
est l'orbite de $\chi$ sous $\Gal(\overline{\ff}/\ff)$.
\item
La représentation ${\bf j}_{\flb}(\chi)$ est supercuspidale si et seulement si 
$\chi\in\Xx_{n}(\flb)$.
\end{enumerate}
\end{theo}

Soit maintenant $[\La,n,0,\b]$ une strate simple de $\A$.  
Fixons des homo\-morphismes~:
\begin{equation*}
\iota_{p,\qlb}:\mu_{p^\infty}(\CC)\to\overline{\QQ}{}^{\times}_\ell,
\quad
\iota_{p,\flb}:\mu_{p^\infty}(\CC)\to\overline{\FF}{}^{\times}_\ell
\end{equation*}
comme au paragraphe \ref{CarSim}
et supposons que le second est le composé du premier avec le morphis\-me 
de réduction mod $\ell$. 
Posons $\J=\J(\b,\La)$, $\J^1=\J^1(\b,\La)$ et $\H^1=\H^1(\b,\La)$. 
Comme $\H^1$ est un pro-$p$-groupe, la réduction mod $\ell$ définit 
une bijection~:
\begin{equation*}
\Cc_{\qlb}(\La,m,\b)\to\Cc_{\flb}(\La,m,\b)
\end{equation*}
pour tout $m$ tel que $0\<m\<-k_0(\b,\La)-1$. 
Appelons \textit{relèvement} d'un $\flb$-ca\-rac\-tère sim\-ple 
$\t$ son image réciproque par cette bijection. 

\begin{prop}
\label{RedBext}
Soit $\t\in\Cc_{\flb}(\La,0,\b)$ un $\flb$-caractère simple, 
soit $\tilde\t$ le relèvement de $\t$
et soit $\tilde\k$ une $\b$-extension de $\tilde\t$.
\begin{enumerate}
\item 
Si $\mathfrak{r}$ est une structure entière de $\tilde\k$,
alors $\mathfrak{r}\otimes\flb$ est une $\b$-ex\-ten\-sion de $\t$.
\item
Si $\k$ est une $\b$-ex\-ten\-sion de $\t$, il y a une 
$\b$-ex\-ten\-sion de $\tilde\t$ dont la réduction mod $\ell$ est 
$\k$. 
\end{enumerate}
\end{prop}

\begin{proof}
La preuve est analogue à celle donnée dans \cite[III.4.18]{Vig1}.
Notons $\k$ la repré\-sentation $\mathfrak{r}\otimes\flb$ de $\J$. 
La restriction de $\tilde\k$ à $\H^1$ est un 
multiple de $\tilde\t$, donc la restriction de 
$\k$ à $\H^1$ est un multiple de $\t$. 
Ainsi la restriction de $\k$ au pro-$p$-groupe $\J^1$ contient 
la représentation de Hei\-senberg $\n$, 
et a la même dimension. 
Enfin, d'après \cite[Lemme I.9.8]{Vig1}, l'ensemble d'entrela\-ce\-ment de 
$\k$ dans $\G$ est égal à 
$\J\mult\B\J$. 
C'est donc une $\b$-extension de $\t$.

Soit maintenant $\k'$ une $\b$-extension de $\t$. 
D'après \eqref{StHonore}, il existe un $\flb$-ca\-rac\-tère $\chi$ 
de $\mult\kk_\E$ tel que $\k$ soit isomorphe à 
$\k'\otimes(\chi\circ\N_{\B/\E})$, 
\ie que $\r_{\ell}(\tilde\k)$ est égale à $\k'\otimes(\chi\circ\N_{\B/\E})$.
Si l'on note $\a$ l'unique $\qlb$-caractère de $\mult\kk_\E$ 
d'ordre premier à $\ell$ dont la réduction mod $\ell$ soit $\chi^{-1}$, 
alors l'image de $\tilde\k'\otimes({\a}\circ\N_{\B/\E})$ par $\r_{\ell}$ est
égale à $\k'$. 
\end{proof}

\begin{rema}
Deux $\b$-extensions de $\tilde\t$ ont des réductions mod $\ell$ 
isomorphes si et seu\-le\-ment si elles sont tordues l'une de l'autre par un
$\qlb$-ca\-rac\-tère d'ordre une puissance de $\ell$.
\end{rema}

On réduit maintenant les types simples modulo $\ell$.

\begin{prop}
\label{RedTypSim}
Soit $\t\in \Cc_{\flb}(\La,0,\b)$ un $\flb$-caractère simple 
et soit $\tilde\t$ son relèvement. 
\begin{enumerate}
\item 
Si $\ll$ est une structure entière d'un $\qlb$-type simple contenant 
$\tilde\t$, alors $\ll\otimes\flb$ est un $\flb$-type simple contenant $\t$. 
\item
Pour tout $\flb$-type simple $\l$ contenant $\t$, il existe un $\qlb$-type simple 
contenant $\tilde\t$ dont la réduction mod $\ell$ est $\l$. 
\end{enumerate}
\end{prop}

\begin{proof}
Traitons d'abord le cas de niveau $0$.
En considérant un type simple de niveau $0$ comme 
une représen\-ta\-tion du quotient $\U(\La)/\U_{1}(\La)$, 
on se ramène au problème de la réduction 
des repré\-sen\-ta\-tions irréductibles cuspidales 
du groupe fini $\GL_{s}(\kk_{\D})$ avec $s\>1$. 

Supposons maintenant que le niveau est non nul. 
On fixe un $\qlb$-type simple $\tilde\l$ contenant $\tilde\t$, 
qu'on écrit $\tilde\k\otimes\tilde\s$. 
Si $\mathfrak{k}$ et $\ss$ sont des structures entières de $\tilde\k$ et $\tilde\s$
respectivement, $\mathfrak{k}\otimes\ss$ est une structure entiè\-re 
de $\tilde\l$.
Le résultat est alors 
une conséquence de la propositions \ref{RedBext} et du cas de niveau $0$. 
\end{proof}

La proposition \ref{RedTypSim} montre que l'existence des 
$\flb$-types simples de $\G$ 
se déduit de l'existence des $\qlb$-types simples 
prouvée dans \cite{VS2,SeSt2}. 

\section{Représentations cuspidales}
\label{RCIC}

Soit $m\>1$ un entier et soit $\G=\GL_{m}(\D)$.
Dans cette section, on effectue la classifica\-tion des représentations 
ir\-ré\-duc\-ti\-bles cuspidales de $\G$ en termes de types simples 
maximaux (théorème \ref{TheZerPos}).
Ceci a été fait par Bushnell-Kutz\-ko \cite{BK} et Sécherre-Stevens \cite{SeSt1} 
dans le cas complexe, et par Vignéras \cite{Vig1} 
pour $\GL_{n}(\F)$ dans le cas modulaire. 
Cette classification permet d'associer certains invariants à une 
représentation ir\-ré\-duc\-ti\-ble cuspidale de $\G$
(paragraphes \ref{InvRho} et \ref{NURHO}).

On étu\-die ensuite les problèmes de la réduction et du 
relèvement des re\-pré\-sentations ir\-ré\-duc\-ti\-bles cuspidales de $\G$. 
On montre que, contrairement au cas déployé traité dans \cite{Vig1}, 
une $\qlb$-re\-pré\-sen\-ta\-tion ir\-ré\-duc\-ti\-ble cuspidale entière 
de $\G$ ne se réduit pas toujours en une représen\-tation 
ir\-ré\-duc\-ti\-ble (théorème \ref{RedCusp}) et qu'une 
$\flb$-re\-pré\-sen\-ta\-tion ir\-ré\-duc\-ti\-ble cus\-pi\-dale de $\G$ 
n'admet pas toujours un relèvement (paragraphe \ref{Neantise}).

Néanmoins, on prouve que toute $\flb$-re\-pré\-sen\-ta\-tion ir\-ré\-duc\-ti\-ble 
supercuspidale de $\G$ contenant un type simple maximal supercuspidal 
se relè\-ve à $\qlb$ (théorème \ref{RelSuperCusp}, remarque 
\ref{RemaRelSuperCusp})
et on donne dans le cas déployé une nouvelle preuve du fait (dû à 
Vignéras) selon lequel la réduction modulo $\ell$ d'une $\qlb$-représentation 
ir\-ré\-duc\-ti\-ble cuspidale entière est irréductible (corollaire 
\ref{MieuxQueV}, remarque \ref{MieuxQueV2}).  

\subsection{Types simples maximaux}
\label{IndCom}

Soit $(\J,\l)$ un type simple maximal de $\G$.
Notons~:
\begin{equation*}
\NJ=\N_{\G}(\J,\l)
\end{equation*} 
le nor\-ma\-li\-sateur de la classe 
d'iso\-mor\-phisme de $\l$ dans $\G$. 
On peut supposer que $(\J,\l)$ est défini par rapport à une strate simple 
$[\La,n,0,\b]$ où $\La$ est une $\Oo_\D$-suite de réseaux stricte de $\D^m$.
Dans ce cas, il est d'usage dans les notations de remplacer cette suite par 
l'ordre héréditaire $\AA=\AA(\La)$ qu'elle définit.  

\begin{prop}
\label{MartinDuGard}
Pour toute représentation $\nl$ prolongeant $\l$ à $\NJ$, 
l'induite $\cind_{\NJ}^{\G}(\nl)$ est une représentation 
ir\-ré\-duc\-ti\-ble cuspidale de $\G$. 
En outre, l'ap\-pli\-ca\-tion~:
\begin{equation}
\label{BijLaCusp}
\nl\mapsto\cind_{\NJ}^{\G}(\nl)
\end{equation}
induit une bijection entre prolongements de $\l$ à $\NJ$ et classes 
d'isomorphisme de repré\-sen\-ta\-tions ir\-ré\-duc\-ti\-bles 
cuspidales de $\G$ contenant $\l$. 
\end{prop}

\begin{proof}
La preuve s'inspire de \cite[\S8]{Vig3} mais des modifications doivent y être 
appor\-tées.  
Soit $\nl$ une re\-pré\-sentation de $\NJ$ pro\-lon\-geant $\l$. 
Comme l'entrelacement de $\nl$ dans $\G$ est $\NJ$ 
(voir la re\-mar\-que \ref{Etoc2}), 
la $\R$-algèbre des endomorphismes de l'induite compacte 
$\cind^{\G}_{\NJ}(\nl)$ est isomorphe à $\R$.
Pour prouver qu'elle est irréductible, nous allons 
appliquer \cite[Lemma 4.2]{Vig3}.

Soit $[\AA,n,0,\b]$ une strate simple de $\A$ définissant $(\J,\l)$.  
On pose $\U=\U(\AA)\cap\mult\B$, qui est un sous-groupe compact 
maximal de $\mult\B$.

\begin{lemm}
\label{LemCritV}
La représentation de $\NJ$ sur la composante $\n$-isotypique de
$\cind_{\NJ}^{\G}(\nl)$ est isomorphe à la somme directe des $\nl^{n}$, 
avec $n\in\N_{\mult\B}(\U)\J/\NJ$, où $\N_{\mult\B}(\U)$ est le 
normalisateur de $\U$ dans $\mult\B$. 
\end{lemm}

\begin{rema}
Dans le cas où $\D=\F$, on a $\NJ=\N_{\mult\B}(\U)\J$, 
de sorte que la repré\-sen\-tation de $\NJ$ sur la 
composante $\n$-isotypique de $\cind_{\NJ}^{\G}(\nl)$ 
est isomorphe à $\nl$~: voir \cite[Corollary 8.4]{Vig3}.
\end{rema}

\begin{proof}
La preuve est analogue à celle de la proposition \ref{Empire}. 
On trouve un isomorphisme~:
\begin{equation*}
\Hom_{\J^1}(\n,\cind_{\NJ}^{\G}(\nl)) \simeq
\bigoplus\limits_{g\in\N_{\mult\B}(\U)\J/\NJ} \Hom_{\J^1}(\n,\nl^g)
\end{equation*}
et le ré\-sul\-tat s'ensuit. 
\end{proof}

Pour appliquer le critère d'irréductibilité \cite[Lemma 4.2]{Vig3}, 
il faut montrer que, pour tout quo\-tient irréductible $\pi$ 
de $\cind_{\NJ}^{\G}(\nl)$, la représentation $\nl$ est un 
quotient de la restriction de $\pi$ à $\NJ$. 
Soit $\pi$ un tel quo\-tient irréductible~; par réciprocité de Frobenius, 
$\nl$ est donc une sous-représentation de la restriction de $\pi$ à $\NJ$.
Comme $\J^1$ est un pro-$p$-groupe, la composante 
$\n$-isotypique de $\pi$ est un quotient de celle de $\cind_{\NJ}^\G(\nl)$ 
com\-me représentation de $\NJ$. 
Mais c'est aussi un facteur direct de la restriction de $\pi$ à $\J^1$
(qui est semi-simple), de sorte que $\pi$ a la propriété attendue.

Enfin, $\cind_{\NJ}^{\G}(\nl)$ est cuspidale, 
puisque ses coefficients sont à support compact modulo le centre de $\G$
(voir \cite[II.2.7]{Vig1}).  

Il ne reste qu'à prouver que l'application \eqref{BijLaCusp} est bijective. 
Notons $\Z$ le centre de $\G$ et fixons un caractère $\om:\Z\to\mult\R$ 
dont la restriction à $\J\cap\Z$ coïncide avec le caractère par lequel 
agit la restriction de  $\l$ à ${\J\cap\Z}$. 
Notons $\l_\om$ la représentation de $\J\Z$ prolongeant $\l$ et dont 
la restriction à $\Z$ agit par le caractère $\om$.
On va prouver que \eqref{BijLaCusp} induit une bijection entre 
prolongements de $\l_\om$ à $\NJ$ et classes 
d'isomorphisme de repré\-sen\-ta\-tions ir\-ré\-duc\-ti\-bles 
cuspidales de $\G$ contenant $\l_\om$. 

Par un argument similaire à celui de 
\cite[III.4.27]{Vig1}, la représentation $\l_\om$ 
admet un prolongement $\nl$ à $\NJ$ et l'application 
$\chi\mapsto\nl\chi$ est une bijection entre les caractères de $\NJ$ 
triviaux sur $\J\Z$ et les pro\-lon\-ge\-ments de $\l_\om$ à $\NJ$.
On en déduit que \eqref{BijLaCusp} est injective. 
Si maintenant $\rho$ est une repré\-sen\-ta\-tion ir\-ré\-duc\-ti\-ble 
cuspidale de $\G$ contenant $\l_\om$, alors $\rho$ est un quotient de 
$\ind^\G_{\NJ}(\nl\otimes\R[\NJ/\J\Z])$ où $\R[\NJ/\J\Z]$ est la 
représentation régulière du groupe fini cyclique $\NJ/\J\Z$.
Il existe donc un caractère $\chi$ de $\NJ$ trivial sur $\J\Z$ tel que $\rho$ soit 
un quotient de (et donc soit isomorphe à) l'induite $\ind^\G_{\NJ}(\nl\chi)$.
Ceci met fin à la démonstration de la proposition \ref{MartinDuGard}.
\end{proof}

Un couple de la forme $(\NJ,\nl)$ produit à partir d'un type simple maximal 
$(\J,\l)$ de $\G$ est appelé un \textit{type simple maximal étendu} de $\G$.  

\subsection{Représentations cuspidales de niveau $0$}

Rappelons qu'une représentation ir\-ré\-duc\-ti\-ble de $\G$ 
est de niveau $0$ s'il y a un $\Oo_\F$-ordre héré\-ditaire $\AA$ 
de $\A=\Mat_m(\D)$ tel qu'elle possède un vecteur non nul 
invariant par $\U^1(\AA)=1+\PP$, où $\PP$ est le radical de $\AA$. 

\begin{theo}
\label{TheZer}
Toute représentation ir\-ré\-duc\-ti\-ble cuspidale de 
niveau $0$ de $\G$ contient un type simple maximal de 
niveau $0$. 
\end{theo}

\begin{proof}
Soit $\rho$ une re\-pré\-sen\-ta\-tion ir\-ré\-duc\-ti\-ble 
cuspidale de niveau $0$ de $\G$, et soit $\AA$ un ordre 
hé\-ré\-di\-tai\-re mi\-ni\-mal parmi ceux pour lesquels 
$\rho$ a des vecteurs invariants par $\U^1=\U^1(\AA)$, 
que l'on peut supposer standard. 
On pose $\U=\U(\AA)$.
On a un iso\-mor\-phisme de groupes~:
\begin{equation}
\label{ParisTexas}
\U/\U^{1}\to\GL_{m_1}(\kk_\D)\times\cdots\times\GL_{m_r}(\kk_\D),
\end{equation}
où $r$ est la période de $\AA$ et les $m_i$ sont des entiers dont 
la somme est égale à $m$. 
De cette fa\c{c}on, le groupe de Galois $\Gal(\kk_{\D}/\kk_{\F})$
opère sur l'ensemble des classes d'iso\-mor\-phis\-me de représentations 
de $\U/\U^{1}$.

\begin{lemm}
\label{LemSig}
Il existe une représentation irréductible $\s$ de $\U$ triviale sur 
$\U^1$ telle qu'on ait 
$\Hom_\U(\s,\rho)\neq0$, et qui est cuspidale en tant que 
représentation de $\U/\U^1$.
\end{lemm}

\begin{proof}
L'existence d'une représentation irréductible $\s$ de $\U$ triviale sur 
$\U^1$ telle que $\Hom_\U(\s,\rho)\neq0$ est donnée par le corollaire 
\ref{LAmourLApresMidi}. 
Supposons qu'une telle représentation $\s$ n'est pas cuspidale comme 
re\-pré\-sen\-ta\-tion de $\U/\U^1$.
Alors la proposition \ref{isc} contredit la minimalité de $\AA$. 
\end{proof}

Fixons une représentation $\s$ de $\U$ satisfaisant aux 
conditions du lemme \ref{LemSig}.  
Le couple $(\U,\s)$ est un type semi-simple homogène 
(de niveau $0$) de $\G$. 
On a le résultat général suivant. 

\begin{prop}
\label{MaxZerPosSS}
Soit $\pi$ une représentation ir\-ré\-duc\-ti\-ble cuspidale de $\G$.  
Si $\pi$ contient un type semi-simple, alors ce type semi-simple est 
un type simple maximal. 
\end{prop}

\begin{proof}
D'après la remarque \ref{SolecismeEvident}, un tel type semi-simple 
ne peut pas être une paire couvrante relativement à un sous-groupe 
de Levi propre de $\G$.
Par conséquent, c'est une paire couvrante d'un 
type simple maxi\-mal $(\J_\M,\l_\M)$ avec $\M=\G$, \ie que 
c'est un type simple maximal.  
\end{proof}

D'après la proposition \ref{MaxZerPosSS}, 
puisque le type semi-simple $(\U,\s)$ apparaît dans une 
représentation cuspidale, c'est un type simple maximal, 
ce qui met fin à la dé\-mons\-tra\-tion du théorème \ref{TheZer}. 
\end{proof} 

\subsection{Représentations cuspidales de niveau non nul}

L'objet de ce paragraphe est de démontrer le résultat suivant.

\begin{theo}
\label{ThePos}
Toute représentation ir\-ré\-duc\-ti\-ble cuspidale de niveau non nul 
de $\G$ contient un type simple maximal de niveau non nul. 
\end{theo}

Soit $\rho$ une représentation ir\-ré\-duc\-ti\-ble cuspidale 
de niveau non nul de $\G$. 
Dans un premier temps, il s'agit de montrer que $\rho$ contient un 
caractère simple.
Fixons comme au paragraphe \ref{CarSim} un homomorphisme $\iota_{p,\R}$ 
et un caractère $\psi_{\F,\CC}:\F\to\mult\CC$. 
Considérons une strate $[\AA,n,n-1,\b]$ de $\A$
(où $\AA$ est un $\Oo_\F$-ordre héréditaire de $\A$).
Il lui correspond le caractère~:
\begin{equation}
\label{DEFPSIB}
\psi_{\b}:x\mapsto\iota_{p,\R}\circ\psi_{\F,\CC}\circ\tr_{\A/\F}(\b(x-1))
\end{equation}
du sous-groupe ouvert compact $\U^n(\AA)=1+\PP^n$, où $\PP$ est le 
radical de $\AA$, et où $\tr_{\A/\F}$ désigne la trace réduite de $\A$ 
sur $\F$.

\begin{lemm}
\label{LemPSIB}
Il existe une strate simple $[\AA,n,n-1,\b]$ de $\A$ telle que la 
restriction de $\rho$ à $\U^{n}(\AA)$ contienne $\psi_\b$.
\end{lemm}

\begin{proof}
La démonstration de Broussous \cite{Br4} dans le cas complexe, 
elle-même ins\-pi\-rée de celle de Bushnell et Kutzko \cite{BK2} 
concernant $\GL_{n}(\F)$, s'adapte ici.
Rappelons-en les principales étapes.
\begin{enumerate}
\item
D'abord, on montre que toute représentation de $\G$ de 
niveau non nul contient une strate fondamentale au sens de 
\cite[Définition 3.9]{SeSt1}. 
\item
Ensuite, on montre qu'une représentation de $\G$
de niveau non nul contenant une strate fondamentale scindée
(au sens de \cite[Définition 3.9]{SeSt1})
a un module de 
Jacquet non nul relativement à un sous-groupe parabolique 
propre de $\G$.
\item
Enfin, on montre que toute représentation de $\G$
de niveau non nul contenant une strate fondamentale non scindée
de $\A$ contient également une strate simple de $\A$.
\end{enumerate}

L'étape 1 (voir par exemple \cite{HM}) repose sur 
\cite[Proposition 1.2.2]{Br4} qui ne dépend pas du corps $\R$.
L'étape 2 repose principalement sur les propositions 2.3.2 
et 2.4.3 de \cite{Br4} (qui cor\-res\-pon\-dent res\-pec\-ti\-vement au théorème 
3.7 et au lemme 3.9 de \cite{BK2})~; 
la première est indépendante de $\R$, 
et la preuve de la seconde reste valable dans le cas modulaire --
notamment l'argument ``inversi\-ble à gauche implique inversible'' de 
\cite[7.15]{BK1} car, contrairement à se qui se passe dans la preuve de la 
proposition \ref{PCTSM}, le groupe sous-jacent est ici un pro-$p$-groupe. 
Enfin l'étape 3 repose sur \cite[Théorème 1.2.5]{Br4}, qui ne dépend pas 
du corps $\R$.
\end{proof}

\begin{lemm}
\label{PotAuFeu}
Il existe une strate simple $[\AA,n,0,\b]$ de $\A$ et un caractère 
simple $\t\in\Cc(\AA,0,\b)$ tels que la restriction de $\rho$ à 
$\H^1(\b,\AA)$ contienne $\t$.
\end{lemm}

\begin{proof}
La démonstration de \cite{SeSt1} dans le cas complexe s'adapte ici. 
Il faut prouver que, si le résultat n'est pas vrai, $\rho$ contient
ou bien une strate scindée, ou bien un caractère scin\-dé au sens de 
la définition 3.22 de \cite{SeSt1}, et que
dans chacun de ces deux cas, 
$\rho$ a un module de Jacquet non nul re\-la\-ti\-ve\-ment 
à un sous-groupe parabolique propre. 
La première étape repose sur \cite[Théorème 3.23]{SeSt1}, dont la 
preuve ne dépend pas de $\R$.
La seconde étape repose principalement sur le théorème 4.3 et le 
corollaire 4.6 de \cite{SeSt1}
(ce dernier cor\-res\-pon\-dant à \cite[Corollary 6.6]{BK2})~; 
la preuve du théorème ne dépend pas du 
corps $\R$ et celle du corollaire reste valable dans le cas modulaire.
\end{proof}

On fixe une strate simple $[\AA,n,0,\b]$ de $\A$ et un caractère 
simple $\t\in\Cc(\AA,0,\b)$ sa\-tisfai\-sant à la condition du lemme 
\ref{PotAuFeu}, en choisissant $\AA$ minimal pour cette propriété. 
On pose $\J=\J(\b,\AA)$ et $\J^1=\J^1(\b,\AA)$ et on fixe une $\b$-extension 
$\k$ de $\t$. 
D'après le lemme \ref{LAmourLApresMidi}, la représentation $\rho$ contient 
une sous-représentation de la forme $\k\otimes\s$, 
avec $\s$ une re\-pré\-sentation irréductible de $\J$ triviale sur $\J^1$.

\begin{lemm}
\label{LemSigJi}
En tant que représentation de $\J/\J^1\simeq\GL_s(\ff_{\D'})^r$, 
la représentation $\s$ est cus\-pi\-dale. 
\end{lemm}

\begin{proof}
Si elle ne l'était pas, la proposition \ref{isc} contredirait la minimalité de 
$\AA$.  
\end{proof}

D'après la propriété \eqref{poumi}, et avec les notations du paragraphe 
\ref{EndoEq}, 
la représentation $\k\otimes\s$ est 
l'induite du type semi-simple homogène $\bk\otimes\bs$.
D'après la proposition \ref{MaxZerPosSS}, ce type semi-simple est un type 
simple maximal. 
Ceci met fin à la dé\-mons\-tra\-tion du théorème \ref{ThePos}. 

\subsection{Invariants associés à une représentation cuspidale}
\label{InvRho}

Le théorème suivant subsume les théo\-rèmes \ref{TheZer} et 
\ref{ThePos}, et il les com\-plè\-te en fournissant une classification des 
classes d'isomorphisme de repré\-sen\-tations 
irré\-duc\-tibles cuspidales de $\G$ par la théorie des 
types simples. 

\begin{theo}
\label{TheZerPos}
L'application~:
\begin{equation}
\label{ThClasCusp}
(\NJ,\nl)\mapsto\cind_{\NJ}^{\G}(\nl)
\end{equation}
induit une bijection entre classes de $\G$-conjugaison de types 
simples maximaux étendus 
et classes d'isomorphisme de représentations 
irréductibles cuspidales de $\G$.
\end{theo}

\begin{proof}
D'après la proposition \ref{MartinDuGard}, cette application est 
bien définie, et elle est surjec\-tive d'après les théorèmes \ref{TheZer} 
et \ref{ThePos}.  
Pour prouver qu'elle est injective, on reprend l'argument de 
\cite[Theorem 7.2]{SeSt2}. 
\end{proof}

Soit $\rho$ une représentation ir\-ré\-duc\-ti\-ble cus\-pi\-da\-le de 
$\G$. 
Nous voulons associer à $\rho$ des invariants. 
Fixons un type simple maximal $(\J,\l)$ de $\G$ contenu dans $\rho$ 
ainsi qu'une strate simple $[\AA,n,0,\b]$ le dé\-fi\-nis\-sant.
Posons $\E=\F(\b)$ et écrivons $\l$ sous la forme $\k\otimes\s$.
Fixons un isomorphisme de $\E$-algèbres \eqref{MonEpousee} et 
notons $\Ga$ le groupe de Galois défini par (\ref{Bliblij}). 
Notons $(\NJ,\nl)$ le type simple maximal étendu prolongeant 
$(\J,\l)$ et contenu dans $\rho$.

\begin{defi}
On note $\ell$ l'exposant caractéristique de $\R$, égal à $1$ si 
$\R$ est de carac\-té\-ristique nulle et égal à la carac\-téristique de $\R$ sinon. 
\end{defi}

Si $\chi$ est un $\R$-caractère non ramifié de $\G$, alors $\rho \chi$ 
est isomorphe à $\rho$ si et seulement si $\nl\chi=\nl$, \ie si et seulement 
si $\chi$ est trivial sur $\NJ$.
D'après la remarque \ref{Etoc2}, c'est encore équivalent à $\chi(\w_\l)=1$, 
où l'élément $\w_\l\in\G$ est défini par \eqref{pilam}.
Comme un caractère non ramifié de $\G$ est caractérisé par sa valeur en 
n'importe quel élément dont la norme réduite est une uniformisante de $\F$, 
le groupe des $\R$-caractères non ramifiés $\chi$ de $\G$ tels que $\rho\chi$ soit 
isomorphe à $\rho$ est fini et cyclique.
Son cardinal, noté~:
\begin{equation*}
n(\rho), 
\end{equation*}
est le plus grand diviseur premier à $\ell$ de la valuation de la norme 
réduite de $\w_\l$ 
(dans le cas où $\R$ est de ca\-rac\-té\-ris\-tique nulle,
voir \cite[Proposition 4.1]{VSU0}). 
Rappelons que $d$ est le degré réduit de $\D$ sur $\F$ et notons~:
\begin{equation*}
\label{InvF}
f(\rho)
\end{equation*}
le quotient de $md$ par l'indice de ramification de $\E$ sur $\F$.  
Notons ensuite~:
\begin{equation*}
\label{InvB}
s(\rho)
\end{equation*}
l'ordre du stabilisateur de $\s$ dans $\Ga$ et~:
\begin{equation*}
b(\rho)
\end{equation*}
le cardinal de l'orbite de $\s$ sous $\Ga$, qui avait été 
noté $b(\l)$ au para\-gra\-phe \ref{Paralipomenes}. 

Remarquons que $s(\rho)$ est l'indice de $\mult\E\J$ dans le 
$\G$-normalisateur de $\l$ et que $s(\rho)b(\rho)=d'$, 
le degré réduit de $\D'$ sur $\E$.
L'entier $f(\rho)$ s'écrit aussi~:
\begin{equation}
\label{Dahl}
f(\rho) = em'd'
\end{equation}
où $e$ est l'indice de ramification de $\E$ sur $\F$, 
ce qui implique que $s(\rho)$ divise $f(\rho)$ et que le quotient de 
$f(\rho)$ par $s(\rho)$ est la valuation de la norme réduite de $\w_\l$.  
On en déduit que~: 
\begin{equation}
\label{Wigmore}
f(\rho) = n(\rho)s(\rho)\ell^{u},
\quad u\>0.
\end{equation}

Rappelons que la classe d'inertie de $\rho$ est l'ensemble des classes 
d'isomorphisme de la forme $\sy{\rho\chi}$ où $\chi$ décrit l'ensemble des caractères 
non ramifiés de $\G$.

\begin{prop}
Les quantités $n(\rho)$, $f(\rho)$, $b(\rho)$, $s(\rho)$ 
ne dépendent que de la 
classe d'iner\-tie de $\rho$, et pas du type simple maximal 
$(\J,\l)$ ni de la strate $[\AA,n,0,\b]$.
\end{prop}

\begin{proof}
Par définition, $n(\rho)$ ne dépend que de la classe d'iner\-tie 
de $\rho$.
Ensuite, fixons un type simple maximal $(\J',\l')$ contenu dans $\rho$, 
une strate simple $[\AA',n',0,\b']$ le définissant et écrivons $\l'$ 
sous la forme $\k'\otimes\s'$.
D'après le théorème \ref{TheZerPos}, le type simple $(\J',\l')$ est conjugué à 
$(\J,\l)$ sous $\G$. 
On peut donc supposer que $(\J',\l')$ est égal à $(\J,\l)$.
D'après \cite[Theorem 9.4]{BSS1}, l'indice de ramification et le degré résiduel de $\E$ 
sur $\F$ ne dépendent pas du choix de la strate simple 
$[\AA,n,0,\b]$, \ie qu'ils 
sont respectivement égaux à l'indice de ramification et au degré résiduel 
de $\F(\b')$ sur $\F$. 
L'entier $f(\rho)$ ne dépend donc 
pas des choix effectués, non plus que le degré 
réduit de $\D'$ sur $\E$.
Il reste donc à prouver que $b(\rho)$ ne dépend pas du choix de 
$[\AA,n,0,\b]$ ni de $\s$. 
Si l'on considère $\s$ comme une représentation de $\J/\J^1$, alors
$b(\rho)$ est égal à l'indice de $\N_{\G}(\l)$ dans $\N_{\G}(\J)$, 
donc $b(\rho)$ et $s(\rho)$ ne dépendent 
que de la classe d'iner\-tie de $\rho$.
\end{proof}

\subsection{Réduction d'une représentation cuspidale entière} 
\label{LeMepris}
\label{NoLogo0}

Soit $\ell$ un nombre premier différent de $p$. 
Soit $\tilde\rho$ une $\qlb$-représentation irréductible 
cuspidale de $\G$.
Fixons un type simple maximal étendu
$(\tilde\NJ,\tilde\nl)$ de $\G$ contenu 
dans $\tilde\rho$, 
\ie que $\tilde\rho$ est isomorphe à l'induite compacte de $\tilde\nl$ 
à $\G$.  

\begin{lemm}
La $\qlb$-représentation $\tilde\rho$ est entière si et seulement si $\tilde\nl$ 
est entière. 
\end{lemm}

\begin{proof}
Si $\tilde\nl$ est entière, alors son induite compacte à $\G$ l'est également 
(paragraphe \ref{SectionPreliminaire}.\ref{focales}), donc $\tilde\rho$ est 
entière. 
Inversement, si $\tilde\rho$ est entière, alors sa restriction à $\tilde\NJ$ l'est 
également, ainsi que $\tilde\nl$ qui en est un fac\-teur direct (voir le lemme
\ref{LemCritV}).
\end{proof}

Supposons dorénavant que $\tilde\rho$ est entière.  
Soit $\J$ le sous-groupe compact maximal de $\tilde\NJ$,
et soit $\tilde\l$ la restriction de $\tilde\nl$ à $\J$.
Alors $(\J,\tilde\l)$ est un type simple maximal de $\G$.
Si $\ll$ est une structure entière de $\tilde\nl$, 
c'est aussi une structure entière de $\tilde\l$ et 
la représen\-tation $\l=\ll\otimes\flb$ 
est un $\flb$-type simple ma\-xi\-mal de $\G$ d'après la 
pro\-po\-si\-tion \ref{RedTypSim}.
Soit $\NJ$ le nor\-ma\-li\-sateur de $\l$ dans $\G$.
Le groupe $\tilde\NJ$ est contenu dans ${\NJ}$ et
(re\-mar\-que \ref{Etoc2}) il est d'in\-di\-ce fini dans ${\NJ}$. 
Posons~:
\begin{equation*}
\label{DefInvA}
a=a_\ell(\tilde\rho)=({\NJ}:\tilde\NJ).
\end{equation*}
Notons $\nu$ le caractère non ramifié de $\G$ obtenu en 
composant la norme réduite de $\G$ dans $\mult\F$, la va\-lua\-tion de 
$\mult\F$ dans $\ZZ$ (normalisée en envoyant une uniformisante sur $1$) et 
l'unique morphisme de $\ZZ$ dans $\overline{\mathbb{F}}{}^{\times}_{\ell}$ 
envoyant $1$ sur l'inverse de $q$ mod $\ell$.

\begin{theo}
\label{RedCusp}
Il y a une $\flb$-repré\-sen\-ta\-tion irréductible cuspidale 
${\rho}$ de $\G$ telle que~:
\begin{equation}
\label{Herodote56}
\r_{\ell}(\tilde\rho)=
\sy{{\rho}}+\sy{{\rho}\nu}+\cdots+\sy{{\rho}\nu^{a-1}}
\end{equation}
dans le groupe de Grothendieck de $\G$. 
\end{theo}

\begin{proof}
On note $\nl^{\flat}$ la représentation de $\tilde\NJ$ sur 
$\ll\otimes\flb$. 
C'est un prolongement de $\l$ à $\tilde\NJ$, que l'on peut 
prolonger à ${\NJ}$ d'après le lemme suivant. 

\begin{lemm}
\label{RedCuspIslaNeraAvant}
Il existe une $\flb$-représentation $\nl$ de ${\NJ}$ 
prolongeant $\nl^{\flat}$.
\end{lemm}

\begin{proof}
\label{RedTypSimExt}
On choisit une $\flb$-représentation $\nl_0$ de ${\NJ}$ 
prolongeant $\l$. 
Les $\flb$-repré\-sen\-ta\-tions de ${\NJ}$ prolongeant $\l$
sont de la forme $\nl_0\chi$, où $\chi$ est un $\flb$-caractère 
de ${\NJ}$ trivial sur $\J$.
Soit $\chi'$ un $\flb$-caractère 
de $\tilde\NJ$ trivial sur $\J$ tel que 
la restriction de $\nl_0$ à $\tilde\NJ$ soit isomorphe à 
$\nl^{\flat}\chi'$.
Puisque le groupe ${\NJ}/\J$ est isomorphe à $\ZZ$,
il existe un caractère $\chi$ de ${\NJ}$ prolongeant $\chi'$. 
Alors $\nl=\nl_0\chi$ est une 
représentation de ${\NJ}$ prolongeant $\nl^{\flat}$.
\end{proof}

On choisit une $\flb$-représentation ${\nl}$ de ${\NJ}$
prolongeant $\nl^{\flat}$. 
D'après la proposition \ref{MartinDuGard}, l'induite compacte 
de $\nl$ à $\G$, notée $\rho$, 
est une $\flb$-représentation irréductible cuspidale.
La réduction mod $\ell$ commutant à l'induction compacte, 
$\r_\ell(\tilde\rho)$ est la semi-simplifiée de l'induite com\-pac\-te de 
$\nl^\flat$ à $\G$.
Écrivons~:
\begin{equation*}
\label{CharnyMax}
\ind^{{{\NJ}}}_{\tilde\NJ}(\nl^\flat)=
{\nl}\otimes\flb[{{\NJ}}/\tilde\NJ]
\end{equation*}
où $\flb[{{\NJ}}/\tilde\NJ]$ désigne la représentation régulière du groupe
cyclique ${{\NJ}}/\tilde\NJ$, \ie l'indui\-te à ${{\NJ}}$ du 
$\flb$-caractère trivial de $\tilde\NJ$. 
La semi-simplifiée de cette dernière est égale à~:
\begin{equation*}
\sy{1}+\sy{\a}+\cdots+\sy{\a^{a-1}}
\end{equation*}
où $\a$ est un générateur quelconque du groupe des $\flb$-caractères de 
${\NJ}/\tilde\NJ$, ce groupe étant cyclique et d'ordre égal au plus grand 
diviseur de $a$ premier à $\ell$.
On en déduit que~:
\begin{equation*}
\r_{\ell}\big(\ind^{{{\NJ}}}_{\tilde\NJ}(\tilde\nl)\big)=
\sy{{\nl}}+\sy{{\nl}\a}+\cdots+\sy{{\nl}\a^{a-1}}.
\end{equation*}
En induisant à $\G$, on trouve que~: 
\begin{equation*}
\r_\ell(\tilde{\rho})=\sy{{\rho}}+\sy{{\rho}\chi}+\cdots+\sy{{\rho}\chi^{a-1}}
\end{equation*}
où $\chi$ est n'importe quel caractère non ramifié de $\G$ prolongeant $\a$.  

Si $a$ est égal à $1$, il n'y a rien d'autre à prouver. 
Sinon, pour terminer la preuve du théorème \ref{RedCusp}, 
il suffit de prouver que la restriction de $\nu$
à $\NJ$ est un générateur du groupe des $\flb$-caractères de 
${\NJ}/\tilde\NJ$.  
Le groupe $\NJ$ étant engendré par $\J$ et l'élément $\w_\l$ défini par 
\eqref{pilam}, il suffit de prouver que l'ordre de $\nu(\w_\l)\in\mult\R$ est 
égal au plus grand diviseur de $a$ premier à $\ell$.
Un calcul simple montre que la valuation de la norme réduite de $\w_\l\in\G$ 
est égale à $f(\rho)s(\rho)^{-1}$. 

\begin{defi}
Pour tout entier $k\>1$ premier à $\ell$, 
notons $\el(k)$ l'ordre de $k$ dans $\mathbb{F}_{\ell}^\times$.
\end{defi}

Compte tenu de \eqref{Wigmore},
on en déduit que la restriction de $\nu$ à $\NJ$ est d'ordre~:
\begin{equation*}
\el(q^{n(\rho)}).
\end{equation*}
Il s'agit donc maintenant de prouver que $a$ est égal à $1$,
ou est de la forme $\el(q^{n(\rho)})\ell^u$ pour un entier $u\>0$.

Reprenons les notations du paragraphe \ref{RedmodlK}.

\begin{lemm}
\label{Moll}
Soient $\ff$ un corps fini de caractéristique $p$, 
soit $\overline{\ff}$ une clôture algé\-brique de $\ff$ 
et soit $m\>1$ un entier. 
Soit un caractère $\tilde\chi:\ff_m^{\times}\to\overline{\ZZ}{}^{\times}_{\ell}$
et soit $\chi$ la ré\-duction de $\tilde\chi$ modulo $\ell$.
On note respectivement $f(\tilde\chi)$ et $f(\chi)$ 
les cardinaux de leurs orbites sous $\Gal(\overline{\ff}/\ff)$. 
\begin{enumerate}
\item
Si $\ell$ ne divise pas l'ordre de $\tilde\chi$, alors 
$f(\chi)= f(\tilde\chi)$.
\item 
Sinon, il existe un entier $u\>0$ tel que~:
\begin{equation}
\label{PafLeChien}
f(\tilde\chi) = f({\chi}) \cdot \el\(q^{f(\chi)}\)\ell^u
\end{equation}
où $q$ désigne le cardinal de $\ff$.
\end{enumerate}
\end{lemm}

\begin{proof}
On note $k$ l'ordre de $\tilde\chi$. 
L'entier $f(\tilde\chi)$ est l'ordre de $q$ dans $(\ZZ/k\ZZ)^{\times}$.
Si $\ell$ est premier à $k$, alors ${\chi}$
est d'ordre $k$, donc on a $f({\chi})=f(\tilde\chi)$.
Sinon, on écrit $k$ sous la forme $k'\ell^{r}$, 
avec $k'$ premier à $\ell$ et $r\>1$.
L'ordre de $q$ dans $(\ZZ/k'\ZZ)^{\times}$ est égal à 
$f({\chi})$, tandis que
l'ordre de $q$ dans $(\ZZ/\ell^{r}\ZZ)^{\times}$ est de la forme 
$\el(q)\ell^{u}$, $u\>0$.
On obtient la formule annoncée en re\-mar\-quant que l'ordre 
de $q$ dans $(\ZZ/k\ZZ)^{\times}$ est le plus petit 
multiple commun aux ordres de $q$ dans $(\ZZ/k'\ZZ)^{\times}$
et $(\ZZ/\ell^{r}\ZZ)^{\times}$ respectivement.
\end{proof}

\begin{lemm}
\label{Calcula}
L'entier $a$ est soit égal à $1$, soit de la forme $\el(q^{n(\rho)})\ell^u$
avec $u\>0$.
\end{lemm}

\begin{proof}
Fixons une décomposition de $\tilde\l$ de la forme 
$\tilde\k\otimes\tilde\s$ et 
identifions $\tilde\s$ à une re\-pré\-senta\-tion irréductible cuspidale 
de $\GL_{m'}(\kk_{\D'})$, avec les notations du paragraphe 
\ref{Pontesprit}. 
Fixons un antécédent $\tilde\chi$ de $\tilde\s$ par la correspondance 
\eqref{Quilty}.
Soit $\chi$ la ré\-duction de $\tilde\chi$ modulo $\ell$ et 
soit $\s$ la réduction de $\tilde\s$ modulo $\ell$, \ie l'image de $\chi$ 
par la correspondance \eqref{QuiltyL}.
Remarquons que~:
\begin{equation}
\label{Moutard}
a
=\frac{b(\tilde\rho)}{b(\rho)}
=\frac{s(\rho)}{s(\tilde\rho)}.
\end{equation}
Ensuite, notons $f(\chi)$ et $f(\tilde\chi)$ les cardinaux des orbites 
de $\chi$ et $\tilde\chi$ sous $\Gal(\overline{\ff}/\ff_{\E})$, 
et notons $f'(\chi)$ et $f'(\tilde\chi)$ les cardinaux 
des orbites de $\chi$ et $\tilde\chi$ sous $\Gal(\overline{\ff}/\ff_{\D'})$.
On a~:
\begin{equation*}
f(\tilde\chi)=b(\tilde\rho)f'(\tilde\chi)
\quad\text{et}\quad
f(\chi)=b(\rho)f'(\chi),
\end{equation*}
ce qui permet d'écrire la relation~:
\begin{equation*}
a
=\frac{f(\tilde\chi)}{f({\chi})}\cdot
\frac{f'({\chi})}{f'(\tilde\chi)}.
\end{equation*}
Si l'on applique le lemme \ref{Moll} avec $\ff=\ff_{\D'}$ puis avec 
$\ff=\ff_{\E}$, on en déduit d'une part que, 
si $\ell$ ne divise pas l'ordre de $\tilde\chi$, alors $a=1$,
et d'autre part que, si $\ell$ divise l'ordre de $\tilde\chi$, alors on a~:
\begin{equation*}
\label{Pip}
a=
\el\(q_{\E}^{f(\chi)}\)\cdot\el\(q_{\D'}^{f'(\chi)}\)^{-1}\cdot\ell^u,
\quad u\>0,
\end{equation*}
où $q_{\D'}$ et $q_{\E}$ désignent les cardinaux de $\ff_{\D'}$ et 
$\ff_{\E}$ respectivement.
Supposons dorénavant qu'on est dans ce second cas.
Remarquons que $f'(\tilde\chi)=m'$ et posons $\delta=m'f'(\chi)^{-1}$.
Si l'on note $f_{\E/\F}$ le degré résiduel de $\E$ sur $\F$, 
on obtient grâce à la relation \eqref{Wigmore}~:
\begin{equation*}
f_{\E/\F}f(\chi)
=f_{\E/\F}b(\rho)f'(\chi)
=n(\rho)\delta^{-1}.
\end{equation*}
Ensuite, on a le lemme suivant.

\begin{lemm}
L'entier $\delta$ est égal à $1$ ou est de la forme
$\el\(q_{\D'}^{f'(\chi)}\)\ell^u$ avec $u\>0$.
\end{lemm}

\begin{proof}
Le caractère $\chi$ est un caractère du groupe multiplicatif de l'unique extension 
de degré $m'$ de $\ff_{\D'}$ contenue dans $\overline{\kk}$.
Il se factorise sous la forme $\mu\circ\N$, où $\mu$ est un caractère du groupe
multiplica\-tif de l'unique extension de degré $f'(\chi)$ de $\ff_{\D'}$
contenue dans $\overline{\kk}$ et où $\N$ est le morphisme de norme qui 
convient.  

Par la correspondance \eqref{QuiltyL}, le caractère $\mu$ définit 
une représentation irréductible supercus\-pidale $\tau$ de 
$\GL_{f'(\chi)}(\kk_{\D'})$. 
D'après \cite{DJ} (voir aussi \cite[III.2.8]{Vig1}), 
la représentation $\s$ est l'unique sous-quotient 
irréductible cuspidal de l'induite parabolique à $\GL_{f'(\chi)}(\kk_{\D'})$ 
de $\tau\otimes\dots\otimes\tau$ ($\delta$ fois) et $\delta$ est 
soit égal à $1$ (si $\s$ est supercuspidale), soit de la forme~:
\begin{equation*}
\delta=\el\(q_{\D'}^{f'(\chi)}\)\ell^u,
\quad u\>0.
\end{equation*}
Le résultat annoncé s'ensuit. 
\end{proof}

On en déduit que $a$ est de la forme
$\el(q^{n(\rho)/\delta})\delta^{-1}\ell^u$ pour un entier $u\>0$,
ce qui prouve le lemme \ref{Calcula}.
\end{proof}

Le lemme \ref{Calcula} met fin à la preuve du théorème \ref{RedCusp}. 
\end{proof}

\begin{rema}
\label{LouisXVII}
La représentation ${\rho}$ dépend du choix de ${\nl}$, 
\ie que changer de ${\nl}$ a pour effet de tordre
${\rho}$ par une puis\-san\-ce de $\nu$.
Cependant $n(\rho)$, $f(\rho)$, $b(\rho)$, $s(\rho)$ 
ne dé\-pen\-dent que de $\ell$ et de la classe d'inertie de 
$\tilde\rho$.
On a $f(\tilde\rho)=f(\rho)$ ainsi que la formule~:
\begin{equation*}
a_\ell(\tilde\rho) \equiv 
\frac{n(\tilde\rho)}{n(\rho)}\!  \mod \ell^\ZZ,
\end{equation*}
qui provient des relations \eqref{Wigmore} et \eqref{Moutard}.
\end{rema}

\begin{coro}
\label{MieuxQueV}
Soit $\tilde\rho$ une $\qlb$-représentation irréductible cuspidale 
entière de $\G$.
Sup\-po\-sons que l'une des deux conditions suivantes est vérifiée~:
\begin{enumerate}
\item 
$\D$ est égale à $\F$ (cas déployé). 
\item
On a $\el(q^d)>m$ (cas banal). 
\end{enumerate}
Alors $a_\ell(\tilde\rho)=1$, \ie que $\r_{\ell}(\tilde\rho)$ est irréductible. 
\end{coro}

\begin{proof}
Dans le cas 1, on a $b(\tilde\rho)=b(\rho)=1$
(voir la re\-mar\-que \ref{Etoc1}) donc $a$ est égal à $1$ 
d'après \eqref{Moutard}.

Dans le cas 2, si l'on reprend les notations de la preuve du théorème 
\ref{RedCusp}, il suffit de montrer que $\ell$ ne divise pas l'ordre de 
$\tilde\chi$. 
Supposons au contraire que $\ell$ divise l'ordre de $\tilde\chi$. 
Celui-ci est un caractère du groupe 
multiplicatif de l'extension de degré $m'$ de $\ff_{\D'}$, donc $\ell$ 
divise aussi l'ordre de ce groupe, c'est-à-dire~:
\begin{equation*}
q^{f(\rho)}-1.
\end{equation*}
Comme $f(\rho)$ divise $md$, on trouve que $q^{md}$ est congru à $1$ 
mod $\ell$, ce qui contredit l'hypothèse. 
\end{proof}

\begin{rema}
\label{MieuxQueV2}
Dans le cas où $\D$ est égale à $\F$, on obtient 
une nouvelle preuve du théo\-rè\-me 
\cite[III.1.1]{Vig1}, qui n'utilise pas la théorie des dérivées.
\end{rema}

\begin{coro}
Si $a_\ell(\tilde\rho)>1$, \ie si $\r_\ell(\tilde\rho)$ est réductible, 
alors $\el\(q^{f(\rho)}\)=1$.
\end{coro}

\begin{proof}
D'après \eqref{Moutard}, l'entier $a_\ell(\tilde\rho)$ divise $s(\rho)$.
D'après le lemme \ref{Calcula}, on en déduit que $\el(q^{n(\rho)})$ 
divise $s(\rho)$.
Le résultat se déduit de la relation \eqref{Wigmore}.
\end{proof}

\begin{exem}
\label{DuPlessis}
On suppose que $\D$ est un corps de quater\-nions sur $\F$
(\ie que $d=2$). 
Soit $(\U,\tilde\chi)$ un $\overline{\QQ}_{\ell}$-type simple de niveau 
$0$ de $\mult\D$, \ie que $\U=\Oo_{\D}^{\times}$ et que le 
ca\-rac\-tè\-re~:
\begin{equation*}
\tilde\chi:\U\to
\overline{\QQ}{}^{\times}_{\ell}
\end{equation*}
est trivial sur $1+\p_{\D}$, de sorte qu'on peut 
le voir comme un caractère de $\mult{\kk}_{\D}$. 
Soient $b(\tilde\chi)$ l'indice du normalisateur de $\tilde\chi$ dans $\mult\D$ 
et $\tilde\rho$ une $\overline{\QQ}_{\ell}$-représentation irréductible 
entière de $\mult\D$ contenant $\tilde\chi$.
Celle-ci est cuspidale de niveau $0$ et de dimension finie égale à $b(\tilde\chi)$.  

Si $b(\tilde\chi)=1$, alors $\tilde\rho$ est un $\overline{\QQ}_{\ell}$-caractère entier 
de $\mult\D$ et sa réduction mod $\ell$ est un 
$\overline{\FF}_{\ell}$-carac\-tè\-re de $\mult\D$.  

Supposons maintenant que $b(\tilde\chi)=2$, \ie que $\tilde\chi^{q}\neq \tilde\chi$.
Soit ${\chi}$ la réduction mod $\ell$ de $\tilde\chi$ et soit $b(\chi)$ l'indice 
du normalisateur de $\chi$ dans $\mult\D$.  
Si l'ordre de $\tilde\chi^{q-1}$ n'est pas une puissance de $\ell$,
alors $b(\chi)=2$ et $\r_{\ell}(\tilde\rho)$ est une 
$\overline{\FF}_{\ell}$-représentation irréductible. 
Sinon, on a~:
\begin{equation}
\label{Petta1}
{\chi}^{q-1}=1,
\end{equation}
\ie que $b(\chi)=1$. 
Si $\ell>2$, $\r_{\ell}(\tilde\rho)$ est la somme des deux 
$\overline{\FF}_{\ell}$-ca\-rac\-tè\-res de $\mult\D$ prolongeant ${\chi}$. 
Si $\ell=2$, alors ${\chi}$ se prolonge de façon unique 
en un $\overline{\FF}_{2}$-caractère de $\mult\D$, que l'on note 
encore ${\chi}$, et $\r_{\ell}(\tilde\rho)$ est égale à ${\chi}+{\chi}$.

Si $q=8$ et $\ell=3$, 
la condition (\ref{Petta1}) est vérifiée pour tout caractère 
$\tilde\chi$.
Toute $\overline{\FF}_{3}$-re\-pré\-sen\-ta\-tion
cuspidale de niveau $0$ de $\mult\D$ est de dimension $1$.
\end{exem}

\subsection{Relèvement d'une représentation cuspidale}
\label{Neantise}

Soit $\ell$ un nombre premier différent de $p$. 

\begin{prop}
\label{RelCusp}
Soit $\rho$ une $\flb$-représentation irréductible cuspidale de $\G$.
Il existe une $\qlb$-représentation irréductible cuspidale entière 
$\tilde\rho$ de $\G$ telle que $\sy\rho$ apparaisse dans $\r_{\ell}(\tilde\rho)$. 
\end{prop}

\begin{proof}
Soit $(\J,\l)$ un $\flb$-type simple maximal de $\G$ contenu 
dans $\rho$.
On note $\NJ$ le nor\-ma\-li\-sateur de $\l$ dans $\G$ et $\nl$ 
le prolongement de $\l$ à $\NJ$ tel que $\rho$ soit isomorphe à 
l'induite compacte de $\nl$ à $\G$.
D'après la proposition \ref{RedTypSim}, il y a un $\qlb$-type 
simple maximal $(\J,\tilde\l)$ re\-le\-vant $\l$.
On note $\tilde\NJ$ le nor\-ma\-li\-sateur de $\tilde\l$ dans $\G$, 
et on fixe un prolongement $\tilde\nl$ de $\l$ à $\tilde\NJ$ dont 
la réduction modulo $\ell$ 
coïncide avec la restriction de $\nl$ à $\tilde\NJ$. 
Alors l'induite compacte de $\tilde\nl$ à $\G$ est irréductible 
cuspidale entière, et sa réduction modulo $\ell$ contient $\rho$ d'après le 
théo\-rè\-me \ref{RedCusp}.
\end{proof}

Dans la situation de la proposition \ref{RelCusp}, il n'y a pas toujours
de $\qlb$-représentation irré\-duc\-ti\-ble cuspidale entière de $\G$ dont la 
réduction mod $\ell$ 
soit exactement égale à $\sy\rho$ (voir l'exemple \ref{ExBadRel}).  
Cependant, si $\rho$ est supercuspidale, on a le résultat important 
suivant. 

\begin{theo}
\label{RelSuperCusp}
Soit $\rho$ une $\flb$-représentation irréductible supercuspidale de $\G$. 
Supposons que $\rho$ contient un type simple maximal de la forme 
$(\J,\k\otimes\s)$ avec $\s$ supercuspidale.
Il y a une $\qlb$-représentation irréductible cuspidale entière 
$\tilde\rho$ de $\G$ telle que $\r_{\ell}(\tilde\rho)=\sy\rho$.
\end{theo}

\begin{rema}
\label{RemaRelSuperCusp}
On prouve dans \cite{MS12} que l'hypothèse sur $\s$ est superflue, \ie que 
toute représentation irréductible supercuspidale contient un type 
simple maximal de la forme in\-di\-quée dans l'énoncé du théorème. 
On y prouve également la réciproque, à savoir que tou\-te représentation 
irré\-duc\-ti\-ble cuspidale de $\G$
contenant un type simple maximal  
$(\J,\k\otimes\s)$ avec $\s$ super\-cus\-pi\-da\-le est elle-même supercuspidale.
\end{rema}

\begin{proof}
Soit $(\J,\l)$ un $\flb$-type simple maximal de $\G$ contenu 
dans $\rho$.
Écrivons $\l$ sous la forme 
$\k\otimes\s$ et fixons un $\flb$-caractère $\chi$ correspondant à 
$\s$ par \eqref{QuiltyL}. 
D'après la pro\-po\-sition \ref{RedTypSim}, il y a un $\qlb$-type 
simple maximal $(\J,\tilde\l)$ relevant $\l$.
Plus précisément, fixons un relèvement $\tilde\k$ de $\k$ et notons 
$\tilde\chi$ l'unique $\qlb$-caractère d'ordre premier à $\ell$ 
dont la réduction mod $\ell$ soit $\chi$. 
Si l'on note $\tilde\s$ la $\qlb$-représentation cuspidale correpondant 
à $\tilde\chi$ par \eqref{Quilty} et si l'on pose 
$\tilde\l=\tilde\k\otimes\tilde\s$, alors $\tilde\NJ=\NJ$ 
(avec les notations de la preuve de la proposition \ref{RelCusp}).
Pour terminer, on suit la preuve de la proposition \ref{RelCusp}. 
\end{proof}

\begin{rema}
\label{Foutriquet}
Si $\rho$ est une $\flb$-représentation irréductible cuspidale mais non
supercuspi\-dale de $\G$, le lemme \ref{Calcula} montre que, 
si $\tilde\rho$ est une $\qlb$-représentation irréductible 
cuspidale entière de $\G$ satisfaisant à la condition de la 
proposition \ref{RelCusp}, le plus grand diviseur de $a_\ell(\tilde\rho)$ 
premier à $\ell$ ne dépend pas de $\tilde\rho$, mais uniquement de la 
classe d'inertie de $\rho$.
\end{rema}

\begin{rema}
Dans le cas banal, \ie si $\el(q^d)>m$, 
toute représentation irré\-duc\-ti\-ble cuspidale de $\G$ 
est supercuspidale, et le théorème \ref{RelSuperCusp}
s'applique (voir \cite{MS2}).
\end{rema}

\begin{exem}
\label{ExBadRel}
Donnons un exemple de $\overline{\FF}_{\ell}$-représentation 
irréductible cuspidale de $\G$ n'ad\-met\-tant aucun 
relèvement à $\overline{\QQ}_{\ell}$.
Selon la remarque \ref{Foutriquet},
il suffit de trouver un exemple où le plus grand diviseur de 
$a_\ell(\tilde\rho)$ premier à $\ell$ est $>1$.
Fixons un corps de quater\-nions $\D$ sur $\F$. 
Soit $(\J,{\s})$ un $\overline{\FF}_{\ell}$-type sim\-ple de 
niveau $0$ de $\G=\GL_{2}(\D)$. 
On peut supposer que $\J=\GL_{2}(\Oo_{\D})$ et considérer $\s$ 
comme une $\overline{\FF}_{\ell}$-représentation irréductible 
cuspidale de $\GL_{2}(\kk_\D)$. 
D'après le théorème \ref{theojames}, 
la représentation $\s$ est paramétrée par un $\overline{\FF}_{\ell}$-caractère 
$\chi$ de $\mult\kk$ (où $\kk$ est une ex\-ten\-sion quadratique fixée de 
$\kk_\D$) admettant un relè\-ve\-ment~:
\begin{equation*}
\label{SectionDesPiquesTilde}
\tilde\chi:\mult\kk\to\overline{\QQ}{}^{\times}_{\ell}
\end{equation*}
tel que $\tilde\chi{}^{q^2}\neq \tilde\chi$. 
On note $\tilde\s$ la $\overline{\QQ}_{\ell}$-représentation irréductible 
cuspidale de $\GL_{2}(\kk_\D)$ para\-métrée par $\tilde\chi$. 
C'est un relèvement de $\s$.
On suppose que~: 
\begin{equation}
\label{AnneBoylen}
\chi^{q^2}=\chi,
\end{equation}
\ie que la représentation $\s$ n'est pas supercuspidale. 
Soit $\rho$ une $\overline{\FF}_{\ell}$-représentation 
ir\-ré\-duc\-ti\-ble de $\GL_{2}(\D)$ dont la restriction à $\J$ 
contient $\s$, et soit $\tilde\rho$ une 
$\overline{\QQ}_{\ell}$-représentation irréductible dont la restriction 
à $\J$ contient $\tilde\s$. 
Ce sont des re\-pré\-sen\-ta\-tions cuspidales de niveau $0$.
Si $b(\rho)=2$, alors $\rho$ se relève en $\tilde\rho$
puisque l'entier $b(\tilde\rho)$ divise $2$ et est un multiple
de $b(\rho)=2$. 
On suppose maintenant que $b(\rho)=1$, de sorte qu'on a~:
\begin{equation}
\label{b=1}
\chi^{q}=\chi.
\end{equation}
Pour que $\rho$ se relève en une $\overline{\QQ}_{\ell}$-représentation 
irréductible cuspidale de $\GL_2(\D)$, il faut et suffit qu'on puisse 
choisir $\tilde\s$ de sorte que son orbite sous l'action de 
$\Gal(\kk_{\D}/\kk_{\F})$ soit de cardinal $1$.
En d'autres termes, il faut et suffit que $\chi$ admette un 
relèvement $\tilde\chi$ tel que~: 
\begin{equation*}
\label{HenryVIII}
\tilde\chi{}^{q^2}=\tilde\chi{}^{q}\neq\tilde\chi,
\end{equation*}
ce qui est impossible. 
En conclusion, une $\overline{\FF}_{\ell}$-re\-pré\-sentation $\rho$ 
irré\-duc\-tible cuspidale non supercuspidale de niveau $0$ de $\GL_2(\D)$ 
a un relèvement à $\overline{\QQ}_{\ell}$ si et seulement si $b(\rho)=2$. 

Par exemple, si $q=4$ et $\ell=17$, 
le groupe $\mult\kk$ est le produit direct d'un groupe cyclique
d'ordre $15$ par un groupe d'ordre $17$.
Un $\overline{\FF}_{17}$-caractère de $\mult\kk$ est 
trivial sur le facteur d'ordre $17$ et vérifie donc la condition 
\eqref{AnneBoylen}, \ie qu'aucune $\overline{\FF}_{17}$-représentation 
irré\-ductible cuspidale de $\GL_2(\kk_\D)$ n'est supercuspidale. 
Par \eqref{b=1}, 
on a $b(\rho)=1$ si et seulement si $\chi^{3}=1$.
Si $\tilde\chi$ est un $\overline{\QQ}_{17}$-caractère de $\mult\kk$ 
d'ordre $17$, 
sa réduction modulo $17$ est le caractère trivial, qui paramètre une 
$\overline{\FF}_{17}$-re\-pré\-sentation irréductible cuspidale non 
supercuspidale de niveau $0$ de $\GL_2(\D)$ n'admettant pas de relèvement 
à $\overline{\QQ}_{17}$.
\end{exem}

\section{Types et algèbres de Hecke}
\label{CostaCafe}

Soit $m\>1$ un entier et soit $\G=\GL_m(\D)$.  
Cette section est consacrée à l'étude des liens entre les représentations
lisses de $\G$ et les modules sur certaines algèbres de Hecke affines. 
Soit $(\K,\tau)$ une paire formée d'un sous-groupe ouvert 
compact $\K\subseteq\G$ et d'une représentation ir\-ré\-duc\-tible $\tau$ de 
$\K$.  

Dans le cas où $\R$ est de caractéristique nulle, on lui associe la 
sous-catégorie pleine $\Rr(\K,\tau)$ de $\Rr_\R(\G)$ formée des 
représentations engendrées par leur com\-po\-san\-te 
$\tau$-iso\-ty\-pi\-que.  
Le foncteur $\Mm_\tau$ défini au paragraphe \ref{PCMJ} est exact, 
et il induit une bijection entre les clas\-ses d'isomorphisme 
de représentations irréductibles de $\Rr(\K,\tau)$ 
et les classes d'isomorphisme de modules simples sur l'algèbre $\Hh(\G,\tau)$. 
On a une notion bien définie de type, à savoir que $(\K,\tau)$ 
est un type dans $\G$ si et seulement si l'une des trois 
propriétés équivalentes suivantes est vérifiée (voir \cite{BK1})~: 
\begin{enumerate}
\item
$\Rr(\K,\tau)$ est sta\-ble par sous-quotients dans $\Rr_\R(\G)$~;
\item
$\Mm_\tau$ induit une équivalence 
de $\Rr(\K,\tau)$ sur la catégorie des modules à droite sur $\Hh(\G,\tau)$~;
\item
$\Rr(\K,\tau)$ est la somme directe d'un nombre fini de 
blocs de Bernstein de $\Rr_\R(\G)$. 
\end{enumerate}

Dans le cas modulaire, 
la situation est plus compliquée. 
D'abord, $\tau$ n'étant pas nécessairement projective dans la 
catégorie des représentations lisses de $\K$, 
on perd une propriété importante~: le fonc\-teur $\Mm_\tau$ 
n'est pas forcément exact, et il n'induit pas 
forcément une bijection entre clas\-ses de représentations irréductibles 
de $\Rr(\K,\tau)$ et classes de modules sim\-ples sur $\Hh(\G,\tau)$. 

Ensuite, on n'a plus de notion claire de ce que devrait être un type dans $\G$~: 
outre le fait que la dé\-com\-po\-sition en blocs indécomposables de $\Rr_\R(\G)$ 
n'est pas connue\footnote{Des travaux de V.~Sécherre et 
S.~Stevens (\cite{SeSt3}) traitent cette question.}, les propriétés 1 et 2 ci-dessus ne 
sont en général pas équivalentes.

Ayant pour objectif la classification des représentations irréductibles de
$\G$ (voir \cite{MS12}), il nous est possible de nous contenter de la théorie des 
types semi-simples grâce à la propriété de quasi-pro\-jec\-tivité introduite 
dans \cite{Vig2}
(voir le paragraphe \ref{RQP}),
sans chercher à développer une théorie générale des types modulaires.
Cette notion de quasi-pro\-jec\-tivité, ainsi que l'étude systématique 
effectuée par Arabia dans l'appendice à \cite{Vig2}, sont au coeur de la 
présente section. 

Au paragraphe \ref{PCIP}, nous prouvons 
la quasi-projectivité des induites compactes des types semi-simples
(proposition \ref{pptf1}) et 
donnons des conditions pour qu'une induite parabolique soit irréduc\-ti\-ble. 
Un corol\-lai\-re est le théorème \ref{ss}, qui permet de ramener le 
problème de la classification de toutes les représentations irréductibles de 
$\G$ à celle des représentations ir\-ré\-ductibles dont le support cuspidal 
est inertiellement équivalent à $\sy{\rho}+\dots+\sy{\rho}=n\cdot\sy{\rho}$
où $n\>1$ divise $m$ et où $\rho$ est une re\-pré\-sentation irré\-duc\-tible 
cuspidale fixée de $\GL_{m/n}(\D)$.  

Dans le paragraphe \ref{DiaComIndPar},
on prouve une propriété de compatibilité du foncteur $\Mm_\bl$ à l'induction, 
lorsque $(\BJ,\bl)$ est un type semi-simple. 
Comme expliqué dans la remarque \ref{onlygodforgives}, cette propriété, 
contrairement à ce qui se passe dans le cas complexe, n'est pas une 
conséquence formelle 
de la pro\-prié\-té de compatibilité aux foncteurs de Jacquet \eqref{Gobineau} 
caractéristique des paires couvrantes. 

La méthode du changement de groupe présentée dans le paragraphe \ref{ChgtGp} 
permet de ramener la classification des représentations ir\-ré\-ductibles de $\G$
dont le support cuspidal est inertiellement équivalent à $n\cdot\sy{\rho}$
(voir plus haut) à celle des représentations irréductibles de $\GL_{n}(\F')$ 
| pour une extension finie $\F'/\F$ convenablement choisie |
admettant des vecteurs non nuls invariants par le sous-groupe d'Iwahori. 
Ceci permet de se ramener, pour certaines questions, au cas où $\D$ est 
égale à $\F$, $n$ est égal à $m$ et $\rho$ est le caractère trivial de 
$\mult\F$.

Dans le paragraphe \ref{NURHO} enfin, on associe à toute représentation 
irréductible cuspidale $\rho$ de $\G$ un caractère non ramifié $\nu_\rho$ de 
$\G$ dont l'introduction est justifiée par la proposition \ref{cuspidal}. 

Cette section est inspirée de \cite{Vig2}, 
même si en définitive nos preuves sont en général différen\-tes.
Certains résultats, comme les propositions \ref{Isomega}, \ref{pptf1}, 
\ref{EgaIndParCasIrrBisens} et \ref{AgorA}, généralisent à un 
groupe non déployé des résultats de \cite{Vig2} pour $\GL_n(\F)$.
En revanche, la proposition \ref{RussSamProp} est nouvelle, même dans le cas 
déployé.  

\subsection{Représentations quasi-projectives}
\label{RQP}

Soit $\Q$ une représentation lisse de $\G$.  
La définition suivante est due à Arabia \cite[A.3]{Vig2}. 

\begin{defi}
La représentation $\Q$ est dite \textit{quasi-projective} si, pour toute 
représen\-tation $\V$ de $\G$ et tout homomorphisme 
surjectif $\h\in\Hom_\G(\Q,\V)$, l'homomorphisme
de $\End_{\G}(\Q)$ dans $\Hom_\G(\Q,\V)$ 
défini par $\a\mapsto\h\circ\a$ est 
surjectif.
\end{defi}

On dit qu'une représentation de $\G$ est sans $\Q$-torsion si elle n'admet 
pas de sous-repré\-sen\-ta\-tion non nulle $\W$ telle que 
$\Hom_\G(\Q,\W)$ soit nul. 
Rappelons que $\Rr=\Rr_{\R}(\G)$ est la catégorie des re\-pré\-sentations 
(lisses) de $\G$ sur des $\R$-espaces vectoriels.
On a le théorème important suivant. 

\begin{theo}[\cite{Vig2}]
\label{qptf}
On suppose que $\Q$ est quasi-projective et de type fini.  
\begin{enumerate}
\item
Le foncteur $\V\mapsto\Hom_\G(\Q,\V)$ définit une 
équi\-va\-len\-ce entre la sous-catégorie pleine de $\Rr$ formée 
des repré\-sen\-ta\-tions sans $\Q$-torsion qui sont quotients 
d'une somme directe de copies de $\Q$, 
et la catégorie des $\End_{\G}(\Q)$-modules à droite. 
\item
Cette équivalence induit une bijection entre les classes de
représenta\-tions irréducti\-bles $\V$ de $\G$ telles que 
$\Hom_\G(\Q,\V)\neq\{0\}$
et les classes de $\End_{\G}(\Q)$-modules irré\-duc\-ti\-bles. 
\end{enumerate}
\end{theo}

\begin{proof}
Le point 1 est donné par \cite[A.3, théorème 4(2)]{Vig2} 
(voir aussi A.2)
et le point 2 est don\-né par \cite[A.5, théorème 10(2)]{Vig2}.
\end{proof}

Soit $\tau$ une représen\-tation irréductible d'un sous-groupe ouvert compact 
$\K$ de $\G$.
On reprend les notations du paragraphe \ref{PCMJ} et on pose~:
\begin{equation*}
\Q=\ind^\G_\K(\tau).
\end{equation*} 
Comme $\tau$ est irréductible, $\Q$ est de type fini. 
Elle n'est pas toujours projective dans $\Rr$, \ie que le foncteur $\Mm_{\tau}$ 
n'est pas toujours exact sur $\Rr$.
Cependant, si elle est quasi-projective, on va voir que ce foncteur est 
exact sur une sous-catégorie pleine suffisamment grande. 

\begin{defi}
On note $\Ee(\K,\tau)$ la sous-ca\-té\-go\-rie pleine de 
$\Rr$ dont les objets sont les sous-quo\-tients de 
repré\-sen\-ta\-tions en\-gen\-drées par leur composante $\tau$-isotypique.
\end{defi}

Cette sous-catégorie est stable par sous-quo\-tients dans $\Rr$. 
Son intérêt réside dans le résultat sui\-vant.
On note $\Hh=\Hh(\G,\tau)$ l'algèbre des endomorphismes de $\Q$. 

\begin{prop}[\cite{Vig2}]
\label{Cocyte2}
\label{ManqueDe}
\label{Cocyte}
On suppose que $\Q=\ind^{\G}_{\K}(\tau)$ est quasi-projective (et de type 
fini). 
\begin{enumerate}
\item
La restriction du foncteur $\Mm_\tau$ à $\Ee(\K,\tau)$ est un foncteur exact. 
\item
Pour tout $\Hh$-module à droite $\mm$, 
l'homomorphisme canonique de $\Hh$-modules~: 
\begin{equation*}
\mm\to\Mm_{\tau}\(\mm\otimes_{\Hh}\ind^{\G}_{\K}(\tau)\)
\end{equation*}
est un isomorphisme.
\end{enumerate}
\end{prop}

\begin{proof}
Comme $\Q$ est quasi-projective, 
le foncteur $\Mm_\tau$ est exact sur la sous-catégorie pleine $\Dd_\Q$ de
$\Rr$ définie dans \cite[A.3]{Vig2}.
D'après \textit{ibid.}, proposition $3$, la catégorie $\Ee(\K,\tau)$ est une 
sous-catégorie de $\Dd_\Q$, ce qui prouve 1. 
Le point 2 est démontré dans la preuve de \textit{ibid.}, théorème $4(2)$. 
\end{proof}

Soit $\P$ un sous-groupe parabolique de $\G$ de facteur de Levi $\M$ 
et de radical uni\-potent $\N$, et soit $\tau_\M$ une représen\-tation 
irréductible d'un sous-groupe ouvert compact $\K_\M$ de $\M$.
On suppose que $(\K,\tau)$ est une paire cou\-vrante de 
$(\K_\M,\tau_\M)$ (voir le paragraphe \ref{PCMJ}).

\begin{prop}
\label{SAvereUtileG1}
Soit $\pi$ une représentation admissible de $\M$ engendrée par sa compo\-sante 
$\tau_\M$-iso\-typique.  
On suppose que $\Q$ est quasi-projective.
Il existe un homomorphisme surjectif~:
\begin{equation*}
\label{Isomega}
\Mm_{\tau_\M}(\pi)\otimes_{\Hh_\M}\Hh\to\Mm_\tau(\ip_{\P^-}^\G(\pi))
\end{equation*}
de $\Hh$-modules à droite. 
\end{prop}

\begin{rema}
La condition d'admissibilité sur $\pi$ vient du fait que notre preuve utilise la 
proposition \ref{GobineauAdj}, qui nécessite la propriété de seconde adjonction.  
\end{rema}

\begin{proof}
On note $\Q$ l'induite compacte de $\bl$ à $\G$
et $\Q_\M$ celle de $\bl_\M$ à $\M$.
Partons de l'application canonique~:
\begin{equation}
\label{ChJoHuDi-G1}
\Mm_{\tau_\M}(\pi)\otimes_{\Hh_\M}\Q_\M\to\pi
\end{equation}
qui est surjective car $\pi$ est engendrée par sa compo\-sante 
$\tau_\M$-iso\-typique.  
En appliquant le foncteur $\ip^\G_{\P^-}$ à \eqref{ChJoHuDi-G1}, 
on obtient une application surjective~:
\begin{equation}
\label{ChJoHuDiG1}
\ip^\G_{\P^-}(\Mm_{\tau_\M}(\pi)\otimes_{\Hh_\M}\Q_\M)
\to\ip^\G_{\P^-} (\pi).
\end{equation}
D'après la proposition \ref{GobineauAdj}, le membre de gauche 
est iso\-mor\-phe à $\Mm_{\tau_\M}(\pi)\otimes_{\Hh_{\M}}\Q$.
En particu\-lier, les deux membres de \eqref{ChJoHuDiG1} sont engendrés 
par leur compo\-sante $\tau$-isotypique. 
Par la proposi\-tion \ref{Cocyte2}, en appliquant le foncteur 
$\Mm_\tau$, on obtient une application surjective~:
\begin{equation*}
\label{ChJoHuDi+G}
\Mm_\tau\(\Mm_{\tau_\M}(\pi)\otimes_{\Hh_{\M}}\Q\)
\to\Mm_\tau(\ip^\G_{\P^-}(\pi))
\end{equation*}
et le membre de gauche 
est isomorphe à $\Mm_{\tau_\M}(\pi)\otimes_{\Hh_{\M}}\Hh$ d'après la 
proposition \ref{Cocyte2}(2).
\end{proof}

\begin{prop}
\label{blablabla}
Soit $\s$ une représentation de longueur finie de $\M$ 
engendrée par sa composante $\tau_{\M}$-isotypique. 
On suppose que $\Q$ est quasi-projective. 
\begin{enumerate}
\item 
Le foncteur $\Mm_\tau$ induit une bijection entre l'ensemble des 
classes d'isomorphisme de quotients irré\-duc\-ti\-bles de $\Q$ apparaissant comme 
sous-quotient de $\ip^{\G}_{\P}(\s)$ et l'ensemble des classes d'iso\-morphisme de 
$\Hh$-modu\-les à droite simples apparaissant comme 
sous-quotient de $\Mm_\tau(\ip^{\G}_{\P}(\s))$. 
\item
Si $\pi$ est un quotient 
irré\-duc\-ti\-ble de $\Q$ apparaissant comme sous-quotient de 
$\ip^{\G}_{\P}(\s)$, alors sa multi\-plicité dans $\ip^{\G}_{\P}(\s)$ est égale à 
la multiplicité de $\Mm_\tau(\pi)$ dans $\Mm_\tau(\ip^{\G}_{\P}(\s))$. 
\end{enumerate}
\end{prop}

\begin{proof}
La représentation $\ip^{\G}_{\P}(\s)$ est de longueur finie
et est engendrée par sa compo\-san\-te $\tau$-isotypique
d'après la proposition \ref{ScotErigene}.
On fixe une suite de composition~:
\begin{equation*}
0=\V_0\subseteq\V_1\subseteq\dots\subseteq\V_n=\ip^{\G}_{\P}(\s)
\end{equation*}
dans $\Rr$ dont les quo\-tients sont ir\-ré\-ductibles. 
Comme $\Ee(\K,\tau)$ est stable par sous-quotients 
dans $\Rr$, elle contient tous les termes $\V_0,\dots,\V_n$
de cette suite de compo\-sition.
Pour $i\in\{0,\dots,n\}$,
notons $\m_i$ le module $\Mm_\tau(\V_i)$.
Comme $\Q$ est quasi-projective, 
$\m_{i+1}/\m_i$ est isomorphe à 
$\Mm_\tau(\V_{i+1}/\V_i)$ qui, compte tenu du théorème \ref{qptf}(2), est soit 
nul, soit un module simple sur $\Hh$. 
\end{proof}

\subsection{Application aux types semi-simples}
\label{PCIP}

Le cas qui nous intéresse particulièrement est celui d'un type 
semi-simple (voir le paragraphe \ref{NonEndoEq}). 

\begin{prop}
\label{pptf1}
Soit $(\BJ,\bl)$ un type semi-simple de $\G$.
Alors l'induite compacte $\cind_\BJ^\G(\bl)$ est quasi-projective et de type fini.  
\end{prop}

\begin{proof}
Nous allons montrer que la restriction de $\cind_\BJ^\G(\bl)$ à $\BJ$, notée 
$\V$, se décompose sous la forme $\X\oplus\W$, 
où $\X$ est une somme directe de copies de $\bl$ et où aucun sous-quotient 
irréductible de $\W$ n'est iso\-morphe à $\bl$. 
Le résultat sera alors une conséquence du lemme 3.1 et du paragraphe 2 de \cite{Vig3}.
Voir aussi \cite[Proposition 3.15]{HS}.

Soit $\bn$ la représentation de $\BJ^1\subseteq\BJ$ 
définie au paragraphe \ref{defeta}. 
Comme $\BJ^1$ est un pro-$p$-groupe et comme $\BJ$ normalise la classe 
d'isomorphisme de $\bn$, la représentation de $\BJ$ sur 
la com\-po\-sante $\bn$-isotypique de $\cind_\BJ^\G(\bl)$ est un 
facteur direct de $\V$. 
La décomposition $\V=\X\oplus\Y$ cherchée se déduit donc de la proposition 
\ref{Empire}.  
\end{proof}

\begin{rema}
\label{CSProjcInd}
Supposons que $\R$ est de caractéristique $\ell>0$
et soit $(\J,\l)$ un type simple de $\G$, pour lequel on reprend les 
notations des paragraphes \ref{Pontesprit} et \ref{Paralipomenes}.
Si $q(\l)$
n'est pas congru à $1$ mod $\ell$, alors d'après 
\cite[III.2.9]{Vig1} la représentation $\s$ est pro\-jec\-ti\-ve comme 
représentation de $\J/\J^1$.
Par inflation, elle est également projective comme représentation de $\J$.
Enfin, le foncteur $\ind^\G_\J$ préservant la projectivité, 
$\ind^\G_\J(\l)$ est pro\-jec\-tive dans la catégorie des 
représentations lisses de $\G$. 
\end{rema}

\begin{defi}
Deux paires cuspidales $(\M,\vr)$ et $(\M',\vr')$ de $\G$ sont 
\textit{iner\-tiellement équiva\-lentes} 
s'il y a un caractère non ramifié $\chi$ de $\M$ tel que la paire 
$(\M',\vr')$ soit conjuguée à $(\M,\vr\chi)$ sous $\G$.
\end{defi}

Si $(\M,\vr)$ est une paire cuspidale de $\G$, on note $[\M,\vr]_\G$
sa classe d'inertie (\ie sa classe d'équivalence inertielle). 

Si $\Om$ est la classe d'inertie d'une paire cuspidale de $\G$, on note~:
\begin{equation}
\label{MellamphyStar}
\Irr_{\R}(\Om)^{\q}=\cusp^{-1}(\Om)
\end{equation}
l'ensemble des classes 
de re\-pré\-sen\-ta\-tions irréductibles de $\G$ 
dont le support cus\-pidal appartient à $\Om$.

Fixons un sous-groupe de Levi standard $\M$ de $\G$, égal à 
$\M_\a$ pour une famille $\a=(m_1,\dots,m_r)$ d'entiers strictement positifs 
de somme $m$.
Pour chaque $i\in\{1,\dots,r\}$, soit $(\BJ_i,\bl_i)$ un type 
semi-simple de $\G_{m_i}$.  
Posons~:
\begin{equation*}
\BJ_\M=\BJ_1\times\dots\times\BJ_r,
\quad
\bl_\M=\bl_1\otimes\dots\otimes\bl_r.
\end{equation*}
Soit $(\BJ,\bl)$ un type semi-simple de $\G$ qui soit une paire couvrante 
de $(\BJ_\M,\bl_\M)$.

\begin{prop}
\label{BoyerDArgens}
Supposons que $(\BJ_\M,\bl_\M)$ est un type simple maximal de $\M$
et soit $\vr$ une représentation irréductible cuspidale de $\M$ 
contenant $\bl_\M$. 
Alors~:
\begin{equation*}
\pi\in\Irr([\M,\vr]_\G)^\q
\quad\Leftrightarrow\quad
\cusp(\pi)\in[\M,\vr]_\G
\quad\Leftrightarrow\quad
\Hom_\BJ(\bl,\pi)\neq\{0\}.
\end{equation*}
\end{prop}

\begin{proof}
Pour tout caractère non ramifié $\chi$ de $\M$,
la représentation $\vr\chi$ est un quotient irréductible de  
$\ind^\M_{\BJ_\M}(\bl_\M)$.
En induisant à $\G$ le long de n'importe quel sous-groupe parabolique 
$\P$ de sous-groupe de Levi $\M$, et compte tenu de \eqref{BlondelDat}, 
on en déduit que tout quotient irréductible de 
$\ip_\P^\G(\vr\chi)$ est un quotient irréductible de $\ind^\G_\BJ(\bl)$.  
Ainsi toute représentation irré\-ductible de $\G$ dont le 
support cuspidal est dans $[\M,\vr]_\G$ contient $\bl$.

Inversement, soit $\pi$ une représentation irréductible de $\G$ contenant $\bl$.  
Par \eqref{Gobineau} son module de Jacquet $\rp_\P^\G(\pi)$ contient $\bl_\M$, 
\ie qu'il y a un morphisme non trivial 
de $\ind^\M_{\BJ_\M}(\bl_\M)$ dans $\rp_\P^\G(\pi)$.
Il y a donc un sous-quotient de $\rp_\P^\G(\pi)$ 
de la forme $\vr\chi$, où $\chi$ est un carac\-tè\-re non rami\-fié de $\M$.
Soient $(\L,\s)$ un représentant du support cuspidal de $\pi$ et 
$\Q$ un sous-groupe parabolique de $\G$ de facteur de Levi $\L$ 
tels que $\pi$ se plonge dans $\ip_{\Q}^\G(\s)$.
Alors $\rp_\P^\G(\ip_{\Q}^\G(\s))$ a 
un sous-quotient isomorphe à $\vr\chi$. 
D'après le lem\-me géométrique (voir \cite[\S2.6]{MSu}),
les paires $(\M,\vr)$ et $(\L,\s)$ sont 
iner\-tiel\-le\-ment équi\-valentes, \ie que $\cusp(\pi)$ appartient à $[\M,\vr]_\G$.
\end{proof}

\begin{rema}
Nous ne connaissons pas d'analogue modulaire de \cite[Theorem 8.3]{BK1} pour 
des paires qui ne sont pas des types semi-simples. 
Plus précisément, soient $(\K,\tau)$ et $(\K_\M,\tau_\M)$ comme au paragraphe 
\ref{PCMJ}, et supposons que la première est une paire couvrante de la 
seconde.  
Supposons qu'il y a une classe inertielle $\Om_\M$ de $\M$ telle que les 
représentations ir\-ré\-ductibles de $\M$ contenant $\tau_\M$ sont 
celles dont le support cuspidal appartient à $\Om_\M$. 
On demande si~:
\begin{equation*}
\Hom_{\K}(\tau,\pi)\neq\{0\}
\quad
\Leftrightarrow
\quad
\cusp(\pi)\in\Om.
\end{equation*}
Même en supposant que les induites compactes de $\tau$ à $\G$ et de 
$\tau_\M$ de $\M$ sont quasi-projectives, la preuve de \cite[Theorem 8.3]{BK1} 
ne s'adapte pas telle quelle. 
Un sous-quotient irréductible d'une induite parabolique d'une représentation 
cuspidale irréductible peut avoir un support cuspidal différent de la classe 
de conjugaison de la paire cuspidale induisante. 
Le module de Jacquet d'une représentation irréductible peut avoir un 
sous-quotient irréductible contenant $\tau_\M$ sans avoir un 
quotient irréductible contenant $\tau_\M$.  
\end{rema}

\begin{prop}
\label{EgaIndParCasIrrBisens}
Soit $\s$ une représentation ir\-ré\-ductible de $\M$ contenant 
$\bl_\M$. 
Alors $\ip^{\G}_{\P}(\s)$ est ir\-ré\-duc\-ti\-ble
si et seulement si~: 
\begin{equation}
\label{FACT}
\Mm_{\bl}(\ip^{\G}_{\P}(\s))
\end{equation}
est un $\Hh$-module irréductible.  
\end{prop}

\begin{proof}
L'une des implications est donnée par le théorème \ref{qptf} joint à la 
proposition \ref{ScotErigene}.
Pour l'autre, soit $\pi$ une sous-représentation irréductible de 
$\ip^{\G}_{\P}(\s)$.  
Fixons une paire cuspi\-da\-le $(\L,\vr)$ dont la classe de $\M$-conjugaison 
est égale à $\cusp(\s)$.  
Sa classe de $\G$-conjugaison est donc égale à $\cusp(\pi)$.  
Fixons un type simple maximal $(\BJ_\L,\bl_\L)$ de $\L$ contenu dans $\vr$ et 
dont la paire $(\BJ_\M,\bl_\M)$ (et par conséquent $(\BJ,\bl)$) 
est une paire couvrante.
D'après la proposition \ref{BoyerDArgens}, 
le $\Hh$-module $\Mm_{\bl}(\pi)$ n'est pas nul~;
il est donc égal à \eqref{FACT}. 
Puisque $\ip^{\G}_{\P}(\s)$ est en\-gen\-drée par sa composante 
$\bl$-isotypique d'après la proposition \ref{ScotErigene}, 
elle est égale à $\pi$, ce qui prouve qu'elle est ir\-ré\-duc\-ti\-ble. 
\end{proof}

\begin{prop}
\label{AgorA}
Soit $\s$ une représentation ir\-ré\-ductible de $\M$ contenant 
$\bl_\M$. 
Supposons que~:
\begin{equation}
\label{Amelia}
\Mm_{\bl_\M}(\s)\otimes_{\Hh_\M}\Hh
\end{equation}
est un $\Hh$-module irréductible.
Alors $\Mm_{\bl}(\ip^\G_\P(\s))$ est 
irréductible et isomorphe à $\Mm_{\bl_\M}(\s)\otimes_{\Hh_\M}\Hh$,
et la représentation $\ip^\G_{\P}(\s)$ est irréductible.
\end{prop}

\begin{proof}
D'après la proposition \ref{ScotErigene}, le module 
$\Mm_\bl(\ip^\G_{\P^-}(\s))$ est non nul.
Il est donc ir\-ré\-ductible d'après la propo\-si\-tion \ref{SAvereUtileG1} 
et l'hypothèse sur \eqref{Amelia}.
Par la proposition \ref{EgaIndParCasIrrBisens}, l'induite 
$\ip^\G_{\P^-}(\s)$ est irréductible. 
D'après \cite[Lemme 4.13]{Dat3} (voir aussi \cite[Proposition 2.2]{MS12}), 
elle est donc iso\-mor\-phe à $\ip^\G_{\P}(\s)$.
\end{proof}

\begin{coro}
\label{Bertuccio}
On suppose que $\Hh$ est libre de rang $1$ sur $\Hh_\M$.
Alors $\ip^\G_\P(\s)$ est irré\-ducti\-ble,
et le $\Hh$-module $\Mm_\bl(\ip^\G_\P(\s))$ est 
irréductible et isomorphe à $\Mm_{\bl_\M}(\s) \otimes_{\Hh_\M}\Hh$. 
\end{coro}

\begin{proof}
L'induite compacte de $\bl_\M$ à $\M$ étant quasi-projective,
$\Mm_{\bl_\M}(\s)$ est irréducti\-ble d'après le théorème \ref{qptf}.
Puisque $\Hh$ est libre de rang $1$ sur $\Hh_\M$, 
l'hypo\-thèse de la proposition \ref{AgorA} est vérifiée. 
On en déduit le résultat voulu.
\end{proof}

On termine ce paragraphe en donnant une application importante du 
corollaire \ref{Bertuccio}. 
Si $\rho$ est une représentation irréductible cuspidale de $\G$,
on note~:
\begin{equation}
\label{Rrho}
\Om_{\rho}
=\{\sy{\rho\chi}\ |\ \chi:\G\to\mult\R\text{ non ramifié}\}
\end{equation}
sa classe d'équivalence inertielle.

\begin{theo}
\label{ss} 
Soit $r\>1$ un entier et soient $\rho_1,\dots,\rho_r$ des 
représentations irré\-ductibles cuspidales deux à deux non 
inertiellement équivalentes. 
Pour chaque $i\in\{1,\dots,r\}$, on fixe un support cuspidal 
$\ss_i$ formé de représentations inertiellement équivalentes 
à $\rho_i$. 
\begin{enumerate}
\item 
Pour chaque entier $i$, soit $\pi_i$ une représentation irréductible de 
support cuspidal $\ss_i$.
Alors l'induite $\pi_1\times\dots\times\pi_r$ est irréductible.
\item
Soit $\pi$ une représentation irréductible de 
support cuspidal $\ss_1+\dots+\ss_r$.
Il existe des représentations $\pi_1,\dots,\pi_r$, 
uniques à isomorphisme près, telles que $\pi_i$ soit de support 
cuspidal $\ss_i$ pour chaque $i$ 
et telles que $\pi_1\times\dots\times\pi_r$ soit 
isomorphe à $\pi$.
\end{enumerate}
\end{theo}

\begin{rema}
Pour $\ss$ un support cuspidal, notons~:
\begin{equation}
\label{NotIrrEtoile}
\Irr(\ss)^\q=\cusp^{-1}(\ss)
\end{equation}
l'ensemble des classes de
représentations irréductibles de support cuspidal égal à $\ss$.  
Si l'on pose $\ss=\ss_1+\dots+\ss_r$, l'application~: 
\begin{equation*}
(\pi_1,\dots,\pi_r)\mapsto\pi_1\times\dots\times\pi_r
\end{equation*}
induit une bijection de 
$\Irr(\ss_1)^\q\times\dots\times\Irr(\ss_r)^\q$ 
dans $\Irr(\ss)^\q$. 
\end{rema}

\begin{proof}
Pour chaque $i$, soit $(\J_i,\l_i)$ un type simple maximal contenu dans 
$\rho_{i}$, soit $n_i$ le nombre de termes dans $\ss_i$ et soit $m_i$ le degré 
de $\rho_i$. 
On fixe un type semi-simple $(\BJ_i,\bl_i)$ de $\G_{m_in_i}$ qui est une 
paire couvrante de $(\J_i\times\dots\times\J_i,\l_i\otimes\dots\otimes\l_i)$. 
On pose~:
\begin{eqnarray*}
\M&=&\G_{m_1n_1}\times\dots\times\G_{m_rn_r},\\
\BJ_\M&=&\BJ_1\times\dots\times\BJ_r,\\
\bl_\M&=&\bl_1\otimes\dots\otimes\bl_r.
\end{eqnarray*}
Soit $(\BJ,\bl)$ un type semi-simple 
qui est une paire couvrante de $(\BJ_\M,\bl_\M)$.  
L'algèbre $\Hh=\Hh(\G,\bl)$ est libre de rang $1$ comme module sur 
$\Hh_\M=\Hh(\M,\bl_\M)$. 
D'après la propo\-sition \ref{pptf1}, 
la re\-pré\-sen\-tation $\ind^\G_\BJ(\bl)$
est quasi-pro\-jec\-tive de type fini.
On est donc dans les conditions d'application du corol\-laire \ref{Bertuccio},
dont on déduit que l'induite $\pi_1\times\dots\times\pi_r$ est 
irréductible.  

Pour le point 2, on remarque que (\ref{Gobineau}) fournit un 
isomorphisme de $\Hh_\M$-modules~: 
\begin{equation*}
\Mm_{\bl}(\pi)\simeq\Mm_{{\bl_\M}}(\rp^\G_\P(\pi))
\end{equation*}
où $\P$ est un sous-groupe parabolique de $\G$ de facteur de Levi $\M$. 
Puisque $\Hh$ est libre de rang $1$ sur $\Hh_\M$, la restriction de 
$\Mm_{\bl}(\pi)$ à $\Hh_\M$ est irréductible. 
C'est donc un module $\mm$ de la forme $\mm_1\otimes\dots\otimes\mm_r$, 
où $\mm_i$ est un $\Hh(\G_{m_in_i},\bl_i)$-mo\-dule irréductible.
Puisque l'induite compacte de $\bl_i$ à $\G_{m_in_i}$ est quasi-projective de 
type fini, il existe d'après le théorème \ref{qptf} 
une représentation irréductible $\pi_i$ de $\G_{m_in_i}$ 
telle que $\Mm_{\bl_i}(\pi_i)$ soit iso\-morphe à $\mm_i$.
En particulier, le support cuspidal de $\pi_i$ 
est formé de représentations inertiellement équivalentes 
à $\rho_i$. 
Posons~:
\begin{equation*}
\pi'=\pi_1\times\dots\times\pi_r
\end{equation*}
qui est irréductible d'après le point 1 ci-dessus.
D'après le théorème \ref{qptf} à nouveau, il suffit de prouver que $\Mm_{\bl}(\pi)$ et 
$\Mm_{\bl}(\pi')$ sont des $\Hh$-modules isomorphes pour en déduire que $\pi$ 
et $\pi'$ sont des représentations isomorphes.
D'après le corollaire \ref{Bertuccio}, on a~:
\begin{equation*}
\Mm_{\bl}(\pi')\simeq\mm\otimes_{\Hh_{\M}}\Hh.
\end{equation*}
La restriction de $\Mm_{\bl}(\pi)$ à $\Hh_\M$ est isomorphe à 
$\mm$, ce dont on déduit par adjonction que $\Mm_{\bl}(\pi)$ est isomorphe 
à $\mm\otimes_{\Hh_{\M}}\Hh$, ce qui donne le résultat voulu.
\end{proof}

\subsection{Compatibilité du foncteur des $\bl$-invariants à l'induction parabolique}
\label{DiaComIndPar}

Soit $m\>1$,
soit $\rho$ une représentation irréductible cuspidale de $\G_m$, 
soit $(\J,\l)$ un type simple maximal contenu dans $\rho$
et soit $[\La,n_\La,0,\b]$ une strate simple de $\A_m$ par rapport à laquelle
$(\J,\l)$ est défini.

Soit un entier $n\>1$ et posons $\G=\G_{mn}$.
Fixons une strate simple $[\La^{\dag},n_{\La^\dag},0,\b]$ de $\A=\A_{mn}$ satisfaisant aux
conditions du paragraphe \ref{EndoEq} permettant de construire un type semi-simple~:
\begin{equation*}
(\BJ,\bl)
\end{equation*}
de $\G$ qui est une paire couvrante du type simple maximal
$(\J\times\dots\times\J,\l\otimes\dots\otimes\l)$ de $\G_m^n$. 
On note $\Hh_{}$ l'algèbre de Hecke de ce type semi-simple 
(dont la structure est donnée par la remarque \ref{Gurgeh} et ne dépend pas du 
choix de la suite de réseaux $\La^\dag$)
et on note $\Mm_{}=\Mm_{\bl}$ le 
foncteur qu'il définit de $\Rr(\G)$ dans la catégorie 
des $\Hh_{}$-mo\-du\-les à droite. 

Soit une famille $\a=(n_1,\dots,n_r)$ d'entiers strictement positifs 
de somme notée $n$. 
Notons $\M$ le sous-groupe de Levi standard $\M_{(mn_1,\dots,mn_r)}$ de $\G$.
L'intersection $\BJ\cap\M$ sera notée $\BJ_\a$ et 
la restriction de $\bl$ à $\BJ\cap\M$ sera notée $\bl_\a$. 
Notons $\Hh_\a$ l'algèbre de Hecke de $(\BJ_\a,\bl_\a)$
et $\Mm_\a$ le fonc\-teur 
qui lui correspond de $\Rr(\M)$ dans la catégorie des $\Hh_\a$-modules à droite.  
Soit enfin~:
\begin{equation}
\label{JA}
j_{\a}:\Hh_\a\to\Hh_{}
\end{equation}
le morphisme \eqref{tePe} cor\-res\-pon\-dant à $(\BJ,\bl)$ 
considéré comme une paire couvrante de $(\BJ_\a,\bl_\a)$.

Rappelons qu'on a défini un entier $f(\rho)\>1$ au paragraphe \ref{InvRho}.
On introduit la notation~:
\begin{equation}
\label{QRHO}
\qr=q^{f(\rho)}.
\end{equation}
L'entier $\qr$ est égal au $q(\l)$ du paragraphe \ref{Paralipomenes}. 
Considérons l'algèbre de Hecke affine $\Hh(n,\qr)$ qui y est définie. 
Il est plus commode d'utiliser ici une présentation de 
cette $\R$-algèbre différente de celle du paragraphe \ref{Paralipomenes}.  
Elle est engendrée par les éléments $\SS_1,\dots,\SS_{n-1}$, 
$\X_1,\dots,\X_n$ et leurs inverses $(\X_1)^{-1},\dots,(\X_n)^{-1}$, 
vérifiant les relations \eqref{R1aff}, \eqref{R2aff}, 
\eqref{R3aff} et les relations~:
\begin{eqnarray}
\label{R4aff}
\X_i\X_j&=&\X_j\X_i, \quad i,j\in\{1,\dots,n\},\\
\label{R5aff}
\X_j\SS_i&=&\SS_i\X_j, 
\quad i\notin\{j,j-1\},\\
\label{R6aff}
\SS_i\X_{i}\SS_i&=&\qr\X_{i+1}, \quad i\in\{1,\dots,n-1\},\\
\X_j(\X_j)^{-1}=(\X_j)^{-1}\X_j&=&1, \quad j\in\{1,\dots,n\}.
\end{eqnarray}

Soit $\Hh(\a,\qr)$ la sous-algèbre engendrée par $\X_1,\dots,\X_n$ et leurs 
inverses $(\X_1)^{-1},\dots,(\X_n)^{-1}$ 
et les $\SS_{i}$ pour $i\in\{1,\dots,n-1\}$ décrivant les entiers qui 
ne sont pas de la forme $n_1+\dots+n_k$ pour $k\in\{1,\dots,r\}$. 
Notons $t_\a$ l'homomorphisme d'inclusion de $\Hh(\a,\qr)$ dans $\Hh(n,\qr)$. 

Soit $\nl$ le type simple maximal étendu prolongeant $\l$ au sous-groupe $\K$ 
tel que la représentation $\rho$ soit isomorphe à l'induite com\-pacte de $\nl$ 
à $\G_m$ (proposition \ref{MartinDuGard}). 
Il détermine un isomorphisme de $\R$-al\-gè\-bres~:
\begin{equation}
\label{MortonFullerton1}
\Psi:\R[\X,\X^{-1}]=\Hh(1,\qr)\to\Hh(\G_m,\l)
\end{equation}
(l'image de $\X$ étant la fonction de support $\K$ prenant la valeur 
$\nl(x)$ en tout $x\in\K$).
On montre le résultat suivant en raisonnant comme dans \cite[\S2.10]{VSU0}.

\begin{prop}
\label{ExistenceHeckeProp}
\begin{enumerate}
\item 
Il y a un unique isomorphisme de $\R$-algèbres $\Psi_{\rho,n}$ 
de $\Hh(n,\qr)$ dans $\Hh$ tel que~:
\begin{equation}
\label{ExistenceHecke}
\Psi_{\rho,n}\circ t_{(1,\dots,1)}=\te^{}_{(1,\dots,1)}\circ(\Psi\otimes\dots\otimes\Psi).
\end{equation}
\item
Si l'on pose $\Psi_{\rho,\a}=\Psi_{\rho,n_1}\otimes\dots\otimes\Psi_{\rho,n_r}$,
alors on a $\Psi_{\rho,n}\circ t_{\a}=\te_{\a}\circ\Psi_{\rho,\a}$.
\end{enumerate}
\end{prop}

Le résultat principal de ce paragraphe est le suivant. 

\begin{prop}
\label{RussSamProp}
Soit $\s$ une représentation admissible de $\M$ 
engen\-drée par sa composante $\bl_\a$-isotypique.
On a un isomorphisme de $\Hh_{}$-modules à droite~:
\begin{equation}
\label{RussSamTh}
\Mm_{}(\ip_{(mn_1,\dots,mn_r)}(\s))\simeq
\Hom_{\Hh_{\a}}(\Hh_{},\Mm_{\a}(\s)).
\end{equation}
\end{prop}

\begin{rema}
\label{onlygodforgives}
Dans le cas où $\R$ est le corps des nombres complexes, 
le résultat est con\-nu pour une représentation lisse $\s$ 
pas nécessairement admissible. 
Il s'agit d'un cas particu\-lier de \cite[Corollary 8.4]{BK1} obtenu à 
partir de (\ref{Gobineau}) par adjonction, en 
utilisant le fait que $\Mm_{}$ induit une équivalence 
de catégories entre la sous-catégorie de $\Rr(\G)$ formée 
des représentations engendrées par leur composante 
$\bl$-isotypique et la catégorie des $\Hh_{}$-modules à droite
(ainsi qu'un résultat analogue pour $\Mm_{\a}$).
Quand $\R$ est de caractéristique non nulle, ces foncteurs 
n'induisent pas en général des équivalences de ca\-té\-gories 
et il faut trouver une autre approche.  
\end{rema}

\begin{rema}
L'hypothèse d'admissibilité vient du fait que notre 
preuve utilise la proposition \ref{SAvereUtileG1} et l'inégalité 
\eqref{QuelEstTonNom}.
\end{rema}

\begin{proof}
Posons $\ip=\ip_{(mn_1,\dots,mn_r)}$ et définissons des fonc\-teurs~:
\begin{equation*}
\textbf{\textsf{F}}:\s\mapsto\Mm_{}(\ip(\s)),
\quad
\textbf{\textsf{G}}:\s\mapsto\Hom_{\Hh_{\a}}(\Hh_{},\Mm_{\a}(\s)) 
\end{equation*}
de $\Rr(\M)$ dans la catégorie des $\Hh_{}$-modules à droite. 
Notons respectivement $\Q_{}$ et $\Q_\a$ les in\-dui\-tes compactes de 
$\bl$ à $\G$ et de $\bl_\a$ à $\M$.
Notons aussi $\Q_\N$ le module de Jacquet $\rp_{(mn_1,\dots,mn_r)}(\Q)$, 
où $\N$ est le radical unipotent du sous-groupe parabolique standard 
$\P=\P_{(mn_1,\dots,mn_r)}$.
Par ad\-jonc\-tion, on a un isomorphisme 
fonctoriel de $\R$-espaces vectoriels~: 
\begin{equation}
\label{MemberSheep}
\textbf{\textsf{F}}(\s)\simeq\Hom_{\M}(\Q_\N,\s).
\end{equation}
La représentation $\Q_{}$ est naturellement un $\Hh_{}$-module à gauche, 
ainsi que le quotient $\Q_{\N}$ puisque l'action de $\G$ commute à celle 
de $\Hh_{}$.
On en déduit une structure de $\Hh_{}$-modules à droite sur le membre de droite 
de \eqref{MemberSheep} faisant de cet isomorphisme de $\R$-espaces vectoriels 
un isomorphisme de $\Hh_{}$-modules à droite. 
En appliquant \eqref{Gobineau} à $\Q_{}$, on obtient 
le résultat suivant. 

\begin{enonce}{Fait}
\label{ShaunTheSheep}
L'isomorphisme \eqref{Gobineau} appliqué à la représentation 
$\Q_{}$ induit un isomorphisme de $(\Hh_{},\Hh_{\a})$-bimodules
de $\Hh_{}$ vers $\Mm_{\a}(\Q_{\N})$.
\end{enonce}

Grâce à \eqref{MemberSheep} et au fait \ref{ShaunTheSheep}, 
on obtient un homo\-morphisme de $\Hh_{}$-modules~:
\begin{equation*}
\boldsymbol{\omega}_{\s}:
\textbf{\textsf{F}}(\s)\simeq\Hom_{\M}(\Q_{\N},\s)
\ \iso{\boldsymbol{\e}_\s}\ \;
\Hom_{\Hh_{\a}}(\Mm_{\a}(\Q_{\N}),\Mm_{\a}(\s))
\simeq\textbf{\textsf{G}}(\s)
\end{equation*}
qui est fonctoriel en $\s$, où $\boldsymbol{\e}_\s$ désigne l'homomorphisme 
fonctoriel de $\R$-espaces vec\-to\-riels obtenu en appliquant le foncteur 
$\Mm_\a$. 

Supposons maintenant que $\s$ est engendrée par sa composante 
$\bl_\a$-isotypique, \ie qu'il y a un homomorphisme surjectif~:
\begin{equation*}
f:\Q_\a^{\SS}\to\s
\end{equation*}
d'une somme directe arbitraire de copies de $\Q_\a$ vers $\s$,
où $\SS$ désigne un ensemble quelconque qui indexe la somme directe.  
Ceci donne le diagramme commutatif~:
\begin{equation*}
\xymatrix{
\textbf{\textsf{F}}(\Q_\a)^{\SS}\ar[d]_{\boldsymbol\omega_{\Q_\a^\SS}}\ar[r]^{\textbf{\textsf{F}}(f)}
&\textbf{\textsf{F}}(\s)\ar[d]^{\boldsymbol\omega_\s}\\
\textbf{\textsf{G}}(\Q_\a)^{\SS}\ar[r]_{\textbf{\textsf{G}}(f)}&\textbf{\textsf{G}}(\s)}
\end{equation*}
où les deux flèches horizontales $\textbf{\textsf{F}}(f)$ et $\textbf{\textsf{G}}(f)$ sont 
surjectives car les deux foncteurs $\textbf{\textsf{F}}$ et $\textbf{\textsf{G}}$ sont exacts 
sur la catégorie $\Ee(\BJ_\a,\bl_\a)$,
d'après les propositions \ref{Cocyte} et \ref{ScotErigene}. 

\begin{lemm}
\label{lemeng}
Si $\boldsymbol\omega_{\Q_\a}$ est surjectif, 
alors $\boldsymbol\omega_\s$ est un isomorphisme pour toute 
re\-pré\-sentation $\s$ admissible et engendrée par sa composante 
$\bl_\a$-isotypique.
\end{lemm}

\begin{proof}
Si $\boldsymbol\omega_{\Q_\a}$ est surjectif,
alors $\boldsymbol\omega_{\Q_\a^\SS}$ est surjectif, 
donc $\boldsymbol\omega_\s$ l'est aussi. 
On en déduit l'inégalité~:
\begin{equation}
\label{QuelEstTonNom}
\dim\textbf{\textsf{F}}(\s)\>\dim\textbf{\textsf{G}}(\s). 
\end{equation}
On a aussi l'inégalité 
$\dim\textbf{\textsf{F}}(\s)\<\dim\textbf{\textsf{G}}(\s)$ 
d'après la proposition \ref{SAvereUtileG1}. 
On en déduit que $\boldsymbol\omega_\s$ est bijectif pour $\s$ admissible 
et engendrée par sa composante $\bl_\a$-isotypique.  
\end{proof}

\begin{lemm}
\label{DeDeux}
L'homomorphisme $\boldsymbol\omega_{\Q_\a}$ est surjectif.
\end{lemm}

\begin{proof}
D'après \eqref{BlondelDat} les représentations $\ip(\Q_\a)$ et $\Q_{}$ sont 
isomorphes. 
On en déduit que $\textbf{\textsf{F}}(\Q_\a)$ est libre de rang $1$ sur $\Hh_{}$. 
Plus précisément, identifions celui-ci à $\Hom_{\M}(\Q_{\N},\Q_\a)$ 
et notons $e$ l'élément de $\Hom_{\M}(\Q_{\N},\Q_\a)$ défini par~: 
\begin{equation*}
f\text{ mod }\Q_{}(\N)\mapsto\Big(x\mapsto\int\limits_{\N}f(ux)\ du\Big)
\end{equation*}
pour tous $f\in\Q_{}$ et $x\in\M$, où $\Q_{}(\N)$ désigne le sous-espace 
de $\Q_{}$ engendré par les vecteurs de la forme $u\cdot f-f$, avec 
$f\in\Q_{}$ et $u\in\N$.
Alors $h\mapsto e*h$ (où $*$ désigne l'action de $\Hh$ à droite sur 
$\Hom_{\M}(\Q_{\N},\Q_\a)$) est un isomorphisme de $\Hh_{}$-modules 
de $\Hh$ vers $\Hom_{\M}(\Q_{\N},\Q_\a)$.

Nous allons vérifier que l'image de $e$ dans $\textbf{\textsf{G}}(\Q_\a)$
est un générateur de ce $\Hh_{}$-module.
Identi\-fions les $\Hh$-modules $\textbf{\textsf{G}}(\Q_\a)$ et 
$\Hom_{\Hh_\a}(\Hh_{},\Hh_\a)$. 
Compte tenu de l'isomorphisme \eqref{BlondelDat} et du fait
\ref{ShaunTheSheep}, l'image de $e$ dans $\textbf{\textsf{G}}(\Q_\a)$
est l'applica\-tion notée $e'$ qui à $\T\in\Hh$ associe la fonction~:
\begin{equation}
\label{spetze3avant}
x\mapsto\int\limits_\N{\T}(ux)\ du
\end{equation}
de $\M$ dans l'espace des $\R$-endomorphismes de $\bl_\a$. 

Soit $\Ww=\Ww_\l$ le groupe défini au paragraphe \ref{Paralipomenes}, soit $\Ww_0$ le 
sous-groupe de $\Ww$ constitué des matrices de permutation dans 
$\Ww$ et soit $\Ww_\a=\Ww_0\cap\M$. 
Rappelons que $\P=\M\N$.

\begin{lemm}
\label{spetze}
Soit $w\in\Ww_0$.
Alors $\BJ w\BJ\cap\P\neq\varnothing$ si et seulement si $w\in\Ww_\a$. 
\end{lemm}

\begin{proof}
Rappelons que $\BJ_\a=\BJ\cap\M$.
Si $w\in\Ww_\a$, alors~:
\begin{equation*}
\BJ w\BJ\cap\P\supseteq\BJ_\a w\BJ_\a\cap\M\neq\varnothing.
\end{equation*}
Inversement, supposons que $\BJ w\BJ\cap\P$ est non vide.  
On peut supposer que le sous-groupe parahorique $\U(\La^\dag)$ est standard. 
Ainsi $\BJ$ est inclus dans $\U(\La^\dag)$~; on a donc 
$\U(\La^\dag)w\U(\La^\dag)\cap\P\neq\varnothing$. 
Notons $\AA$ l'ordre héréditaire standard de $\A_{mn}$ formé des matrices 
à coefficients entiers dont la réduction mod $\p_\D$ est triangulaire 
supérieure par blocs de taille $m$, et posons $\U=\U(\AA)$. 

Le groupe $\U$ contient $\U(\La^\dag)$, donc $\U w\U\cap\P\neq\varnothing$. 
Soient $\Ww_{{\rm 1}}$ le sous-groupe des permutations de 
$\GL_{mn}(\Oo_\D)$ et $\X_1$ un système de 
représentants des doubles classes de $\Ww_{{\rm 1}}$ 
modulo $\Ww_{{\rm 1}}\cap\U$ tel que $\X_1\cap\M$ soit 
un système de représentants des doubles classes de $\Ww_{{\rm 1}}\cap\M$ 
modulo $\Ww_{{\rm 1}}\cap\U\cap\M$. 
Alors on a~:
\begin{equation*}
\coprod\limits_{x\in\X_1}\U x\U \cap\P
=\GL_{mn}(\Oo_\D)\cap\P
=\coprod\limits_{x\in\X_1\cap\M}\U x\U \cap\P. 
\end{equation*}
On en déduit que $w$ appartient à $\X_1\cap\M$, 
donc que $w\in\Ww_0\cap\M=\Ww_\a$.
\end{proof}

Grâce à la proposition \ref{ExistenceHeckeProp}, identifions 
$\Hh$ et $\Hh_\a$ avec $\Hh(n,\qr)$ et $\Hh(\a,\qr)$ respectivement, 
de façon que 
le morphisme injectif \eqref{JA} corresponde au morphisme d'inclusion.
Pour tout $w\in\Ww$, notons $\SS_w$ l'élément de $\Hh$ défini par \eqref{DEFSW}. 

Pour $w\in\Ww_0$, la classe $w\Ww_{\a}$ possède 
un unique élément de longueur minimale.
Ces éléments de $\Ww_0$ de longueur mini\-male forment un système de 
représentants de $\Ww_0$ modulo $\Ww_{\a}$ noté $\EuScript{D}_\a$. 

\begin{enonce}{Fait}
\label{benmakhlouf}
$\Hh_{}$ est un $\Hh_\a$-module à droite libre de base $({\SS}_w)_{w\in\EuScript{D}_\a}$
et un $\Hh_\a$-module à gau\-che libre de base $({\SS}_{w^{-1}})_{w\in\EuScript{D}_\a}$
\end{enonce}

Pour tout $w\in\EuScript{D}_\a$, 
notons $\Y_w$ l'élément de $\textbf{\textsf{G}}(\Q_\a)$ défini par~:
\begin{equation*}
\Y_w({\SS}_{w'}) = 
\left\{
\begin{array}{ll}
\SS_1 & \text{si $w'=w$}, \\
0 & \text{sinon},
\end{array}
\right.
\end{equation*}
pour $w'\in\EuScript{D}_\a$.
D'après le fait \ref{benmakhlouf}, 
le $\Hh_\a$-module à gauche $\textbf{\textsf{G}}(\Q_\a)$ est libre de base 
$\{\Y_w, w\in\EuScript{D}_\a\}$.
Remarquons grâce au lemme \ref{spetze} 
que l'application $e'$ décrite plus haut 
(\ie l'image de $e$ dans $\textbf{\textsf{G}}(\Q_\a)$) est égale à $\Y_1$.
Nous allons montrer que $\Y_1$ engendre $\textbf{\textsf{G}}(\Q_\a)$
en tant que $\Hh$-module à droite\footnote{Nous nous inspirons 
d'arguments non publiés qui nous ont été communiqués par Vanessa Miemietz, 
que nous remercions.}. 

\begin{lemm}
\label{NoIdeal}
L'application $\T\mapsto\Y_1\T$ de $\Hh$ vers $\textbf{\textsf{G}}(\Q_\a)$ 
est un homomorphisme injectif de $\Hh$-modules à droite.
\end{lemm}

\begin{proof}
Soit $\T\in\Hh$ non nul, que l'on écrit (grâce au fait \ref{benmakhlouf})~:
\begin{equation*}
\T=\sum\limits_{w\in\EuScript{D}_\a} h(w)\SS_{w^{-1}},
\quad
h(w)\in\Hh_\a.
\end{equation*}
Pour $w\in\Ww_0$, notons $l(w)$ la longueur de $w$ (relativement aux 
$\{s_1,\dots,s_{n-1}\}$ du paragraphe \ref{Paralipomenes}) et posons~:
\begin{equation*}
l = l(\T) = {\rm min}\ \{l(w)\ |\ w\in\EuScript{D}_\a \text{ et } h(w)\neq0\}.
\end{equation*}
Prouvons par récurrence sur $l$ que $\Y_1$ est non nul sur l'idéal 
$\T\Hh$. 
Si $l=0$, alors $\Y_1(\T)=h_1\neq0$.  
Supposons maintenant que $l\>1$. 
Fixons un $w\in\EuScript{D}_\a$ tel que $l(w)=l$ et un $s\in\Ww_0$ de 
longueur $1$ tel que $l(sw)<l$. 
Pour $w'\in\EuScript{D}_\a$, écrivons 
$sw'=w'(s)z$ avec $w'(s)\in\EuScript{D}_\a$ et $z\in\Ww_\a$.
Alors~:
\begin{equation*}
\T\SS_s = \sum\limits_{w'\in\EuScript{D}_\a} h(w')\SS_{w'^{-1}}\SS_s
\end{equation*}
contient le terme non nul $h(w)\SS_{(sw)^{-1}}=h(w)\SS_{z^{-1}}\SS_{w(s)^{-1}}$ 
avec $l(w(s))<l$. 
Supposons qu'il exis\-te un $w'\in\EuScript{D}_\a$ tel que $\SS_{w'^{-1}}\SS_s$ 
ait une composante non nulle dans $\Hh_{\a}\SS_{w(s)^{-1}}$.
Si $l(sw')>l(w')$, alors~:
\begin{equation*}
\SS_{w'^{-1}}\SS_s=\SS_{(sw')^{-1}}\in\Hh_\a\SS_{w'(s)^{-1}}.
\end{equation*}
On en déduit que $w'(s)=w(s)$, ce qui implique que $w'=w$.
Si $l(sw')<l(w')$, alors~: 
\begin{equation*}
\SS_{w'^{-1}}\SS_s 
= \SS_{(sw')^{-1}} \SS_s^2
= (q(\rho)-1)\cdot\SS_{w'^{-1}} + q(\rho)\cdot\SS_{(sw')^{-1}}.
\end{equation*}
On a $l(w')\>l>l(w(s))$ donc ici encore $w'(s)=w(s)$, 
ce qui implique que $w'=w$.

Ainsi la composante de $\T\SS_s$ dans $\Hh_{\a}\SS_{w(s)^{-1}}$ est égale à 
$h(w)\SS_{(sw)^{-1}}$, qui est non nulle.
On en déduit que $l(\T\SS_s)<l$ puis, par hypothèse de récurrence, que 
$\Y_1$ n'est pas nulle sur $\T\SS_s\Hh$.
En particulier, $\Y_1$ n'est pas nulle sur $\T\Hh$.
\end{proof}

\begin{lemm}
\label{LE2}
Pour tout $w\in\EuScript{D}_\a$, il existe un $\T_w\in\Hh$ tel que 
$\Y_w=\Y_1\T_w$.
\end{lemm}

\begin{proof}
Soit $\Hh^0$ la sous-algèbre de $\Hh$ engendrée par les $\SS_{w}$ pour 
$w\in\Ww_0$, et soit $\Hh^0_\a$ l'intersection de $\Hh_\a$ avec $\Hh^0$. 
La restriction à $\Hh^0$ de $\T\mapsto\Y_1\T$ 
induit un morphisme injectif de $\Hh^0$-modules à droite 
de $\Hh^0$ dans $\Hom_{\Hh^0_\a}(\Hh^0,\Hh_\a^0)$.
Les deux membres ayant la même dimension sur $\R$, égale à l'ordre de 
$\Ww_0$, cet homomorphisme est bijectif. 
Pour tout $w\in\EuScript{D}_\a$,
il y a donc un $\T_w\in\Hh^0$ tel que $\Y_w$ et $\Y_1\T_w$ 
coïncident sur $\Hh^0$.
Comme $\Hh^0$ engendre $\Hh$ comme $\Hh_\a$-module à droite, 
on a $\Y_w=\Y_1\T_w$. 
\end{proof}

\begin{lemm}
\label{LE3}
Soit $\d$ l'élément de plus grande longueur dans $\EuScript{D}_\a$.
Notons ${}_{\d}\Hh$ le $\Hh$-module à droite $\Hh$ muni de la 
structure de $\Hh_\a$-module à gauche définie par~:
\begin{equation*}
\SS_{z}\cdot\T = \SS_{\d^{-1}z\d}\T
\end{equation*}
pour $\T\in\Hh$ et $z\in\Ww_\a$.
L'application $\T\mapsto\Y_{\d}\T$ est un morphisme de 
$(\Hh_\a,\Hh)$-bimodules de ${}_{\d}\Hh$ dans $\textbf{\textsf{G}}(\Q_\a)$. 
\end{lemm}

\begin{proof}
D'abord, la structure de $\Hh_\a$-module à gauche sur ${}_{\d}\Hh$ 
est bien définie car $\d$ normalise $\EuScript{D}_\a$.
Ensuite, l'application 
$\T\mapsto\Y_{\d}\T$ est un morphisme de $\Hh$-modules à droite. 
Pour prouver que c'est un morphisme de bimodules, 
il suffit de vérifier que $\Y_{\d}(\SS_{z}\SS_w)=\SS_{z}\cdot\Y_\d(\SS_w)$ 
pour tous $z\in\Ww_\a$ et $w\in\EuScript{D}_\a$.
Si $w=\d$, alors $\d$ est de longueur minimale dans  
$\d\EuScript{D}_\a=\EuScript{D}_\a\d$, donc on a~:
\begin{equation*}
\Y_{\d}(\SS_{z}\SS_\d)=\Y_{\d}(\SS_{z\d})=\Y_{\d}(\SS_{\d}\SS_{\d^{-1}z\d})
=\SS_{z}\cdot\Y_{\d}(\SS_\d).
\end{equation*}
Si $w\neq\d$, il suffit de prouver que la composante de $\SS_{z}\SS_w$ 
dans $\SS_{\d}\Hh_{\a}$ est nulle.
Cette composante s'écrit sous la forme $\SS_{\d}h$ avec $h\in\Hh_\a$. 
Ainsi~:
\begin{equation*}
\SS_{w}=(\SS_z)^{-1}\SS_{\d}h\in\SS_\d\Hh_\a
\end{equation*}
car $\Hh_\a\SS_\d=\SS_\d\Hh_\a$.
\end{proof}

\begin{lemm}
\label{LE4}
L'application $\T\mapsto\Y_{\d}\T$ de $\Hh$ vers $\textbf{\textsf{G}}(\Q_\a)$ 
est bijective. 
\end{lemm}

\begin{proof}
On prouve l'injectivité de cette application 
par un argument analogue à celui de la preuve du lemme \ref{NoIdeal}.

Pour la surjectivité, on prouve, par un argument analogue à celui de la preuve
du lemme \ref{LE2}, que pour tout $w\in\EuScript{D}_\a$, il existe un 
$\U_w\in\Hh$ tel que $\Y_w=\Y_{\d}\U_w$.
L'image de $\T\mapsto\Y_{\d}\T$ contenant les $\Y_{w}$ pour 
$w\in\EuScript{D}_\a$, qui forment une base de 
$\textbf{\textsf{G}}(\Q_\a)$ comme $\Hh_\a$-module à gauche, 
on déduit du lemme \ref{LE3} que cette application est surjective. 
\end{proof}

Comme $\Y_\d=\Y_1\T_\d$ (voir le lemme \ref{LE2}), 
la surjectivité de l'application $\T\mapsto\Y_{\d}\T$ implique celle de 
$\T\mapsto\Y_{1}\T$.
Avec le lemme \ref{NoIdeal}, ceci prouve que $\T\mapsto\Y_1\T$ est 
non seulement surjectif mais bijectif. 
Ceci met fin à la preuve du lemme \ref{DeDeux}. 
\end{proof}

La proposition \ref{RussSamProp}
se déduit maintenant des lemmes \ref{lemeng} et \ref{DeDeux}.  
\end{proof}

\begin{coro}
\label{RussSamPropCoro}
Soit $\s$ une sous-représentation d'une représentation 
admissible de $\M$ en\-gen\-drée par sa composante 
$\bl_\a$-isotypique.
On a un isomorphisme~:
\begin{equation}
\label{RussSamThCoro}
\Mm_{}\(\ip_{(mn_1,\dots,mn_r)}(\s)\)\simeq
\Hom_{\Hh_{\a}}(\Hh_{},\Mm_{\a}(\s))
\end{equation}
de $\Hh_{}$-modules à droite. 
\end{coro}

\begin{proof}
Par hypothèse, il y a des représentations $\pi_1$, $\pi_2$ admissibles 
en\-gen\-drées par leurs composantes $\bl_\a$-isotypiques telles qu'on ait 
une suite exacte~:
\begin{equation*}
0\to\s\fr{i}\pi_1\fr{f}\pi_2\to0
\end{equation*}
dans la catégorie $\Ee(\BJ_\a,\bl_\a)$. 
Ceci donne le diagramme commutatif~:
\begin{equation*}
\xymatrix{
\textbf{\textsf{F}}(\s)\ar[d]_{\boldsymbol\omega_{\s}}\ar[r]^{\textbf{\textsf{F}}(i)}
&\textbf{\textsf{F}}(\pi_1)\ar[d]^{\boldsymbol\omega_{\pi_1}}\ar[r]^{\textbf{\textsf{F}}(f)}
&\textbf{\textsf{F}}(\pi_2)\ar[d]^{\boldsymbol\omega_{\pi_2}}\\
\textbf{\textsf{G}}(\s)\ar[r]_{\textbf{\textsf{G}}(i)}
&\textbf{\textsf{G}}(\pi_1) \ar[r]_{\textbf{\textsf{G}}(f)}
&\textbf{\textsf{G}}(\pi_2)}
\end{equation*}
où $\textbf{\textsf{F}}(i)$ et $\textbf{\textsf{G}}(i)$ sont injectives et 
$\textbf{\textsf{F}}(f)$ et $\textbf{\textsf{G}}(f)$ surjectives 
car les foncteurs $\textbf{\textsf{F}}$ et $\textbf{\textsf{G}}$ sont exacts 
sur $\Ee(\BJ_\a,\bl_\a)$. 
Comme $\boldsymbol\omega_{\pi_1}$ et $\boldsymbol\omega_{\pi_2}$ 
sont des isomorphismes d'après la proposition \ref{RussSamProp},
le lemme du serpent implique que 
$\boldsymbol\omega_{\s}$ est un iso\-mor\-phisme. 
\end{proof}

\subsection{Changement de groupe}
\label{ChgtGp}
\label{MortonFullerton}

Reprenons les notations du paragraphe \ref{DiaComIndPar}. 
D'après la proposition \ref{BoyerDArgens}, 
pour toute représen\-ta\-tion irréductible $\pi$ de $\G$, 
on a $\Mm_{}(\pi)\neq0$ si et seulement si~:
\begin{equation}
\label{DeLaFolie}
\cusp(\pi)=
\sy{\rho\chi_1}+\dots+\sy{\rho\chi_n},
\quad
\chi_i:\G_m\to\R^{\times}\text{ non ramifié},
\quad
i\in\{1,\dots,n\},
\end{equation}
\ie si $\cusp(\pi)$ appartient à la classe d'inertie $\Om_{\rho,n}$
de $\sy{\rho}+\dots+\sy{\rho}=n\cdot\sy{\rho}$.
Selon la pro\-po\-si\-tion \ref{pptf1} et le théorème \ref{qptf}, 
et grâce à l'isomorphisme $\Psi_{\rho,n}$ de la proposition 
\ref{ExistenceHeckeProp} permet\-tant d'identifier $\Hh$ 
et $\Hh(n,\qr)$, le foncteur $\Mm_{}$ in\-duit une bijection~:
\begin{equation*}
\boldsymbol{\xi}_{\rho,n}:\Irr(\Om_{\rho,n})^\q\to\Irr(\Hh(n,\qr))
\end{equation*}
entre l'ensemble $\Irr(\Om_{\rho,n})^\q$ des représentations irréductibles 
de $\G$ de support cuspidal de la forme \eqref{DeLaFolie}
et l'ensemble des classes de $\Hh(n,\qr)$-modules à droite irréductibles. 

Pour $n=1$ en particulier, et si l'on identifie $\Irr(\Hh(1,\qr))$ et $\mult\R$, 
on a le lemme suivant, que l'on prouve comme dans \cite[\S4.2]{VSU0}. 

\begin{lemm}
\label{JustAbove}
Pour tout caractère non ra\-mi\-fié $\chi$ de $\G_m$, on a~:
\begin{equation*}
\boldsymbol{\xi}_{\rho,1}(\rho\chi)=\chi(\w_\l)^{-1}
\end{equation*}
où $\w_\l$ est l'élément de $\G_m$ défini par \eqref{pilam}. 
\end{lemm}

\begin{prop}
\label{blablablaBK}
Soit $\s$ une représentation irréductible de $\M$ contenant $\bl_\a$,
et soit $\pi$ une représentation irréductible de $\G$ contenant $\bl$. 
La multiplicité de $\pi$ dans le socle (respectivement, le cosocle) 
de $\ip_{(mn_1,\dots,mn_r)}(\s)$ est égale à 
la multiplicité de $\Mm(\pi)$ dans le socle (respectivement, le cosocle) 
de $\Mm(\ip_{(mn_1,\dots,mn_r)}(\s))$.  
\end{prop}

\begin{proof}
D'abord, si $\pi$ est isomorphe à une sous-représentation 
(respectivement, à un quo\-tient)
de $\ip(\s)=\ip_{(mn_1,\dots,mn_r)}(\s)$, alors $\Mm(\pi)$ 
est isomorphe à un sous-module 
(respectivement, à un quo\-tient) 
de $\Mm(\ip(\s))$ parce que le foncteur $\Mm$ 
est exact sur $\Ee(\BJ,\bl)$. 
Ensuite, un tel $\pi$ sous-représentation ou quo\-tient
de $\ip(\s)$ appartient à $\Irr(\Om_{\rho,n})^\q$,
donc $\Mm(\pi)$ est un module ir\-ré\-ductible. 
Ainsi $\Mm(\ip(\s))$ est sans $\Q$-torsion,
et le point 1 du théorème \ref{qptf} con\-duit au résultat voulu. 
\end{proof}

Fixons maintenant une extension finie $\F'$ de $\F$ dont le corps résiduel est 
de cardinal $\qr$.
On pose $\G'=\GL_n(\F')$. 
Plus généralement, on ajoutera un $'$ pour dési\-gner 
les objets corres\-pondant au cas où $\rho$ est le caractère trivial de $\F'^{\times}$.
Notons $\Iw'$ le sous-groupe d'Iwahori standard de $\G'$ 
et $\Mm_{}'$ le foncteur $\V\mapsto\V^{\Iw'}$ de $\Rr(\G')$ dans la 
catégorie des mo\-du\-les à droite sur $\Hh(\G',\Iw')$.
De façon analogue à ce qui précède, on obtient une bijection~:
\begin{equation*}
\boldsymbol{\xi}_{1_{\F^{\prime\times}},n}:\Irr(\Om_{1_{\F'^{\times}},n})^\q\to\Irr(\Hh(n,\qr))
\end{equation*}
entre l'ensemble des classes de représentations irréductibles
ayant des vecteurs non nuls invariants par $\Iw'$ et 
celui des classes de $\Hh(n,\qr)$-modules à droite 
irréductibles (voir l'exemple \ref{EmilPost}).

La composée~: 
\begin{equation}
\label{GG'n}
{\bf\Phi}_{\rho,n}^{}=\boldsymbol{\xi}_{1_{\F^{\prime\times}},n}^{-1}
\circ\boldsymbol{\xi}_{\rho,n}^{}:
\Irr(\Om_{\rho,n})^\q\to\Irr(\Om_{1_{\F'^{\times}},n})^\q
\end{equation}
est bijective.
Étudions sa compatibilité au support cuspidal. 
On obtient le lemme suivant par l'utilisation conjointe des
propositions \ref{RussSamProp} et \ref{blablablaBK}. 
On note $\Hh=\Hh(n,\qr)$ et $\Hh_\a=\Hh(\a,\qr)$
pour simplifier.

\begin{lemm}
\label{lemmeLPNTS}
Soient $\chi_1,\dots,\chi_n$ des caractères non ramifiés de $\G_m$. 
Alors $\boldsymbol{\xi}_{\rho,n}$ induit une bijection entre l'ensemble 
des classes de représentations irréductibles de support cuspidal égal à 
$\sy{\rho\chi_1}+\dots+\sy{\rho\chi_n}$
et l'ensemble des 
classes de modules irréductibles isomorphes à un sous-module de~: 
\begin{equation}
\label{LPNTS}
\Hom_{\Hh_{(1,\dots,1)}}(\Hh,\boldsymbol{\xi}_1(\rho\chi_1)\otimes
\dots\otimes\boldsymbol{\xi}_1(\rho\chi_n)).
\end{equation}
\end{lemm}

\begin{rema}
On a un résultat analogue en remplaçant sous-représentation et sous-mo\-du\-le 
par représentation quotient et module quotient.  
\end{rema}

\begin{prop}
\label{PhiPreserveSupp}
Soit $\pi$ une représentation irréductible dans $\Irr(\Om_{\rho,n})^\q$, 
dont on écrit le support cuspidal sous la forme \eqref{DeLaFolie}. 
Alors $\cusp({\bf\Phi}_{\rho,n}(\pi))=
{\bf\Phi}_{\rho,1}(\rho\chi_1)+\dots+{\bf\Phi}_{\rho,1}(\rho\chi_n)$.
\end{prop}

\begin{proof}
On peut supposer que $\pi$ est une sous-représentation de l'induite 
para\-bo\-li\-que $\rho\chi_1\times\dots\times\rho\chi_n$.
D'après le lemme \ref{lemmeLPNTS}, 
$\boldsymbol{\xi}_{\rho,n}(\pi)$ est un sous-module irré\-duc\-ti\-ble de
\eqref{LPNTS}. 
Si $\chi$ est un caractère non ramifié de $\mult\F$, on note $\chi'$ le 
caractère non ramifié de $\F'^{\times}$ prenant en une uniformisante de $\F'$
la même valeur que $\chi$ en une uniformisante de $\F$.
Ceci définit une bijection $\chi\mapsto\chi'$ entre 
caractères non ramifiés de $\G$ et caractères non ramifiés de $\G'$.
Le lemme suivant, qui décrit la compatibilité de ${\bf\Phi}_{\rho,n}$ à la torsion
non ramifiée, découle du lemme \ref{JustAbove}. 

\begin{lemm}
\label{MissMackenzie}
Pour toute représentation ir\-ré\-duc\-ti\-ble $\pi$ 
et tout caractère non ra\-mi\-fié $\chi$ de $\G$,
on a ${\bf\Phi}_{\rho,n}(\pi\chi)={\bf\Phi}_{\rho,n}(\pi)\chi'$.
\end{lemm}

Notons $\pi'$ la représentation ${\bf\Phi}_{\rho,n}(\pi)$. 
Compte tenu du lemme \ref{MissMackenzie} et de  
\eqref{ExistenceHecke}, on déduit de ce qui précède 
que $\boldsymbol{\xi}_{1_{\F^{\prime\times}},n}(\pi')$ est un sous-module irréductible de~:
\begin{equation*}
\Hom_{\Hh_{(1,\dots,1)}}(\Hh,\boldsymbol{\xi}_{1_{\F^{\prime\times}},1}(\chi'_1)
\otimes\dots\otimes\boldsymbol{\xi}_{1_{\F^{\prime\times}},1}(\chi'_n)).
\end{equation*}
D'après le lemme \ref{lemmeLPNTS}, on en déduit que $\pi'$ est une 
sous-représentation de $\chi'_1\times\dots\times\chi'_n$, ce qui termine la 
démonstration de la proposition \ref{PhiPreserveSupp}. 
\end{proof}

\begin{prop}
\label{MisMac}
Supposons que $\a=(n_1,n_2)$ et soit $\pi_i\in\Irr(\Om_{\rho,n_i})^\q$ pour $i=1,2$.
Alors~:
\begin{enumerate}
\item
$\pi_{1}\times\pi_{2}$ est irréductible si et seule\-ment si 
${\bf\Phi}_{\rho,n_1}(\pi_{1})\times{\bf\Phi}_{\rho,n_2}(\pi_{2})$ est 
irréductible.
\item
Si c'est le cas, alors on a l'égalité~:
\begin{equation*}
{\bf\Phi}_{\rho,n}(\pi_{1}\times\pi_{2})={\bf\Phi}_{\rho,n_1}(\pi_{1})\times{\bf\Phi}_{\rho,n_2}(\pi_{2}).
\end{equation*}
\end{enumerate}
\end{prop}

\begin{proof}
Pour $i=1,2$, on note $\pi_i'$ la représentation
${\bf\Phi}_{\rho,n_i}(\pi^{}_i)$. 
Si l'induite $\pi_{1}\times\pi_{2}$ est irréductible, alors elle appartient à 
$\Irr(\Om_{\rho,n})^\q$ et, par la proposition \ref{RussSamProp}, 
on a des isomorphismes de $\Hh$-modules irréductibles~:
\begin{eqnarray}
\label{BeckySharp}
\boldsymbol{\xi}_{\rho,n}(\pi_{1}\times\pi_{2})&\simeq&
\Hom_{\Hh_{(n_1,n_2)}}(\Hh,\boldsymbol{\xi}_{\rho,n_1}(\pi_1)
\otimes\boldsymbol{\xi}_{\rho,n_{2}}(\pi_2)),\\
\label{BeckySharp2}
\boldsymbol{\xi}_{1_{\F^{\prime\times}},n}(\pi'_{1}\times\pi'_{2})
&\simeq&\Hom_{\Hh_{(n_1,n_2)}}(\Hh,
\boldsymbol{\xi}_{1_{\F^{\prime\times}},n_1}(\pi'_1)\otimes
\boldsymbol{\xi}_{1_{\F^{\prime\times}},n_{2}}(\pi'_2)).
\end{eqnarray}
D'après la proposition \ref{ExistenceHeckeProp}, 
le membre de droite de \eqref{BeckySharp2}
correspond à celui de \eqref{BeckySharp}.
On déduit de la proposition \ref{EgaIndParCasIrrBisens} 
que $\pi'_{1}\times\pi'_{2}$ est irréductible, 
puis qu'on a une égalité entre 
$\boldsymbol{\xi}_{\rho,n}(\pi_{1}\times\pi_{2})$ et 
$\boldsymbol{\xi}_{1_{\F^{\prime\times}},n}(\pi'_{1}\times\pi'_{2})$.
\end{proof}

\begin{coro}
Reprenons les hypothèses de la proposition \ref{MisMac}.
Supposons que tous les sous-quotients irréductibles de 
$\pi_{1}\times\pi_{2}$ sont dans $\Irr(\Om_{\rho,n})^\q$, et que 
tous les sous-quotients ir\-ré\-ductibles de 
${\bf\Phi}_{\rho,n_1}(\pi_{1})\times{\bf\Phi}_{\rho,n_2}(\pi_{2})$ 
sont dans $\Irr(\Om_{1_{\F'^{\times}},n})^\q$. 
Alors~:
\begin{enumerate}
\item
Les représentations $\pi_{1}\times\pi_{2}$ et 
${\bf\Phi}_{\rho,n_1}(\pi_{1})\times{\bf\Phi}_{\rho,n_2}(\pi_{2})$ 
ont la même longueur.
\item
$\pi_{1}\times\pi_{2}$ est indécomposable si et seulement si
${\bf\Phi}_{\rho,n_1}(\pi_{1})\times{\bf\Phi}_{\rho,n_2}(\pi_{2})$ l'est. 
\end{enumerate}
\end{coro}

\begin{proof}
On fixe une suite de composition~:
\begin{equation*}
0=\V_0\subsetneq\V_1\subsetneq\dots\subsetneq\V_r=\pi_{1}\times\pi_{2}
\end{equation*}
où les $\V_i$ sont des sous-représentations de $\pi_{1}\times\pi_{2}$ telles 
que, pour $i\in\{0,\dots,r-1\}$, le quotient $\V_{i+1}/\V_i$ soit 
irréductible. 
Par hypothèse, ces sous-quotients sont dans $\Irr(\Om_{\rho,n})^\q$. 
On déduit de la proposition \ref{blablabla} que la longueur de 
$\Mm_{}(\pi_{1}\times\pi_{2})$ est égale à la longueur de 
$\pi_{1}\times\pi_{2}$.  
De façon analogue, la longueur de~:
\begin{equation}
\label{A5}
\Mm'({\bf\Phi}^{}_{\rho,n_1}(\pi_{1})\times{\bf\Phi}^{}_{\rho,n_2}(\pi_{2}))
\end{equation}
est égale à celle de 
${\bf\Phi}_{\rho,n_1}(\pi_{1})\times{\bf\Phi}_{\rho,n_2}(\pi_{2})$.  
Le résultat se déduit du fait que $\Mm_{}(\pi_{1}\times\pi_{2})$ et 
\eqref{A5}
sont isomorphes (voir la proposition \ref{ExistenceHeckeProp}) 
donc ont la même longueur.
Pour l'indécomposabilité, on procède de façon analogue en utilisant 
le point 1 du théorème \ref{qptf}.  
\end{proof}

\subsection{Le caractère $\nu_{\rho}$ associé à une représentation cuspidale}
\label{NURHO}

Soit $m\>1$ un entier, et soit $\rho$ une représentation 
ir\-ré\-duc\-ti\-ble cus\-pi\-da\-le de $\G=\G_{m}$. 
Au para\-gra\-phe \ref{InvRho}, on a associé à $\rho$ des invariants numériques 
$n(\rho),s(\rho),f(\rho)\>1$.
Soit $\nu$ le caractère non ramifié de $\G$ défini au paragraphe 
\ref{LeMepris}. 
On pose~:
\begin{equation*}
\nu_{\rho}=\nu^{s(\rho)},
\end{equation*}
qui ne dépend que de la classe d'inertie de $\rho$, et on pose~:
\begin{equation*}
\ZZ_{\rho}=\{\sy{\rho\nu_{\rho}^{i}}\ |\ i\in\ZZ\}.
\end{equation*}
On note~:
\begin{equation}
\label{ERHO}
\ee(\rho)
\end{equation}
l'ordre de $\qr$ dans $\R$ (égal à $+\infty$ si $\R$ est de caractéristique 
nulle). 
Les quantités $\qr$ et $\ee(\rho)$ ne dépendent 
que de la classe d'iner\-tie de $\rho$.
On prouve le résultat important suivant, qui justifie l'introduction 
de $\nu_{\rho}$.

\begin{prop}
\label{cuspidal}
Soit $\rho'$ une représentation irréductible cus\-pi\-da\-le 
de $\G_{m'}$, $m'\>1$.
Alors l'in\-duite $\rho\times\rho'$ est réductible si et seulement si 
$\rho'$ est isomorphe à $\rho\nu_\rho^{}$ ou $\rho\nu_{\rho}^{-1}$.
\end{prop}

\begin{rema}
Dans le cas complexe, le résultat est donné par \cite[Theorem 4.6]{VSU0}
dont la preuve consiste à se ramener au groupe $\GL_n(\F)$, $n\>1$ 
pour lequel le résultat est dû à Bernstein-Zelevinski.  
Dans le cas modulaire, nous pourrions suivre la même méthode de réduction 
au cas déployé 
et appliquer \cite[III.1.15]{Vig1}, qui est prouvé comme dans le cas 
complexe grâce à la théorie des dérivées.  
Nous préférons utiliser une méthode plus directe qui ne s'appuie pas sur la théorie 
des dérivées. 
\end{rema}


\begin{proof}
D'après le théorème \ref{ss}, il suffit de traiter le cas où $\rho$ et $\rho'$ 
sont iner\-tielle\-ment équivalentes, \ie que 
$\rho$ et $\rho'$ contiennent un même type simple maximal 
$(\J,\l)$.  
On forme la paire~:
\begin{equation*}
(\J_\M,\l_\M)=(\J\times\J,\l\otimes\l)
\end{equation*} 
avec 
$\M=\G\times\G$ puis un type semi-simple $(\BJ,\bl)$ 
comme au paragraphe \ref{DiaComIndPar} avec $n=2$. 
Écrivons $\rho'$ sous 
la forme $\rho\chi$ avec $\chi$ un caractère non ramifié de $\G$.

D'après la proposition \ref{EgaIndParCasIrrBisens},
la représentation induite
$\rho\times\rho\chi$ est réductible si et seulement si le $\Hh$-module 
$\Mm_\bl(\rho\times\rho\chi)$ est réductible.
D'après la proposition \ref{RussSamProp} et le lemme \ref{JustAbove}, 
ce $\Hh$-module est induit à partir du caractère 
$1\otimes\chi(\w_{\l})^{-1}$ de $\Hh_\M$. 
Il est de dimension $2$ sur $\R$. 
Il est donc réductible si et seulement s'il contient un caractère, 
ce qui, compte tenu de la description de $\Hh$ par générateurs et 
relations au paragraphe \ref{ChgtGp},
est le cas si et seulement si 
$\chi(\w_{\l})$ est égal à $\qr$ ou $\qr^{-1}$. 
Un calcul simple montre que la valuation de la norme réduite de 
$\w_{\l}\in\G$ est égale à $f(\rho)s(\rho)^{-1}$, ce dont on déduit 
l'égalité~: 
\begin{equation}
\label{Pekun}
\nu_\rho(\w_{\l})^{-1}=\qr.
\end{equation}
Le résultat se déduit du fait que, pour qu'un caractère non ramifié $\xi$ de $\G$ 
vérifie $\rho\xi\simeq\rho$, il faut et il suffit que $\xi(\w_{\l})=1$. 
\end{proof}

\begin{coro}
\label{Orleans}
On utilise les notations du paragraphe \ref{ChgtGp}. 
Soient $a,b\in\ZZ$ tels que $a\<b$. 
Alors $\boldsymbol{\xi}_n$ induit une bijection entre 
sous-représentations irréductibles de 
$\rho\nu_\rho^{a}\times\rho\nu_\rho^{a+1}\times\dots\times\rho\nu_\rho^b$ 
et sous-modules irré\-duc\-ti\-bles du module induit~:
\begin{equation*}
\Hom_{\Hh_{(1,\dots,1)}}(\Hh,\qr^a\otimes\qr^{a+1}\otimes\dots\otimes\qr^b).
\end{equation*}
\end{coro}

\begin{rema}
On a un résultat analogue en remplaçant sous-représentation et sous-mo\-du\-le 
par représentation quotient et module quotient.  
\end{rema}

\begin{proof}
On utilise les lemmes \ref{lemmeLPNTS} et \ref{JustAbove} 
et la formule \eqref{Pekun}.
\end{proof}

\begin{rema}
\label{SurLeCalculDeERho}
D'après la remarque \ref{CSProjcInd}, si 
$\ee(\rho)>1$, l'induite $\ind^\G_\J(\l)$ est 
projective pour tout type simple maximal $(\J,\l)$ 
contenu dans $\rho$. 
\end{rema}

On termine cette section sur la relation importante suivante, 
qui est nécessaire à la classification des 
représentations irréductibles cuspidales en fonction des 
supercuspidales dans \cite{MS12}. 

\begin{lemm}
\label{BienAise}
On a $\ee(\rho)={\rm card}\ \ZZ_\rho$. 
\end{lemm}

\begin{proof}
On suppose que $\R$ est de caractéristique $\ell$ non nulle 
et on note $e$ l'ordre de $q$ dans $\mult\FF_{\ell}$. 
Si $a,b$ sont des entiers $\>1$, on note $(a,b)$ leur plus grand 
diviseur commun. 
Pour alléger les notations, on pose $n=n(\rho)$, $s=s(\rho)$ et 
$f=f(\rho)$. 

\begin{lemm}
\label{CalculPGCD}
On a $(e,n(e,s))=(e,ns)$.
\end{lemm}

\begin{proof}
Soit $u\>1$ un entier divisant $e$ et $ns$. 
Écrivons $u_1=(u,n)$ et $u_2=u/(u,n)$. 
On a donc $u=u_1u_2$ avec $u_1$ divisant $n$ et $u_2$ divisant $s$. 
Comme $u$ divise $e$, l'entier $u_2$ divise $(e,s)$, donc $u$ divise $n(e,s)$, 
ce qui prouve le résultat attendu. 
\end{proof}

Pour calculer le cardinal de $\ZZ_\rho$, on fait opérer sur la classe 
inertielle $\Om_\rho$ le groupe cyclique en\-gen\-dré par $\nu_\rho$. 
Ce groupe cyclique est d'ordre $e/(e,s)$. 
On obtient donc~:
\begin{equation*}
{\rm card}\ \ZZ_\rho=
\frac{e/(e,s)}{\(n,e/(e,s)\)}=\frac{e}{(e,n(e,s))}
=\frac{e}{(e,ns)},
\end{equation*}
la dernière égalité provenant du lemme \ref{CalculPGCD}. 
Par définition, l'entier $\ee(\rho)$ est égal à $e/(e,f)$. 
Le résultat provient alors de la formule \eqref{Wigmore} et du fait que 
$e$ est premier à $\ell$. 
\end{proof}

\section{Le foncteur $\KM$}
\label{NaisKM}

Dans cette section, nous introduisons un outil technique important 
permettant de lier la théo\-rie des représentations 
de $\G=\G_m$, $m\>1$ 
à celle d'un groupe linéaire général sur un corps fini de caractéristique 
$p$. 

Soit $[\La,n,0,\b]$ une strate sim\-ple de $\A=\A_m$, et soit $\k$ une 
$\b$-extension d'un caractère simple $\t\in\Cc(\La,0,\b)$. 
Elle définit un foncteur exact~:
\begin{equation*}
\KM=\KM_\k:\pi\mapsto\Hom_{\J^1}(\k,\pi)
\end{equation*}
de $\Rr_\R(\G)$ dans la catégorie des représentations de $\J/\J^1$. 
Dans le cas où le sous-groupe paraho\-rique $\U(\La)\cap\mult\B$ est 
maximal, le quotient $\J/\J^1$ peut être identifié (non canoniquement) 
au grou\-pe $\GB=\GL_{m'}(\ff_{\D'})$ moyennant le choix d'un 
isomorphisme $\Phi$ comme en \eqref{MonEpousee}.
Ce foncteur permet de ramener à $\G$ certaines propriétés connues 
des représentations de $\GB$.
Cette section est con\-sacrée à l'étude des propriétés de ce 
foncteur, qui est utilisé dans \cite{MS12}. 

Ce foncteur a déjà été utilisé dans \cite{Vig1} pour étudier les 
représentations de Steinberg générali\-sées de $\GL_n(\F)$
(voir par exemple le lemme 5.13 de \cite{Vig1}),
et de façon plus systématique 
dans \cite{SZ} pour définir une stratification de la catégorie des 
représentations complexes de $\GL_n(\F)$ affinant la décomposition de 
Bernstein.

Au paragraphe \ref{Brehier}, nous établissons des conditions d'annulation 
importantes du foncteur $\KM$, énoncées en termes d'endo-classes. 
Dans les paragraphes \ref{ComPIP} et \ref{Holveg}, 
nous prouvons la compatibi\-lité de $\KM$ 
à l'induction et à la restriction paraboliques. 

Pour le groupe $\GL_n(\F)$, la compatibilité de $\KM$ 
à l'induction parabolique est prouvée dans 
\cite{SZ} dans le cas complexe (voir \textit{ibid.}, Proposition 5.7)
mais la preuve qui y est donnée ne s'applique pas au cas modulaire
(car elle repose sur certaines équivalences de catégories provenant 
des types simples).
Dans le cas modulaire, elle est prouvée dans \cite{Vig1} 
dans un cas particulier (voir \textit{ibid.}, Lemme 5.12). 

A notre connaissance, la compatibilité de $\KM$ 
à la restriction parabolique pour $\GL_n(\F)$ n'était 
connue ni dans le cas complexe, ni dans le cas modulaire. 

\subsection{Définition} 
\label{PremPara}

Dans toute cette section, nous fixons une strate simple 
$[\La_{{\rm max}},n_{{\rm max}},0,\b]$ de $\A$
telle que l'inter\-sec\-tion $\U(\La_{{\rm max}})\cap\mult\B$ 
soit un sous-groupe compact maximal de $\mult\B$, 
et un caractère simple~:
\begin{equation}
\label{CSM}
\tmax\in\Cc(\La_{{\rm max}},0,\b).
\end{equation}
Un tel caractère simple de $\G$ est dit \textit{maximal}. 

\begin{rema}
\label{IanMcEwan}
Un caractère simple est maximal si et seulement son induite compacte à $\G$ 
possède un quotient irréductible cuspidal. 
En effet, si $\tmax$ est un caractère simple maximal de $\G$, l'in\-dui\-te compacte à 
$\G$ d'un type simple contenant $\tmax$ possède un quotient irréductible cus\-pi\-dal 
d'après la proposition \ref{MartinDuGard}. 
Inversement, si $\t$ est un caractère simple apparaissant dans une 
représentation irréductible cuspidale $\rho$, les preuves des théo\-rèmes 
\ref{TheZer} et \ref{ThePos} montrent que $\t$ est maximal 
et que $\rho$ contient un type simple maximal de la forme $\k\otimes\s$
avec $\k$ une $\b$-ex\-ten\-sion de $\t$.
\end{rema}

Ce caractère simple maximal étant fixé, 
on fixe comme dans le paragraphe \ref{Pontesprit} un isomorphisme de
$\E$-algèbres $\Phi:\B\to\Mat_{m'}(\D')$
envoyant $\U(\La_{{\rm max}})\cap\mult\B$ 
sur le sous-groupe compact maximal standard $\GL_{m'}(\Oo_{\D'})$.  
Fixons aussi une $\b$-ex\-ten\-sion $\kmax$ de $\tmax$ et posons~:
\begin{equation*}
\J^{}_{{\rm max}}=\J(\b,\La^{}_{{\rm max}}),
\quad
\J^{1}_{{\rm max}}=\J^1(\b,\La^{}_{{\rm max}}),
\quad 
\GB=\J^{}_{{\rm max}}/\J^{1}_{{\rm max}}.
\end{equation*}
Le groupe $\GB$ est canoniquement isomorphe à 
$(\U(\La_{{\rm max}})\cap\mult\B)/(\U^1(\La_{{\rm max}})\cap\mult\B)$ 
et sera iden\-tifié au grou\-pe $\GL_{m'}(\ff_{\D'})$ grâce à $\Phi$.
Étant donnée une représentation lisse $(\pi,\V)$ de $\G$, on pose~:
\begin{equation*}
\label{DefKaMax}
\V(\kmax)=\Hom_{\J^1_{{\rm max}}}(\kmax,\V)
\end{equation*}
que l'on munit de l'action de $\J_{{\rm max}}$ définie, 
pour $x\in\J_{{\rm max}}$ et $f\in\V(\kmax)$,
par la formule~:
\begin{equation*}
\label{ACTION}
x\cdot f=\pi(x)\circ f\circ\kmax(x)^{-1}.
\end{equation*}
Pour cette action, $\J^{1}_{{\rm max}}$ opère trivialement, 
de sorte que cette formule définit une représentation 
de $\GB$ sur $\V(\kmax)$, que l'on note $\pi(\kmax)$. 
On définit ainsi un foncteur~:
\begin{equation*}
\label{Georgie}
\KM=\KM_{\kmax,\Phi}:\pi\mapsto\pi(\kmax)
\end{equation*}
de $\Rr(\G)$ dans la catégorie $\Rr(\GB)$ des $\R$-représentations de $\GB$. 
Comme $\J^1_{{\rm max}}$ est un pro-$p$-groupe, c'est un foncteur exact, 
et comme toute représentation lisse ir\-ré\-ductible de $\G$ est ad\-mis\-sible, 
il préserve le fait d'être de longueur finie.

\begin{rema}
Le foncteur $\KM$ dépend des choix de $\kmax$ et de $\Phi$. 
D'abord, choisissons une autre $\b$-extension~:
\begin{equation*}
\kmax'=\kmax^{}\otimes(\chi\circ\N_{\B/\E})
\end{equation*}
où $\chi$ est un carac\-tère de 
$\mult\Oo_\E$ trivial sur $1+\p_\E$ comme dans \eqref{StHonore}. 
Si $\pi$ est une représentation de $\G$, les 
représentations $\pi(\kmax^{})$ et $\pi(\kmax')$ sont tordues l'une de l'autre par 
le caractère $\chi\circ\nn\circ\det$ où 
$\nn$ est la norme de $\kk_{\D'}$ sur $\kk_\E$.
Si l'on fixe un autre isomorphisme 
$\Phi':\B\to\Mat_{m'}(\D')$ envoyant $\U(\La_{{\rm max}})\cap\mult\B$ sur 
$\GL_{m'}(\Oo_{\D'})$, 
le théorème de Skolem-Noether entraîne qu'il est conjugué à $\Phi$ 
sous le normalisateur de $\GL_{m'}(\Oo_{\D'})$. 
Ceci induit sur $\pi(\kmax)$ un automorphisme 
de conjugaison par un élément du produit semi-direct 
$\Ga\rtimes\GB$ avec $\Ga=\Gal(\ff_{\D'}/\ff_{\E})$.  
\end{rema}

\subsection{Conditions d'annulation} 
\label{Brehier}

On étudie maintenant des conditions d'annulation du foncteur $\KM$.
On commence par le cas simple suivant.

\begin{lemm}
\label{DeliceDOrient}
Soit $\rho$ une représentation irréductible cuspidale de $\G$.
\begin{enumerate}
\item 
Si $\rho$ ne contient pas $\tmax$, alors $\KM(\rho)=0$.
\item 
Sinon, il existe une représentation 
irréductible cuspidale $\s$ de $\GB$ telle que $\rho$ contien\-ne 
le type simple maximal $\kmax\otimes\s$, et on a un isomorphisme
de représentations de $\GB$~:
\begin{equation}
\label{KateVava}
\KM(\rho)\simeq\s\oplus\s^{\Fr}\oplus\dots\oplus\s^{\Fr^{b(\rho)-1}},
\end{equation}
où $b(\rho)\>1$ est l'invariant associé à $\rho$ au paragraphe 
\ref{InvRho}
et $\Fr$ un générateur du groupe de Galois de $\kk_{\D'}$ sur $\kk_{\E}$.  
\end{enumerate}
\end{lemm}

\begin{rema}
Dans le cas où $\D=\F$, on a toujours $b(\rho)=1$
et $\KM(\rho)=\s$ dans le cas 2.
\end{rema}

\begin{proof}
D'après la remarque \ref{IanMcEwan}, toute 
représentation irréductible cuspidale $\rho$ de $\G$ contenant $\tmax$ 
contient un type simple maximal $\l$ de la forme $\kmax\otimes\s$
où $\s$ est une représenta\-tion irréductible cuspidale de $\GB$.
On fixe un type simple maximal étendu $\nl$ pro\-longeant 
$\l$ tel que $\rho$ soit isomorphe à l'induite compacte de $\nl$ à $\G$.
La formule \eqref{KateVava} 
est alors une conséquence du lemme \ref{LemCritV}.
\end{proof}

\begin{coro}
\label{CaTombeTresTresBien}
Soient $\rho$ et $\rho'$ des représentations irréductibles cuspidales 
de $\G$ con\-tenant $\tmax$.
Alors $\KM(\rho)$ et $\KM(\rho')$ sont des représentations de $\GB$ 
iso\-morphes si et seulement si $\rho$ et $\rho'$ sont inertiellement 
équi\-va\-lentes. 
\end{coro}

\begin{proof}
Étant donné un facteur irréductible $\s$ de $\KM(\rho)$, la représentation $\rho$ 
contient le type simple maxi\-mal $\l=\kmax\otimes\s$. 
De façon analogue, la représentation $\rho'$ contient le type simple maximal 
$\l'=\kmax\otimes\s'$ où $\s'$ est un facteur irréductible de $\KM(\rho')$.  

Si $\KM(\rho)$ et $\KM(\rho')$ sont iso\-morphes, on peut choisir $\s'=\s$, 
ce qui implique que $\rho$ et $\rho'$ contiennent toutes deux $\l$, donc 
qu'elles sont inertiellement équi\-va\-lentes.

Inversement, si $\rho$ et $\rho'$ sont inertiellement équi\-va\-lentes, elles 
contiennent toutes les deux $\l$ et le résultat s'ensuit. 
\end{proof}

La question de savoir si un caractère simple apparaît ou non dans 
une représentation irréduc\-tible est liée à la notion d'endo-classe \cite{BSS1}.  
Notons~: 
\begin{equation*}
\label{endoclass}
\Tmax
\end{equation*}
l'endo-classe associée à $\tmax$.
Si $\rho$ est une représentation irréductible cuspidale, elle contient 
un type simple maximal $(\J,\l)$.
L'endo-classe du caractère simple contenu dans $\l$ ne dépend pas du choix 
de $(\J,\l)$ d'après le théorème \ref{TheZerPos}~:
on l'appelle l'endo-classe de $\rho$. 

\begin{prop}
\label{Cavalcanti}
\label{OnaeuChaud}
Soient $\rho_1,\dots,\rho_n$ des représentations irréductibles cuspidales. 
Supposons que l'induite $\rho_1\times\dots\times\rho_n$ admette un 
sous-quotient irréductible contenant $\tmax$.
Alors $\rho_1,\dots,\rho_n$ sont toutes d'endo-classe $\Tmax$. 
\end{prop}

\begin{proof}
Supposons d'abord que $\R$ est de caractéristique nulle.
Soit $\pi$ un sous-quo\-tient irréductible de $\rho_1\times\dots\times\rho_n$ 
contenant $\tmax$.
Fixons une strate sim\-ple $[\La,n_{\La},0,\b]$ de $\A$ telle que~:
\begin{enumerate}
\item 
$\U(\La)\subseteq\U(\La_{{\rm max}})$~;
\item
$\pi$ contient le transfert $\t\in\Cc(\La,0,\b)$ 
de $\t_{{\rm max}}$~;
\end{enumerate}
et telle que le sous-groupe parahorique 
$\U(\La)$ soit minimal pour ces propriétés. 
D'après le corollaire \ref{LAmourLApresMidi}, 
il existe une représentation irréductible $\xi$ de 
$\J=\J(\b,\La)$ triviale sur $\J^1=\J^1(\b,\La)$ telle que 
$\pi$ contient $\k\otimes\xi$ où 
$\k$ est une $\b$-extension de $\t$.
Grâce à la proposition \ref{isc} et à la minimalité de $\U(\La)$, 
la représentation $\xi$ 
vue comme une représentation de $\J/\J^1$ (identifié à un 
groupe réductif fini) est cuspidale. 
Ainsi $\pi$ contient un type semi-simple homogène $(\BJ,\bl)$ 
contenant $\t$.  
Selon le théo\-rème principal de \cite{SeSt2}
(voir aussi \textit{ibid.}, proposition 8.1),
ce type semi-simple est un type pour le bloc de Bernstein 
cor\-res\-pondant à la classe d'équivalence inertielle de
$\rho_1\otimes\dots\otimes\rho_n$, 
et les représentations cuspidales
$\rho_1,\dots,\rho_n$ sont toutes d'endo-classe $\Tmax$. 

Supposons maintenant que $\R$ est de caractéristique non nulle $\ell$. 
Quitte à tordre  $\rho_1,\dots,\rho_n$ par des caractères non ramifiés 
convenables, 
on peut supposer qu'elles sont 
toutes définies sur $\flb$  (voir \cite[II.4]{Vig1}).
On peut donc supposer que $\R=\flb$. 

Pour chaque $i\in\{1,\dots,n\}$, fixons une $\qlb$-représentation 
irréductible cuspidale $\tilde\rho_i$ telle que $\rho_i$ apparaisse dans 
$\r_\ell(\tilde\rho_i)$ comme dans la proposition \ref{RelCusp}. 
Par hypothèse, $\tmax$ apparaît comme un sous-quotient de la res\-tric\-tion 
de $\rho_1\times\dots\times\rho_n$ à $\H^1_{{\rm max}}$. 
Il existe donc un sous-quotient irréductible $\tilde\d$ de la res\-tric\-tion de 
$\tilde\rho_1\times\dots\times \tilde\rho_n$ à $\H^1_{{\rm max}}$ dont 
la réduction mod $\ell$ contient $\tmax$.

\begin{lemm}
\label{Metro}
Soit $\pi$ une représentation irréductible d'un $p$-groupe fini $\H$
telle que $\r_\ell(\pi)$ possède des vecteurs $\H$-invariants non nuls.
Alors $\pi$ est le carac\-tère trivial de $\H$.
\end{lemm}

\begin{proof}
Raisonnons par récurrence sur le cardinal de $\H$.
Si $\H$ est d'ordre $p$, il est cyclique et $\pi$ est un caractère.  
Les valeurs prises par $\pi$ sont des racines $p$-ièmes de l'unité,
et ce sont aussi des racines de l'unité d'ordre une puissance de $\ell$
puisque $\r_{\ell}(\pi)$ est trivial.
On en déduit que $\pi$ est trivial. 

Supposons maintenant que $\H$ est d'ordre strictement supérieur à $p$
et notons $\V$ l'espace de $\pi$. 
Comme $\H$ est réso\-lu\-ble, il possède un sous-groupe distingué 
$\H'$ d'indice $p$.
La restriction de $\V$ à $\H'$ admet un facteur irréductible 
$\W$ dont la réduction mod $\ell$ possède des 
vecteurs $\H$-invariants non nuls.
Par hypothèse de récurrence, $\W$ est de dimension $1$ et 
$\H'$ agit dessus trivialement.
Il existe donc un $\H'$-homomorphisme non trivial du caractère 
trivial de $\H'$ vers la restriction de $\V$ à $\H'$.
Par réciprocité de Frobenius, il existe un $\H$-homomorphisme 
non trivial~:
\begin{equation*}
\ind^{\H}_{\H'}(1)\to\V.
\end{equation*}
Puisque $\V$ est irréductible, cet homomorphisme est surjectif,
donc $\V$ est un facteur direct de $\ind^{\H}_{\H'}(1)$, ce dont on déduit 
que $\V$ est une représentation irréductible de $\H$ triviale sur $\H'$.  
Puisque $\H/\H'$ est cyclique, $\V$ est de dimension $1$.
On termine comme dans le cas où $\H$ est d'ordre $p$.
\end{proof}

Notons $\tilde\t_{{\rm max}}$ le relèvement de $\tmax$.  
En appliquant le lemme \ref{Metro} à $\tilde\d\tilde\t_{{\rm max}}^{-1}$, 
on voit que $\tilde\d$ est égal à $\tilde\t_{{\rm max}}$.  
Il existe donc un sous-quotient irréductible $\tilde\pi$ de 
$\tilde\rho_1\times\dots\times \tilde\rho_n$ contenant 
$\tilde\t_{{\rm max}}$. 
D'après le cas de caractéristique nulle traité plus haut, 
les représentations cuspidales $\tilde\rho_1,\dots,\tilde\rho_n$ 
sont toutes d'endo-classe $\Tmax$. 
En réduisant mod $\ell$, on obtient le résultat voulu. 
\end{proof}

\begin{prop}
\label{EhBienVoila}
Soient $\rho_1,\dots,\rho_n$ des représentations irréductibles cuspidales 
d'endo-clas\-se $\boldsymbol{\Theta}_{{\rm max}}$ dont la somme des degrés 
vaut $m$. 
Soit $\pi$ un sous-quotient irréductible de $\rho_1\times\dots\times\rho_n$.  
Écrivons~:
\begin{equation*}
\cusp(\pi) = \sy{\tau_1}+\dots+\sy{\tau_r}
\end{equation*}
le support cuspidal de $\pi$.
Alors $\tau_1,\dots,\tau_r$ sont toutes d'endo-classe $\Tmax$ et $\KM(\pi)$
est non nulle.
\end{prop}

\begin{proof}
Pour tout $k\in\{1,\dots,r\}$, soit $m_k$ le degré de $\tau_k$
et posons $\a=(m_1,\dots,m_r)$. 
Quitte à renuméroter $\tau_1,\dots,\tau_r$, 
le mo\-du\-le de Jacquet $\rp_\a(\rho_1\times\dots\times\rho_n)$ a 
pour sous-quotient la représentation $\tau_1\otimes\dots\otimes\tau_r$. 
Soit $k\in\{1,\dots,r\}$.
D'après le lemme géo\-mé\-trique \cite[2.6]{MSu}, il y a des entiers 
$i_1,\dots,i_s\in\{1,\dots,n\}$ tels que $\tau_k$ est un
sous-quotient de l'induite parabolique $\rho_{i_1}\times\dots\times\rho_{i_s}$. 
On déduit de la proposition \ref{OnaeuChaud} que l'endo-classe de $\tau_k$ est 
égale à $\Tmax$.
Ceci prouve la première partie de la proposition. 

Il nous reste à prouver que $\pi$ contient le caractère simple $\tmax$.
Soit $(\J_{\M},\l_{\M})$ un type simple maximal du sous-groupe de Levi standard 
$\M=\M_\a$ contenu dans $\tau_1\otimes\dots\otimes\tau_r$. 
D'après la pro\-po\-sition \ref{BoyerDArgens}, la représentation $\pi$ contient un type 
semi-simple $(\BJ,\bl)$ qui est une paire couvrante de $(\J_{\M},\l_{\M})$, 
et qui est homogène car $\tau_1,\dots,\tau_r$ ont la 
même endo-classe $\Tmax$.
Ce type semi-simple homogène est construit à partir d'une 
strate simple $[\La,n_\La,0,\b]$, et $\bl$ contient le caractère simple 
$\t\in\Cc(\La,0,\b)$ transfert de $\tmax$. 
D'après la proposition \ref{isc}, la représentation $\pi$ contient 
le transfert $\t'\in\Cc(\La',0,\b)$ pour toute strate simple $[\La',n_{\La'},0,\b]$ 
telle que $\U(\La)\subseteq\U(\La')$. 
Si l'on choisit $\La'$ de sorte que le sous-groupe parahorique 
$\U(\La')\cap\mult\B$ est maximal, $\La'$ est 
(à dilatation et translation près)
conjuguée à $\La_{{\rm max}}$ sous $\G$ et on en déduit que $\pi$ contient 
le caractère simple $\tmax$.
\end{proof}

\subsection{Compatibilité à l'induction parabolique}
\label{ComPIP}

Soit $\a=(m_1,\dots,m_r)$ une famille d'entiers $\>1$ de somme 
$m$, ce qui définit une décomposition~:
\begin{equation}
\label{DecAlpha}
\D^m=\D^{m_1}\oplus\dots\oplus\D^{m_r}
\end{equation}
en $\D$-espaces vectoriels à droite. 
On a fixé au paragraphe \ref{PremPara}
une strate sim\-ple $[\La_{{\rm max}},n_{{\rm max}},0,\b]$ de $\A$. 
Suppo\-sons que $\E=\F(\b)$ stabilise la décomposition \eqref{DecAlpha}, 
\ie que $\E$ est contenu dans la sous-algè\-bre dia\-go\-nale 
$\A_{m_1}\times\dots\times\A_{m_r}\subseteq\A$, et que $\La_{{\rm max}}$ 
soit conforme à \eqref{DecAlpha} au sens de \cite[Définition 1.13]{SeSt1}. 
On pose $\M=\M_\a$, $\N=\N_\a$ et $\P=\P_\a$.

Pour chaque $i\in\{1,\dots,r\}$, 
on pose $\La_{{\rm max},i}=\La_{{\rm max}}\cap\D^{m_i}$.
C'est une suite de réseaux de $\D^{m_i}$ qui définit une strate simple 
$[\La_{{\rm max},i},n_{{\rm max}},0,\b]$ de $\A_{m_i}$, 
et le transfert $\t_{{\rm max},i}\in\Cc(\La_{{\rm max},i},0,\b)$ de $\tmax$ 
est un carac\-tè\-re simple maximal.
Notons aussi $\B_i$ le centralisateur de $\E$ dans $\A_{m_i}$. 

Le produit $\B_1\times\dots\times\B_r$ 
s'identifie à une sous-$\E$-algèbre de $\B$. 
La restriction de $\Phi$ à $\B_i$ définit un isomorphisme de $\E$-algèbres
$\Phi_i$ de $\B_i$ dans $\Mat_{m'_i}(\D')$
où $m'_i$ est l'entier associé à $m_i$ par \eqref{MPRIME}. 
Fixons une strate simple $[\La,n_\La,0,\b]$ telle que~:
\begin{enumerate}
\item
pour chaque $i\in\{1,\dots,r\}$, la suite $\La\cap\D^{m_i}$ est dans la classe 
affine de $\La_{{\rm max},i}$~;
\item 
$\U(\La)\subseteq\U(\La_{{\rm max}})$ et $\U(\La)\cap\mult\B$ est le
sous-groupe parahorique standard de $\mult\B$
corres\-pon\-dant à la famille $(m'_1,\dots,m'_r)$.
\end{enumerate}

Soit $\t\in\Cc(\La,0,\b)$ le trans\-fert de $\tmax$ et notons $\k$ le 
trans\-fert de $\kmax$ en une $\b$-extension de $\t$.
La représentation de $\J(\b,\La)\cap\M$ sur l'espace des 
vec\-teurs $\J(\b,\La)\cap\N$-invariants de $\k$, notée $\k_{{\rm max},\a}$, 
est de la forme~:
\begin{equation*}
\k_{{\rm max},\a}=\k_{{\rm max},1}\otimes\dots\otimes\k_{{\rm max},r}
\end{equation*}
où $\k_{{\rm max},i}$ est une $\b$-extension de $\t_{{\rm max},i}$, 
pour chaque $i$.
On obtient un foncteur $\KM_i=\KM_{\k_{{\rm max},i},\Phi_i}$ de 
$\Rr(\G_{m_i})$ dans $\Rr(\GB_i)$, où 
$\GB_i=\J^{}_{{\rm max},i}/\J^{1}_{{\rm max},i}$ 
est iden\-ti\-fié à $\GL_{m'_i}(\ff_{\D'})$. 
On pose aussi~:
\begin{eqnarray*}
\J^{}_{{\rm max},\a}&=&\J^{}_{{\rm max},1}\times\dots\times\J^{}_{{\rm max},r},\\
\J^1_{{\rm max},\a}&=& 
\J^{1}_{{\rm max},1}\times\dots\times\J^{1}_{{\rm max},r},\\
\MB&=&\J^{}_{{\rm max},\a}/\J^{1}_{{\rm max},\a},
\end{eqnarray*}
ce dernier étant identifié au sous-groupe de Levi standard 
$\GB_1\times\dots\times\GB_r$ de $\GB$.
On définit un foncteur~:
\begin{equation*}
\KM_{\a}:\pi\mapsto\Hom_{\J^1_{{\rm max},\a}}(\k_{{\rm max},\a},\pi)
\end{equation*}
de $\Rr(\M)$ dans $\Rr(\MB)$.
On a le résultat suivant. 

\begin{prop}
\label{SZ}
Pour $i\in\{1,\dots,r\}$, soit $\pi_i$ une représentation irréductible 
de $\G_{m_i}$. 
Il y a un isomorphisme de représentations de $\GB$~:
\begin{equation*}
\KM(\pi_1\times\dots\times\pi_r)\simeq
\KM_{1}(\pi_1) \times\dots\times \KM_{r}(\pi_r).
\end{equation*}
\end{prop}

\begin{proof}
Le résultat est vrai dans le cas où $\R$ est de caractéristique nulle~: 
la preuve de Schneider et Zink \cite[Proposition 5.7]{SZ} est encore valable. 
On peut donc supposer que $\R$ est de caractéristique $\ell$ non nulle. 
Posons $\pi=\pi_1\otimes\dots\otimes\pi_r$.
Par la suite, $\ip_\a$ désignera aussi bien l'induction parabolique de 
$\M$ à $\G$ le long de $\P$ que l'induc\-tion parabolique de 
$\MB$ à $\GB$ le long de 
$\PB=(\U(\La)\cap\mult\B)\J^1_{{\rm max}}/\J^1_{{\rm max}}$.

\begin{lemm}
\label{PIganiol}
Pour toute représentation $\V$ de $\M$, 
il existe un homomorphisme injectif~:
\begin{equation}
\label{AuFait!}
\k_{{\rm \max}}\otimes\IA_{\a}(\KM_{\a}(\V))\to\ip_\a(\V)
\end{equation}
de représentations de $\J_{{\rm max}}$.
\end{lemm}

\begin{proof}
On reprend la première partie de la preuve de \cite[III.5.12]{Vig1}.
\end{proof}

En appliquant le foncteur exact $\KM$, on déduit du lemme \ref{PIganiol}, 
pour toute représentation $\V$ de $\M$, un homomorphisme injectif~:
\begin{equation}
\label{AuFait!KM}
\IA_{\a}(\KM_{\a}(\V))\to\KM(\ip_\a(\V))
\end{equation}
de représentations de $\GB$.
Dans le cas où $\V=\pi$, l'homomorphisme injectif \eqref{AuFait!KM} s'écrit~:
\begin{equation}
\label{skamp}
\KM_{1}(\pi_1)\times\dots\times\KM_{r}(\pi_r)\to
\KM(\pi_1\times\dots\times\pi_r)
\end{equation}
et on va prouver que c'est un isomorphisme. 

\begin{lemm}
\label{RedKM}
\label{LaissezMoiEcrireNimporteQuoi}
Soit $\tilde\k_{{\rm max}}$ une $\b$-extension relevant $\kmax$,
et soit $\tilde\KM$ le foncteur associé.
Pour toute $\qlb$-représentation entière $\tilde\pi$ de 
longueur finie de $\G$, on a
$\r_{\ell}(\sy{\tilde\KM(\tilde\pi)})=\sy{\KM(\r_{\ell}(\tilde\pi))}$.
\end{lemm}

\begin{proof}
Il suffit de le prouver dans le cas où $\tilde\pi$ est irréductible.
Soient $\mathfrak{k}_{{\rm max}}$ une struc\-ture entière de 
$\kmax$ et $\ll$ une structure entière de $\tilde\pi$.
Alors $\Hom_{\J^1_{{\rm max}}}(\mathfrak{k}_{{\rm max}},\ll)$ est une 
struc\-ture entière de $\KM(\tilde\pi)$.
\end{proof}

\begin{lemm}
\label{Berg}
Si $\pi$ est cuspidale, alors \eqref{skamp} est un isomorphisme. 
\end{lemm}

\begin{proof}
Quitte à tordre $\pi$ par un caractère non ramifié de $\M$, 
on peut supposer que son caractère central est à valeurs dans $\flb$, 
donc que $\pi$ est définie sur $\flb$. 
Il suffit donc de prouver le résultat dans le cas où $\R=\flb$, ce que 
l'on suppose désormais. 

D'après la proposition \ref{RelCusp}, il y a une $\qlb$-représentation 
irréductible cuspidale entière $\tilde\pi$ de $\M$ telle que 
$\r_{\ell}(\tilde\pi)\>\sy{\pi}$. 
Comme le résultat est valable en caractéristique $0$, on a un isomorphisme~: 
\begin{equation}
\label{AuFait!KMqlb}
\IA_{\a}(\tilde\KM_{\a}(\tilde\pi))\to \tilde\KM(\ip_\a(\tilde\pi))
\end{equation}
de $\qlb$-représentations de $\GB$.
La représentation $\tilde\pi$ est de la forme 
$\tilde\pi_1\otimes\dots\otimes\tilde\pi_r$.
Pour chaque $i$, la représentation $\tilde\pi_i$ est irréductible cuspidale 
entière,
et d'après le théorème \ref{RedCusp} il existe un entier $a_i\>1$ tel que 
$\r_\ell(\tilde\pi_i)$ soit égale à la somme des $\pi_i\nu^k$ pour 
$k\in\{0,\dots,a_i-1\}$. 
Compte tenu du lemme \ref{RedKM} et du fait que $\r_\ell$ com\-mute à 
l'induction parabolique, la réduction mod $\ell$ du membre de gauche de 
\eqref{AuFait!KMqlb} donne l'égalité~:
\begin{equation}
\label{Chinchine}
\sum\limits_{\boldsymbol{k}}
\KM_{1}(\pi_1\nu^{k_1})\times\dots\times\KM_{r}(\pi_r\nu^{k_r})
=
\sum\limits_{\boldsymbol{k}}
\KM(\pi_1\nu^{k_1}\times\dots\times\pi_r\nu^{k_r})
\end{equation}
dans le groupe de Grothendieck $\Gg(\G,\flb)$, 
où $\boldsymbol{k}$ décrit les fa\-mil\-les d'entiers $(k_1,\dots,k_r)$ tels que 
$k_i\in\{0,\dots,a_i-1\}$.  
Comme on a un morphisme injectif \eqref{skamp} entre représentations de 
dimension finie de $\GB$ non seulement 
pour $\pi$ mais plus généralement pour $\pi$ tordue par n'importe quel 
carac\-tère de $\G$, et 
compte tenu de \eqref{Chinchine}, on déduit que \eqref{skamp} 
est un iso\-mor\-phis\-me de représentations de $\GB$.
\end{proof}

Supposons maintenant que $\pi$ est irréductible mais pas cuspidale, 
et prouvons que \eqref{skamp} est un iso\-mor\-phis\-me par 
récurrence sur la dimension de $\M$.

Il y a un sous-groupe de Levi standard $\M_{\g}\subsetneq\M$ et une 
représenta\-tion irré\-ductible cuspidale $\vr$ de $\M_{\g}$ dont 
l'induite à $\M$, notée $\pi'$, admet une sous-représentation isomorphe à $\pi$. 
Au moyen du lemme \ref{PIganiol}, on a le diagramme commutatif~:
\begin{equation*}
\label{IlSuffitBarthes0}
\xymatrix{
\IA_{\a}(\KM_{\a}(\pi))\ar[d]\ar[r]&\KM(\ip_\a(\pi))\ar[d]\\
\IA_{\a}(\KM_{\a}(\pi'))\ar[r]&\KM(\ip_\a(\pi'))}
\end{equation*}
dans lequel toutes les flèches sont injectives. 

Supposons d'abord que $\M$ est de dimension minimale 
pour la propriété que $\mult\E\subseteq\M$.
Dans ce cas, pour tout $i\in\{1,\dots,r\}$, 
$\B_i$ est une algèbre à division sur $\E$ et 
toute représentation ir\-ré\-duc\-tible contenant $\t_{{\rm max},i}$ est 
cuspidale.
Puisque $\pi$ a été supposée non cuspidale, il y a un $k$ tel que 
$\pi_k$ ne contient pas $\t_{{\rm max},k}$, 
\ie que le membre de gauche de \eqref{skamp} est nul.
Il suffit donc de prouver que le membre de droite l'est également. 
Si tel n'était pas le cas, 
la proposition \ref{Cavalcanti} impliquerait que chaque facteur de $\vr$ est 
d'endo-classe $\Tmax$, et la proposition \ref{EhBienVoila} 
contredirait l'hypothèse sur $\pi_k$. 

Revenons au diagramme ci-dessus. 
Par hypothèse de récurrence, $\M_\g$ étant de dimension 
moin\-dre que celle de $\M$,
la représentation $\KM_\a(\pi')=\KM_\a(\ip_\g^\a(\vr))$
(où $\ip_\g^\a$ désigne l'induction para\-bo\-li\-que de $\M_\g$ à $\M$)
est isormorphe à $\ip_\g^\a(\KM_\g(\vr))$.
On peut donc remplacer la ligne du bas par~:
\begin{equation*}
\label{IlSuffitBarthes}
\IA_{\g}(\KM_{\g}(\vr))\to\KM(\ip_{\g}(\vr)).
\end{equation*}
Comme $\vr$ est cuspidale, ce morphisme est bijectif d'après le lemme 
\ref{Berg}. 
Écrivons~:
\begin{equation*}
\sy{\pi'}=\sy{\pi}+\sy{\pi''}
\end{equation*} 
dans le groupe de Grothendieck $\Gg(\M,\R)$, où $\pi''$ est le quotient de 
$\pi'$ par $\pi$.
Tous les foncteurs qui interviennent étant exacts, on a~:
\begin{equation*}
0=\sy{\KM(\ip_{\g}(\vr))}-\sy{\IA_{\g}(\KM_{\g}(\vr))}=
\left(\sy{\KM(\ip_{\a}(\pi))}-\sy{\IA_{\a}(\KM_{\a}(\pi))}\right)+
\left(\sy{\KM(\ip_{\a}(\pi''))}-\sy{\IA_{\a}(\KM_{\a}(\pi''))}\right)
\end{equation*}
dans le groupe de Grothendieck $\Gg(\G,\R)$. 
D'après \eqref{AuFait!KM} appliqué à $\pi$ puis à $\pi''$, 
chaque terme de la somme ci-dessus est une 
représentation, \ie une somme à coefficients positifs ou nuls de 
représentations irréductibles. 
La somme totale étant nulle, chacun de ces termes l'est aussi.
On a donc
$\sy{\KM(\ip_{\a}(\pi))}-\sy{\IA_{\a}(\KM_{\a}(\pi))}=0$,
ce dont on déduit que \eqref{skamp} est un iso\-mor\-phisme. 
\end{proof}

\begin{rema}
\label{Soliloque5}
On a en fait prouvé que, 
pour toute représentation $\pi$ de longueur finie de $\M$, 
l'homomorphisme injectif \eqref{AuFait!KM} est un isomorphisme 
de représentations de $\GB$.
\end{rema}

\begin{coro}
Pour $i\in\{1,\dots,r\}$, soit $\rho_i$ une représentation irréductible 
cuspidale de $\G_{m_i}$ contenant $\t_{{\rm max},i}$ et soit 
$(\J_{{\rm max},i},\k_{{\rm max},i}\otimes\s_i)$ un type simple maximal
contenu dans $\rho_i$. 
Alors~:
\begin{equation}
\label{KoseiInoue}
\KM(\rho_1\times\dots\times\rho_{r}) \simeq
\bigoplus\limits_{i_1}\dots\bigoplus\limits_{i_r}
\s_1^{\Fr^{i_1}}\times\dots\times\s_r^{\Fr^{i_r}},
\end{equation}
où $\Fr$ est un générateur de $\Gal(\kk_{\D'}/\kk_\E)$ et 
où chaque $i_j$ décrit $\{0,\dots,b(\rho_j)-1\}$. 
\end{coro}

\subsection{Compatibilité à la restriction parabolique}
\label{Holveg}

Reprenons les notations du paragraphe \ref{ComPIP} et posons 
$\NB=(\U_1(\La)\cap\mult\B)\J^1_{{\rm max}}/\J^1_{{\rm max}}$, 
qui est le radical unipotent de $\PB$.
On a une décomposition de Levi $\PB=\MB\NB$.
Toute re\-pré\-sentation de $\J=\J(\b,\La)$ triviale sur $\J^1=\J^1(\b,\La)$ peut 
être vue comme une re\-pré\-sen\-tation de $\MB$.

\begin{prop}
\label{PAtaPI}
Soit $\pi$ une représentation de $\G$.
\begin{enumerate}
\item
On a un isomorphisme ca\-no\-ni\-que de représentations de $\MB$~:
\begin{equation}
\label{MmeDeChevreuse}
\KM(\pi)^{\NB}\simeq\Hom_{\J^1}(\k,\pi).
\end{equation}
\item
Pour toute représentation irréductible $\xi$ de $\J$ triviale sur $\J^1$, 
on a un homomorphis\-me injectif de $\R$-algèbres~:
\begin{equation}
\label{MmeDeLavalliere}
\End_{\GB}(\IA_\a(\xi))\simeq
\Hh(\J_{{\rm max}},\kmax|_{(\U(\La)\cap\mult\B)\J^1_{{\rm max}}}\otimes\xi)
\hookrightarrow\Hh(\G,\k\otimes\xi)
\end{equation}
et un isomorphisme~:
\begin{equation*}
\Hom_{\MB}(\xi,\KM(\pi)^\NB)\simeq\Hom_{\J}(\k\otimes\xi,\pi)
\end{equation*}
de $\End_{\GB}(\IA_\a(\xi))$-modules.  
\end{enumerate}
\end{prop}

\begin{proof}
On prouve l'assertion 1 comme dans \cite[\S5]{SZ} grâce aux propriétés de 
transfert entre $\b$-extensions, et on obtient l'homomorphis\-me injectif 
\eqref{MmeDeLavalliere} en reprenant les arguments de \cite[\S4.2]{VS3}.
Ensuite, l'isomorphisme \eqref{MmeDeChevreuse} induit un 
isomorphisme de $\R$-espaces vectoriels~: 
\begin{equation}
\label{netetlisse}
\Hom_{\J}(\k\otimes\xi,\pi)
\simeq\Hom_{\MB}(\xi,\KM(\pi)^{\NB}).
\end{equation}
L'homomorphis\-me injectif \eqref{MmeDeLavalliere} induit une structure de 
$\End_{\GB}(\IA_\a(\xi))$-module sur le membre de gauche, et par réciprocité
de Frobenius~:
\begin{equation*}
\Hom_{\MB}(\xi,\KM(\pi)^{\NB})\simeq\Hom_{\GB}(\IA_\a(\xi),\KM(\pi))
\end{equation*}
on a une structure de $\End_{\GB}(\IA_\a(\xi))$-module sur le membre de 
droite, faisant de \eqref{netetlisse} un iso\-mor\-phisme de 
$\End_{\GB}(\IA_\a(\xi))$-modules. 
\end{proof}

Soit $\vk$ la représentation de $\BJ=\H^1(\b,\La)(\J\cap\P)$ 
définie au paragraphe \ref{EndoEq} dont l'induite à $\J$ est isomorphe à $\k$.  
Pour toute représentation lisse $\pi$ de $\G$, on a des isomorphismes~:
\begin{equation*}
\KM(\pi)^{\NB}
\simeq\Hom_{\J^1}(\k,\pi)
\simeq\Hom_{\BJ^1}(\vk,\pi)
\end{equation*}
de représentations de $\MB$, le premier provenant de 
\eqref{MmeDeChevreuse}. 
D'autre part, on a l'égalité~:
\begin{equation*}
\KM_\a(\rp_\a(\pi))=
\Hom_{\J^1_{{\rm max},\a}}(\k_{{\rm max},\a},\rp_\a(\pi))
\end{equation*}
par définition de $\KM_\a$.
Comme la restriction de $\bk$ à $\BJ\cap\M=\J_{{\rm max},\a}$ est égale à 
$\k_{{\rm max},\a}$, l'homo\-morphisme surjectif canonique de $\pi$ vers son 
module de Jacquet $\rp_\a(\pi)$ induit un mor\-phis\-me~:
\begin{equation*}
\KM(\pi)^{\NB}\to\KM_\a(\rp_\a(\pi))
\end{equation*}
de représentations de $\MB$. 

\begin{prop}
\label{ZS}
Soit $\pi$ une représentation de $\G$.
L'application naturelle~:
\begin{equation*}
\KM(\pi)^{\NB}\to\KM_\a(\rp_\a(\pi))
\end{equation*}
est un isomorphisme de représentations de $\MB$. 
\end{prop}

\begin{proof}
Soit $\bn$ la restriction de $\bk$ à $\BJ^1$ comme au paragraphe \ref{defeta}, 
et soit $\n_{{\rm max},\a}^{}$ la restriction de $\k_{{\rm max},\a}^{}$ au sous-groupe 
$\J^1_{{\rm max},\a}$. 
D'après la proposition \ref{jnpc}, la paire $(\BJ^1,\ve)$ est une paire
cou\-vran\-te de $(\J^1_{{\rm max},\a},\n_{{\rm max},\a}^{})$.
On déduit le résultat voulu de \eqref{Gobineau}. 
\end{proof}

\begin{rema}
Soit $[\AA_{{\rm max}},n_{{\rm max}},0,\b]$ une strate simple de $\A$,
où $\AA_{{\rm max}}$ est l'unique ordre héréditaire standard de $\A$ 
tel que $\AA_{{\rm max}}\cap\B$ soit un ordre héréditaire maximal de $\B$.
Pour tout entier $n\>1$, 
plongeons $\E$ diagona\-lement dans $\A_{mn}\simeq\Mat_n(\A)$.
Il y a un unique ordre héréditaire standard $\AA_{{\rm max}}^{n}$ de $\A_{mn}$ 
dont l'intersection avec le centralisateur de $\E$ soit un ordre maximal.

Comme au début de cette section, on fixe un caractère simple 
maximal $\t_{{\rm max}}$ et une $\b$-extension $\kmax$.
Notons $\t_{{\rm max}}^{n}\in\Cc(\AA_{{\rm max}}^{n},0,\b)$ le transfert de ce 
caractère simple et $\kmax^{n}$ la $\b$-extension com\-patible à 
$\k_{{\rm max}}$ au sens où, si l'on choisit $\a=(m,\dots,m)$, on a
(avec les notations du paragraphe \ref{ComPIP})
$\k_{{\rm max},\a}^{n}=\k_{{\rm max}}^{}\otimes\dots\otimes\k_{{\rm max}}^{}$.
Il correspond à cette $\b$-extension un foncteur~:
\begin{equation*}
\KM_{n}:\Rr(\G_{mn})\to\Rr(\GL_{m'n}(\kk_{\D'})). 
\end{equation*}
Soient $n_1,\dots,n_r\>1$ des entiers tels que $n_1+\dots+n_r=n$ et, 
pour chaque $i\in\{1,\dots,r\}$, soit $\pi_i$ une repré\-sen\-tation 
irréductible de $\G_{mn_i}$. 
Alors on a un isomorphisme canonique~:
\begin{equation*}
\KM_n(\pi_1\times\dots\times\pi_r)\simeq
\KM_{n_1}(\pi_1)\times\dots\times\KM_{n_r}(\pi_r).
\end{equation*}
Si l'on pose $\a=(n_1,\dots,n_r)$, on a un foncteur~: 
\begin{equation*}
\KM_{\a}:\Rr(\M_{(mn_1,\dots,mn_r)})\to
\Rr(\GL_{m'n_1}(\kk_{\D'})\times\dots\times\GL_{m'n_r}(\kk_{\D'}))
\end{equation*}
correspondant à la représentation 
$\kmax^{n_1}\otimes\dots\otimes\kmax^{n_r}$ du groupe 
$\J(\b,\AA_{{\rm max}}^{n_1})\times\dots\times\J(\b,\AA_{{\rm max}}^{n_r})$. 
Notons $\NB$ le sous-groupe unipotent standard 
de $\GL_{m'n}(\kk_{\D'})$ associé à $(m'n_1,\dots,m'n_r)$.
Si $\pi$ est une représentation de $\G_{mn}$, 
on a un isomorphisme canonique~:
\begin{equation*}
\KM_{n}(\pi)^{\NB}\simeq
\KM_\a(\rp_{(mn_1,\dots,mn_r)} (\pi))
\end{equation*}
de représentations de 
$\GL_{m'n_1}(\kk_{\D'})\times\dots\times\GL_{m'n_r}(\kk_{\D'})$.
\end{rema}

\appendix

\section{Une majoration d'entrelacement}
\label{App}

\begin{center}
\textbf{(par Vincent Sécherre et Shaun Stevens)}
\end{center}

\medskip

Soit $\A$ une $\F$-algèbre centrale simple et soit $\V$ un $\A$-module 
simple. 
On note $\D$ la $\F$-algèbre opposée à $\End_\A(\V)$ et on identifie $\A$ 
et $\End_\D(\V)$. 
On pose $\G=\mult\A$.

Si $\psi$ est un caractère d'un sous-groupe $\H$ de $\G$, 
on notera $\I_\G(\psi)$ son ensemble d'entrelacement dans $\G$,
\ie l'ensemble des $g\in\G$ tels que $\psi(gxg^{-1})=\psi(x)$ 
pour tout $x\in\H\cap g^{-1}\H g$.

\subsection{}

Dans ce paragraphe, nous supposons être dans la situation de \cite[\S2.2]{Br4}. 
Soit $[\La,m,m-1,b]$ une strate fondamentale dans $\A$.  
Supposons qu'on a une décomposition de $\D$-espaces vectoriels 
$\V=\V^1\oplus\V^2$ préservée par $b$.
On pose $\A^{ij}=\Hom_\D(\V^j,\V^i)$ pour 
$i,j\in\{1,2\}$ et~:
\begin{equation*}
\A\ =\ \begin{pmatrix} \A^{11} & \A^{12} \\ \A^{21} & \A^{22} \end{pmatrix}.
\end{equation*}
On note aussi $\A^i=\A^{ii}$, $\mathscr{M}=\A^1\oplus\A^2$ et 
$\M=\mathscr{M}^\times$.
Pour $k\in\ZZ$, on pose~:
\begin{equation*}
\aa_k(\La) = \{a\in\A\ |\ a(\La_l)\subseteq\La_{l+k} \text{ pour tout } l\in\ZZ\}.
\end{equation*}

Écrivons $b=b_1+b_2$ avec $b_i\in\A^{i}$, $i=1,2$ et supposons que les strates 
$[\La^i,m,m-1,b_i]$ ont des polynômes caractéristiques premiers entre eux. 
Pour $0\< t< m$, posons~:
\begin{equation*}
\hh_{m,t}=\begin{pmatrix} \aa_m & \aa_{t+1} \\ \aa_m & \aa_m \end{pmatrix}, 
\qquad 
\jj_{m,t}=\begin{pmatrix} \aa_1 & \aa_1 \\ \aa_{m-t} & \aa_1 \end{pmatrix}
\end{equation*}
et définissons des sous-groupes ouverts et compacts de $\G$ par 
$\H_{m,t}=1+\hh_{m,t}$ et $\J_{m,t}=1+\jj_{m,t}$. 
Comme le commutateur de $\H_{m,t}$ est inclus dans $\U_{m+1}(\La)$, 
la formule~: 
\begin{equation*}
\psi_b:1+x\mapsto \psi_\F\circ\tr_{\A/\F}(bx)
\end{equation*}
(voir \eqref{DEFPSIB}) définit un caractère de $\H_{m,t}$. 
D'après \cite[2.3.3-2.3.8]{Br4}, on a le résultat suivant\footnote{
Contrairement à \cite{Br4},
nous ne supposons pas que $\La$ est stricte, 
ce qui ne change rien à l'argument, ni que $b_1$ est $\La^1$-inversible, 
ce que nous pourrions supposer quitte à échanger les rôles de $b_1$ et de 
$b_2$. 
}.

\begin{lemm}\label{lem:maps}
Soit~$b_i'\in b_i+\aa_{1-m}$ pour~$i=1,2$, soit~$c\in\aa_{1-m}^{21}$ et soit~$r\in\ZZ$. 
\begin{enumerate}
\item 
Pour~$r\>0$, l'application~$x\mapsto b_1'x-xb_2'+xcx$ induit un 
isomorphisme $\aa_{r}^{12}\to\aa_{r-m}^{12}$.
\item 
L'application~$x\mapsto b_1'x-xb_2'$ induit un 
isomorphisme $\aa_{r}^{12}\to\aa_{r-m}^{12}$.
\item 
L'application~$x\mapsto xb_1'-b_2'x$ induit un 
isomorphisme $\aa_{r}^{21}\to\aa_{r-m}^{21}$.
\end{enumerate}
\end{lemm}

Étant donnée une partie $\X$ de $\A$, on note $\X^*$ l'ensemble des $a\in\A$ 
tels que 
$\psi_\F\circ\tr_{\A/\F}(ax)=1$ 
pour tout $x\in\X$. 

\begin{coro}\label{cor:toM}
Soit~$x\in b+\hh_{m,t}^*$.
Alors il y a un~$j\in\J_{t,l}$ tel que~$jxj^{-1}\in b+(\hh_{t,l}^*\cap\mathscr{M})$.
\end{coro}

\begin{proof}
Nous suivons la preuve de~\cite[3.7 Lemma 3]{BK2}. 
Écrivons~:
\begin{equation*}
x=\begin{pmatrix} b_1' & a \\ c & b_2' \end{pmatrix}
\end{equation*}
où~$a\in\aa_{1-t}^{12}$ et~$c\in\aa_{-l}^{21}\subseteq\aa_{1-t}^{21}$.  
D'après le lemme~\ref{lem:maps}(1), il y a~$v\in\aa_1^{12}$ tel que 
$b'_1v-vb'_2+vcv=a$, et ainsi ~$(1+v)x(1+v)^{-1}$ est triangulaire 
supérieure. 
Comme~$1+v\in\J_{t,l}$, on se ramène ainsi au cas où~$a=0$.  

Supposons donc que $a$ est nul, et soit $w\in\aa_{t-l}^{21}$ tel que $wb_1'-b_2'w=c$ 
(voir le lemme~\ref{lem:maps}(3)). 
Alors $(1+w)^{-1}x(1+w)\in\mathscr{M}$
et~$1+w\in\J_{t,l}$, comme demandé. 
\end{proof}

\begin{prop}\label{prop:mincase}
On a $\I_{\G}\(\psi_b|\H_{m,t}\) \subseteq \J_{m,t}\M\J_{m,t}$.
\end{prop}

\begin{proof}
Soit~$g\in\I_{\G}\(\psi_b|\H_{m,t}\)$.
Il existe $x,y\in\hh_{m,t}^*$ tels que~:
\begin{equation*}
g^{-1}(b+x)g = b+y.
\end{equation*}
D'après le corollaire \ref{cor:toM}, on peut remplacer $g$ par~$j_1gj_2$ 
avec~$j_1,j_2\in\J_{m,t}$ et supposer que $x,y$ sont dans 
$\hh_{m,t}^*\cap\mathscr{M}$. 
Le reste de la preuve se termine comme dans \cite{Br4,BK2} 
en écrivant $g$ par blocs et en remarquant que les points 2 et 3 
du lemme \ref{lem:maps} impliquent que les blocs non diagonaux de 
$g$ sont nuls, de sorte que $g\in\M$. 
\end{proof}

\subsection{}
\label{P:rss}

Dans ce paragraphe, nous supposons être dans la situation de \cite[\S4]{SeSt1}. 
Nous avons donc une strate simple $[\La,n,m,\b]$ de $\A$, 
une décomposition $\V=\V^1\oplus\V^2$ en $\E\otimes\D$-modules à droite 
et un élé\-ment $c\in\aa_{-m}(\La)$ stabilisant cette décomposition. 
Posons~:
\begin{equation*}
\N=1+\A^{12},
\quad
\N^-=1+\A^{21},
\quad
\M=(\A^{1})^\times\times(\A^{2})^\times,
\quad
\P=\M\N,
\quad
\P^-=\M\N^-.
\end{equation*}
Pour tout $\Oo_\F$-réseau $\xx$ de $\A$ et tout $i,j\in\{1,2\}$, nous
écrirons $\xx^{ij}=\xx\cap\A^{ij}$.

Pour tout $k\>1$, posons $\H^k=\H^k(\b,\La)$ et $\J^k=\J^k(\b,\La)$, 
et notons $\hh_k$ et $\jj_k$ les $\Oo_\F$-réseaux de $\A$ définis par 
$\H^k=1+\hh_k$ et $\J^k=1+\jj_k$. 
Posons $q=-k_0(\b,\La)$, $r=\lfloor q/2\rfloor+1$ et $s=\lceil q/2\rceil$. 

Pour tout $x\in\A$, posons $a_\b(x)=\b x-x\b$.
Pour $k\in\ZZ$, posons~:
\begin{eqnarray*}
\bb_k&=&\aa_k(\La)\cap\B,\\
\nn_k&=&\nn_k(\b,\La),\\
\mm_k&=&\aa_k(\La)\cap\nn_{k-q}+\jj_{s}.
\end{eqnarray*}
et posons $\Om^k=1+\mm_k$ pour $k\>1$.
C'est un pro-$p$-sous-groupe ouvert de $\U_1(\La)$ 
norma\-lisant $\H^{l}$ pour tout $l\>r$
(et aussi $\H^{q-k+1}$ si $k\<q$)
et qui est normalisé par $\U_1(\La)\cap\mult\B$.

Fixons un caractère simple $\t\in\Cc(\La,m-1,\b)$. 
Posons~:
\begin{eqnarray*}
\Om&=&(\U_1(\La)\cap\mult\B)\Om^{q-m+1},\\
\K&=&(\H^m\cap\P^-)\cdot(\Om\cap\N),
\end{eqnarray*}
et notons $\xi$ le caractère de $\K$ 
trivial sur $\K\cap\N$ et $\K\cap\N^-$ et coïncidant avec $\t\psi_c$ 
sur $\K\cap\M$. 

\subsection{}
\label{ETAPE1}

Dans ce paragraphe, nous supposons que $m\>r$. 
Pour tout $r\<t\<m$, posons~:
\begin{equation*}
\K_t = (\H^m\cap\P^-)\cdot(\H^t\cap\N)
\end{equation*}
et notons $\xi_t$ la restriction de $\xi$ à $\K_t$. 
Posons $\Xi_t=(\Xi_t\cap\N^-)\cdot(\Xi_t\cap\M)\cdot(\Om\cap\N)$ avec~:
\begin{eqnarray*}
\Xi_t\cap\M&=&(\U_1(\La)\cap\mult\B)\Om^{q-m}\cap\M,\\
\Xi_t\cap\N^- &=&(\U_{m-t+1} (\La)\cap\mult\B)\Om^{q-t+1}\cap\N^-.
\end{eqnarray*}

\begin{lemm}
L'ensemble $\Xi_t$ est un sous-groupe de $\U_1(\La)$. 
\end{lemm}

\begin{proof}
D'abord, le groupe $\Xi_t\cap\M$ normalise $\Xi_t\cap\N$ et 
$\Xi_t\cap\N^-$ car $\U_1(\La)\cap\mult\B$ et $\Om^k$ nor\-ma\-li\-sent 
$\Om^l$ pour tous $k,l\>1$.

On vérifie ensuite que $(\Xi_t\cap\N)\cdot(\Xi_t\cap\N^-)$ est inclus dans 
$\Xi_t$.  
Étant donné que~:
\begin{equation*}
\begin{pmatrix} 1&u\\0&1 \end{pmatrix}
\begin{pmatrix} 1&0\\v&1 \end{pmatrix}
=
\begin{pmatrix} 1&0\\v(1+uv)^{-1}&1 \end{pmatrix}
\begin{pmatrix} 1+uv&0\\0&(1+vu)^{-1}\end{pmatrix}
\begin{pmatrix} 1&(1+uv)^{-1}u\\0&1 \end{pmatrix}
\end{equation*}
pour tout $u\in\aa_1^{12}$ et tout $v\in\aa_1^{21}$,
il suffit de prouver que~:
\begin{eqnarray*}
(\bb^{12}_1+\mm^{12}_{q-m+1})(\bb_{m-t+1}^{21}+\mm^{21}_{q-t+1}) 
&\subseteq&\bb^{1}_1+\mm^{1}_{q-m},\\
(\bb_{m-t+1}^{21}+\mm^{21}_{q-t+1})(\bb^{12}_1+\mm^{12}_{q-m+1})
&\subseteq&\bb_1^2+\mm^{2}_{q-m},
\end{eqnarray*}
ce qui est une conséquence du fait que $\mm_k\mm_l\subseteq\mm_{k+l}$ et 
$\bb_k\mm_l\subseteq\mm_{k+l}$ pour tous $k,l\>0$. 
\end{proof} 

Pour $t=m$, on a $\K_m=\H^m$ et 
(voir la preuve de \cite[Théorème 4.3]{SeSt1}) 
l'entrelacement dans $\G$ de $\xi_m=\t\psi_c$ est inclus dans 
$\Xi_m(\M\cap\mult\B)\Xi_m$. 

\begin{prop}
On a $\I_\G(\xi_t)\subseteq\Xi_t(\M\cap\mult\B)\Xi_t$. 
\end{prop}

On va prouver cette proposition par récurrence sur $t$.
Posons~:
\begin{equation*}
^-\K_t = (\H^m\cap\N^-)\cdot(\H^{m+1}\cap\M)\cdot(\H^t\cap\N)
\end{equation*}
(qui est un sous-groupe de $\K_t$ car $\H^k$ normalise $\H^l$ et 
$\hh^k\hh^l\subseteq\hh^{k+l}$ pour tous $k,l\>1$)
et notons $^-\xi_t$ la res\-triction de $\xi_t$ à $^-\K_t$. 

\begin{lemm}
\label{AL2}
Le groupe $\Xi_t\cap\P$ normalise $^-\xi_t$. 
\end{lemm}

\begin{proof}
D'abord le groupe $(\U_1(\La)\cap\mult\B)\Om^{q-m}\cap\M$ normalise $^-\K_t$ 
car il normalise $\H^k$ pour tout $k\>r$.  
Ainsi il normalise $^-\xi_t$ si et seulement s'il normalise 
la restriction à $\H^{m+1}\cap\M$ de ce caractère, ce qui découle de 
\cite[Théorème 2.23]{SeSt1}.

Passons à $\Xi_t\cap\N=\Om\cap\N$.
Comme ce groupe normalise $^-\K_t$, on a~:
\begin{equation*}
\begin{pmatrix} x&0\\0&y \end{pmatrix}^{-1}
\begin{pmatrix} 1&u\\0&1 \end{pmatrix}
\begin{pmatrix} x&0\\0&y \end{pmatrix}
\begin{pmatrix} 1&-u\\0&1 \end{pmatrix}
=
\begin{pmatrix} 1&x^{-1}uy-u\\0&1 \end{pmatrix}
\in\H^t\cap\N\subseteq\Ker(^-\xi_t)
\end{equation*}
pour tout $u\in\bb_1+\mm^{12}_{q-m+1}$, donc il normalise aussi la restriction de 
$^-\xi_t$ à $^-\K_t\cap\P$.
Écrivons~:
\begin{equation*}
\begin{pmatrix} 1&u\\0&1 \end{pmatrix}
\begin{pmatrix} 1&0\\h&1 \end{pmatrix}
\begin{pmatrix} 1&-u\\0&1 \end{pmatrix}
=
\begin{pmatrix} 1&0\\h(1+uh)^{-1}&1 \end{pmatrix}
\begin{pmatrix} 1+uh&0\\0&(1+hu)^{-1}\end{pmatrix}
\begin{pmatrix} 1&-(1+uh)^{-1}uhu\\0&1 \end{pmatrix}
\end{equation*}
pour tout $u\in\bb^{12}_1+\mm^{12}_{q-m+1}$ et tout $h\in\hh^{21}_m$. 
Comme $\Om$ normalise $\H^m$, chacun des trois facteurs du 
membre de droite est dans $\H^m$.
En outre le facteur du milieu est dans 
$\H^{m+1}\cap\M$, ce qui suit du fait qu'il est à la fois dans $\H^m\cap\M$ 
et dans $1+(\bb_1+\mm_{q-m+1})\hh_m\subseteq\U_{m+1}(\La)$.
Ainsi $\Om\cap\N$ nor\-malise $^-\K_t$, et il normalise $^-\xi_t$ si 
et seulement si $^-\xi_t$ est trivial sur les éléments de la forme~:
\begin{equation}
\label{salediag}
\begin{pmatrix} 1+uh&0\\0&(1+hu)^{-1}\end{pmatrix}
\in \H^{m+1}\cap\M
\end{equation}
avec $u\in\bb^{12}_1+\mm^{12}_{q-m+1}$ et $h\in\hh^{21}_m$.  
Si l'on applique le même raisonnement à $\t$ 
plutôt qu'à $^-\xi_t$, 
on en déduit que $\t$ est trivial sur les éléments de la forme 
\eqref{salediag} puisqu'il est normalisé par $\Om\cap\N$.
Comme $\t$ et $^-\xi_t$ coïncident sur $\H^{m+1}\cap\M$, on en déduit que 
$^-\xi_t$ est normalisé par $\Om\cap\N$.
\end{proof}

Supposons maintenant que $\I_\G(\xi_t)\subseteq\Xi_t(\M\cap\mult\B)\Xi_t$ 
pour un $r<t\<m$, et soit $g\in\I_\G(\xi_{t-1})$.
On a \textit{a fortiori} $g\in\I_\G(\xi_t)$, de sorte que~:
\begin{equation*}
g = (1+x)b(1+y)^{-1}
\end{equation*}
avec $b\in\M\cap\mult\B$ et $1+x,1+y\in\Xi_t$.
Comme $\Xi_t\cap\P=\Xi_{t-1}\cap\P$ normalise $^-\xi_{t-1}$ 
d'après le lemme \ref{AL2}, on peut supposer que $1+x,1+y\in\Xi_t\cap\N^-$.

\begin{lemm}
\label{AL3}
Pour tout $1+z\in\Xi_t\cap\N^-$, on a $^-\xi_{t-1}^{1+z}={}^-\xi_{t-1}^{}\psi_{a_\b(z)}$.
\end{lemm}

\begin{proof}
D'après la preuve du lemme \cite[Lemme 4.7]{SeSt1}, 
on a $\xi_{t-1}^{1+z}=\xi_{t-1}^{}\psi_{a_\b(z)}$ 
pour tout $1+z\in\Xi_t\cap\N^-$. 
Il suffit donc de prouver que $\Xi_t\cap\N^-$ normalise $^-\K_{t-1}$.
Remarquons que~:
\begin{equation*}
^-\K_{t-1}=\H^{m+1}(\H^m\cap\N^-)(\H^{t-1}\cap\N)
\end{equation*} 
et que $\Xi_t\cap\N^-$ normalise $\H^{m+1}$ et $\H^m\cap\N^-$
(voir \cite[Proposition 2.30]{SeSt1}).
Il reste donc à étudier l'action de $\Xi_t\cap\N^-$ sur $\H^{t-1}\cap\N$. 
On a~:
\begin{equation*}
\begin{pmatrix} 1&0\\u&1 \end{pmatrix}
\begin{pmatrix} 1&h\\0&1 \end{pmatrix}
\begin{pmatrix} 1&0\\-u&1 \end{pmatrix}
=
\begin{pmatrix} 1&0\\-uhu(1-hu)^{-1}&1 \end{pmatrix}
\begin{pmatrix} 1-hu&0\\0&(1-uh)^{-1}\end{pmatrix}
\begin{pmatrix} 1&(1-hu)^{-1}h\\0&1 \end{pmatrix}
\end{equation*}
pour tout $u\in\bb^{21}_{m-t+1}+\mm^{21}_{q-t+1}$ et tout 
$h\in\hh^{12}_{t-1}$. 

Comme $\Xi_t\cap\N^-$ normalise $\H^{t-1}$, le facteur de droite du membre de 
droite est dans $\H^{t-1}\cap\N$. 
Ensuite le facteur du milieu est dans 
$\H^{t-1}\cap\M\cap\U_{m+1}(\La)\subseteq\H^{m+1}\cap\M$. 

D'après \cite[Lemme 2.30]{SeSt1}, le groupe $\Xi_t\cap\N^-$ normalise 
$\H^{t-1}$, donc le facteur de gauche appar\-tient à $\H^{t-1}\cap\N^-$.
Comme $1-hu\in\H^{m+1}\subseteq\H^{t-1}$, on a $uhu\in\hh^{21}_{t-1}$. 
Mais on a aussi $u\in\aa_{m-t+1}$ et $h\in\aa_{t-1}$, de sorte que 
$uhu\in\hh^{21}_{t-1}\cap\aa_{m}\subseteq\hh^{21}_{m}$. 
Ainsi le facteur de gauche est dans $\H^{m}$.
\end{proof}

Écrivons maintenant que $g\in\I_\G(^-\xi_{t-1})$. 
Ainsi $b$ entrelace $^-\xi_{t-1}^{1+x}$ avec $^-\xi_{t-1}^{1+y}$, donc~:
\begin{equation*}
b^{-1}a_\b(x) b \equiv a_\b(y) \mod{b^{-1}(^-\kk_{t-1}^*)b+(^-\kk_{t-1}^*)}
\end{equation*}
où $^-\kk_{t-1}$ est le $\Oo_\F$-réseau défini par $^-\K_{t-1}=1+{}^-\kk_{t-1}$.

\begin{lemm}
\label{AL4}
On a $(b^{-1}(^-\kk_{t-1}^*)b+(^-\kk_{t-1}^*))\cap\A^{21}=
(b^{-1}\hh_{t-1}^*b+\hh_{t-1}^*)\cap\A^{21}$.
\end{lemm}

\begin{proof}
Soit $\xx$ un $\Oo_\F$-réseau de $\A$ se décomposant en blocs $\xx^{ij}$, 
$i,j\in\{1,2\}$.
Alors~:
\begin{eqnarray*}
\xx^*\cap\A^{21}&=&\{a\in\A^{21}\ |\
\psi_{\A}(ax)=1\text{ pour tout $x\in\xx$}\}\\
&=&\{a\in\A^{21}\ |\ \psi_{\A^2}(ax)=1\text{ pour tout $x\in\xx^{12}$}\}.
\end{eqnarray*}
Ensuite, comme $b\in\M$, on a~:
\begin{equation*}
(b^{-1}\xx^*b+\xx^*)\cap\A^{21}=b_2^{-1}(\xx^*\cap\A^{21})b_1^{}+(\xx^*\cap\A^{21})
\end{equation*}
avec $b=b_1+b_2$ et $b_1\in(\A^1)^\times$ et $b_2\in(\A^2)^\times$.
Pour $\xx={}^-\kk_{t-1}$, on a $\xx^{12}=\hh_{t-1}^{12}$
donc on obtient le résultat attendu.
\end{proof}

D'après \cite[Proposition~2.27]{SeSt1}, il y a $x',y'\in\mm_{q-t+2}^{21}$ tels que 
$a_\b(b^{-1}xb-y) = a_\b(b^{-1}x'b-y')$.
Ainsi~:
\begin{equation*}
(b^{-1}xb-y)-(b^{-1}x'b-y')\in(b^{-1}\mm_{q-t+1}b+\mm_{q-t+1})\cap\B\cap\A^{21}.
\end{equation*}
et le membre de droite est égal à
$(b^{-1}\bb_{q-t+1}b+\bb_{q-t+1})\cap\A^{21}$
d'après \cite[Proposition~2.27]{SeSt1}.
Il y a donc $x'',y''\in\bb^{21}_{m-t+1}+\mm_{q-t+2}^{21}$ tels qu'on ait 
$b^{-1}xb-y=b^{-1}x''b-y''$, 
ce qui donne l'éga\-li\-té $g=(1+x'')b(1+y'')^{-1}$. 

Comme $1+x'',1+y''\in(\U_{m-t+1}(\La)\cap\mult\B)\Om^{q-t+2}\cap\N^-$, 
on peut écrire 
$g=(1+u)\g(1+v)^{-1}$ avec $1+u,1+v\in\Om^{q-t+2}\cap\N^-$
et
$\g\in(\U_{m-t+1}(\La)\cap\mult\B)b(\U_{m-t+1}(\La)\cap\mult\B)$.

L'élément $\g$ entrelace $\xi_{t-1}^{1+u}$ et $\xi_{t-1}^{1+v}$, donc aussi 
leurs restrictions au groupe~:
\begin{equation*}
\K_{t-1}\cap\mult\B=1+
\begin{pmatrix} \bb_m & \bb_{t-1} \\ \bb_m & \bb_m \end{pmatrix}.
\end{equation*}
Comme $\xi_{t-1}^{1+z}=\xi_{t-1}\psi_{a_\b(z)}$ pour tout 
$1+z\in\Xi_t\cap\N^-$, comme $\psi_{a_\b(u)}$ et 
$\psi_{a_\b(v)}$ sont triviaux sur $\K_{t-1}\cap\mult\B$, et comme 
$\g$ entrelace la restriction de $\t$ à $\K_{t-1}\cap\mult\B$, 
il entrelace aussi la restriction de $\psi_c$ à $\K_{t-1}\cap\mult\B$.
En appliquant la proposition \ref{prop:mincase}, on trouve que 
$g\in\Xi_{t-1}(\M\cap\mult\B)\Xi_{t-1}$.

\subsection{}
\label{ETAPE2}

On suppose toujours que $m\>r$, mais on pose cette fois-ci~:
\begin{equation*}
\Q_t = (\H^m\cap\P^-)\cdot(\Om^t\cap\N)
\end{equation*}
pour tout $q-m+1\<t\<s$, et on note $\zeta_t$ la restriction de $\xi$ à $\Q_t$.  

Posons $\Ga_t=(\Ga_t\cap\N^-)\cdot(\Xi_{m}\cap\P)$ avec~:
\begin{eqnarray*}
\Ga_t\cap\N^- &=&(\U_{m-t+1} (\La)\cap\mult\B)\H^{q-t+1}\cap\N^-.
\end{eqnarray*}
En appliquant le même calcul que dans le paragraphe précédent, mais en 
échangeant les rôles de $\hh$ et de $\mm$, on prouve que 
$\I_\G(\zeta_t)\subseteq\Ga_t(\M\cap\mult\B)\Ga_t$.
Pour cela, on a besoin de la variante suivante de \cite[Proposition~2.27]{SeSt1}, 
que l'on prouve de la même façon à partir de \cite[Lemma 6.3.2]{BK2}.

\begin{prop}
\label{ExactSequences}
Soit $s$ une corestriction modérée sur $\A$ relative à $\E/\F$.
Pour tout entier $0\<m\<q-1$, la suite~:
\begin{equation*}
\label{setilE}
0\to\bb_{m+1}\to\hh^{m+1}
\fr{a_\b}(\mm_{q-m})^*\fr{s}\bb_{m+1-q}\to0
\end{equation*}
est exacte.
Si on d\'esigne cette suite par 
$0\to\mathfrak{l}_{1}\to\mathfrak{l}_{2}
\to\mathfrak{l}_{3}\to\mathfrak{l}_{4}\to0$, 
alors la suite~:
\begin{equation*}
\label{setilEij}
0
\to h^{-1}\mathfrak{l}^{ij}_{1}h+\mathfrak{l}^{ij}_{1}
\to h^{-1}\mathfrak{l}^{ij}_{2}h+\mathfrak{l}^{ij}_{2}
\to h^{-1}\mathfrak{l}^{ij}_{3}h+\mathfrak{l}^{ij}_{3}
\to h^{-1}\mathfrak{l}^{ij}_{4}h+\mathfrak{l}^{ij}_{4}
\to 0
\end{equation*}
est exacte pour tout $h\in\mult\B\cap\M$ et tous $i,j\in\{1,2\}$.
\end{prop}

Il ne nous reste qu'à prouver le résultat suivant. 
Posons $\Xi=\K((\U_1(\La)\cap\mult\B)\Om^{q-m}\cap\M)$.

\begin{prop}
\label{positron}
On a $\I_\G(\xi)\subseteq\Xi(\M\cap\mult\B)\Xi$.
\end{prop}

\begin{proof}
Soit $g\in\I_\G(\xi)$.
En restreignant à~:
\begin{equation*}
^-\Q_{q-m+1}=(\H^m\cap\N^-)\cdot(\H^{m+1}\cap\M)\cdot(\Om^t\cap\N)
\end{equation*}
on peut écrire $g=(1+u)\g(1+v)^{-1}$ avec 
$\g\in(\U_{2m-q}(\La)\cap\mult\B)(\M\cap\mult\B)(\U_{2m-q}(\La)\cap\mult\B)$
et $1+u,1+v\in\H^{m}\cap\N^-$.
L'élément $\g$ entrelace $\xi^{1+u}$ et $\xi^{1+v}$, donc aussi 
leurs restrictions à~:
\begin{equation*}
\K\cap\mult\B=1+
\begin{pmatrix} \bb_m & \bb_{1} \\ \bb_m & \bb_m \end{pmatrix}.
\end{equation*}
On en déduit qu'il 
entrelace aussi la restriction de $\psi_c$ à $\K\cap\mult\B$.
En appliquant la proposition \ref{prop:mincase}, on trouve que 
$g\in\Xi(\M\cap\mult\B)\Xi$.
\end{proof}

\subsection{}

Dans ce paragraphe, on suppose que $m\in\{1,\dots,q-1\}$.

\begin{coro}
\label{cor:xiKint}
On a $\I_\G(\xi)\subseteq\K\M\K$.
\end{coro}

\begin{proof}
Si $m\>r$, c'est une conséquence de la proposition \ref{positron}.

Si $m<r$, remarquons que $\Om^{q-m}=\Om^{q-m+1}=\J^s$ 
et $\Om=\J^1$ et $\K=\H^m(\J^1\cap\N)$. 
Soit $g\in\I_\G(\xi)$.
En restreignant $\xi$ de $\K$ à $\K'=\H^r(\J^1\cap\N)$ et en appliquant le 
corollaire avec $m=r$, on obtient $g\in\K'\M\K'$. 
Comme $\K'\subseteq\K$, on en déduit le résultat voulu. 
\end{proof}

\subsection{}

Soit maintenant $(\BJ,\bl)$ un type semi-simple de $\G$, comme 
au paragraphe \ref{NonEndoEq} dont nous repre\-nons les notations, 
et soit $\bn$ la représentation de $\BJ^1$ définie au paragraphe
\ref{defeta}.

Le résultat suivant est nouveau, même dans le cas complexe et même 
dans le cas de $\GL_n(\F)$.

\begin{lemm} 
\label{NicolasRostov}
On a $\I_\G(\bn)\subseteq \BJ^1\L\BJ^1$.
\end{lemm}

\begin{proof}
La preuve se fait par récurrence sur $l$, le cas $l=1$ correspondant 
à un type semi-simple homogène et étant déjà traité au paragraphe \ref{EndoEq}. 
Nous supposons donc que $l\>2$. 

Pour~$t\in\{0,\dots,n\}$, 
notons~${\bf\Theta}_i^{(t)}$ le ps-caractère déterminé par 
la restriction de $\t_i$ à $\H^{t+1}(\b_i,\La^i)$. 
Pour~$t\in\{1,\dots,n\}$, on définit une relation d'équivalence~$\sim_t$ 
sur~$\{1,\ldots,r\}$ par~: 
\begin{equation*}
i\sim_t j \ \Leftrightarrow \ {\bf\Theta}_i^{(t-1)} \approx {\bf\Theta}_j^{(t-1)}
\end{equation*}
où $\approx$ désigne la relation d'endo-équivalence entre ps-caractères. 

Notons $t$ le plus petit entier $\>1$ tel que 
${\bf\Theta}_1^{(t)},\dots,{\bf\Theta}_r^{(t)}$ 
sont tous endo-équivalents et soit $\I$ une clas\-se d'équivalence pour la 
relation~$\sim_t$.
Posons~:
\begin{equation*}
\Y^1=\bigoplus_{i\in\I}\V^i,
\quad
\Y^2=\bigoplus_{i\not\in\I}\V^i.
\end{equation*}
Remarquons que $\Y^1\ne\V$, par minimalité de~$t$.  
Soit $\M^*$ le stabilisateur dans $\G$ 
de la décompo\-sition $\V=\Y_1\oplus\Y_2$
et soit $\P^*$ le sous-groupe parabolique stabilisateur de $\Y_1$, de 
radical unipo\-tent $\N^*$. 
Nous allons prouver que~:
\begin{equation}
\label{eqn:intc}
\I_\G(\bn) \subseteq \BJ^1\M^*\BJ^1.
\end{equation}
Le résultat se déduira alors de l'hypothèse de récurrence, puisque 
$\bl|_{\BJ\cap\M^*}=\bl^{(1)}\otimes\bl^{(2)}$, où $\bl^{(1)}$ et 
$\bl^{(2)}$ sont des types semi-simples avec strictement moins 
de $l$ endo-classes. 

Prouvons maintenant la relation \eqref{eqn:intc}. 
Soit~$([\La,n,0,\g],\vartheta,t)$ une approximation commune de 
$(\t_1,\dots,\t_r)$. 
Posons~:
\begin{equation*}
\Om=(\U^1(\La)\cap\B_\g)(1+\mm_{q-t+1}(\g,\La)),
\quad
\K=\H^t(\g,\La)(\Om\cap\N^*),
\end{equation*}
où $\B_\g$ est le centralisateur de $\g$ dans $\A$. 
Posons $c=\g-\b$ et $c_i={\rm e}_{i}c{\rm e}_{i}$, où~${\rm e}_{i}$ est la 
projection sur~$\Y^i$ de noyau $\Y^{3-i}$ pour $i\in\{1,2\}$. 
Soit $\xi$ le caractère de $\K$ coïncidant avec 
$\t\psi_c$ sur~$\H^t(\g,\La)$ et qui est trivial sur $\K\cap\N^*$. 
Nous sommes donc dans la situation du paragraphe~\ref{P:rss}.

Par construction, on a 
$\BJ = (\H^t(\g,\La)\cap\N^{*-}) \cdot (\BJ\cap\M^*)\cdot (\Om\cap\N^*)$
et la restriction de $\bl$ à $\K$ est un multiple de $\xi$. 
En particulier, c'est aussi vrai pour $\bn$ (car $\K\subseteq\BJ^1$) et~: 
\begin{equation*}
\I_\G(\bn) \subseteq \I_\G(\xi) \subseteq \K\M^*\K
\end{equation*}
d'après le corollaire~\ref{cor:xiKint}.  
Comme $\K\subseteq\BJ^1$, on a prouvé~\eqref{eqn:intc} et le résultat 
s'ensuit. 
\end{proof} 

\providecommand{\bysame}{\leavevmode ---\ }
\providecommand{\og}{``}
\providecommand{\fg}{''}
\providecommand{\smfandname}{\&}
\providecommand{\smfedsname}{\'eds.}
\providecommand{\smfedname}{\'ed.}
\providecommand{\smfmastersthesisname}{M\'emoire}
\providecommand{\smfphdthesisname}{Th\`ese}

\end{document}